\documentclass{amsart}

\usepackage{amssymb, latexsym, MnSymbol, amscd ,graphicx,psfrag}
\usepackage[mathscr]{euscript}

\newcommand{\Ga}{\Gamma}
\newcommand{\Si}{ {\Sigma} }
\newcommand{\Up}{ {\Upsilon} }
\newcommand{\si}{ {\sigma} }
\newcommand{\ep}{ {\epsilon} }

\newcommand{\bA}{ \mathbb{A} }
\newcommand{\bC}{ {\mathbb{C}} }
\newcommand{\bD}{\mathbb{D}}
\newcommand{\bE}{\mathbb{E}}
\newcommand{\bF}{\mathbb{F}}
\newcommand{\bH}{\mathbb{H}}
\newcommand{\bK}{\mathbb{K}}
\newcommand{\bL}{\mathbb{L}}
\newcommand{\bP}{\mathbb{P}}
\newcommand{\bQ}{\mathbb{Q}}
\newcommand{\bR}{\mathbb{R}}
\newcommand{\bS}{\mathbb{S}}
\newcommand{\bT}{\mathbb{T}}
\newcommand{\bU}{\mathbb{U}}

\newcommand{\bZ}{\mathbb{Z}}

\newcommand{\cA}{\mathcal{A}}
\newcommand{\cB}{\mathcal{B}}
\newcommand{\cE}{\mathcal{E}}

\newcommand{\cI}{\mathcal{I}}
\newcommand{\cL}{\mathcal{L}}
\newcommand{\cM}{\mathcal{M}}
\newcommand{\cO}{\mathcal{O}}
\newcommand{\cP}{\mathcal{P}}

\newcommand{\cS}{\mathcal{S}}
\newcommand{\cT}{\mathcal{T}}
\newcommand{\cX}{\mathcal{X}}
\newcommand{\cV}{\mathcal{V}}

\newcommand{\age}{\mathrm{age}}
\newcommand{\Aut}{\mathrm{Aut}}
\newcommand{\CR}{ {\mathrm{CR}} }
\newcommand{\Hess}{\mathrm{Hess}}
\newcommand{\Hom}{\mathrm{Hom}}
\newcommand{\End}{\mathrm{End}}
\newcommand{\Jac}{ {\mathrm{Jac}} }
\newcommand{\Ker}{\mathrm{Ker}}
\newcommand{\Res}{\mathrm{Res}}
\newcommand{\Spec}{\mathrm{Spec}}
\newcommand{\Pic}{\mathrm{Pic}}
\newcommand{\can}{ {\mathrm{can}} }
\newcommand{\eff}{ {\mathrm{eff}} }
\newcommand{\ext}{ {\mathrm{ext}} }
\newcommand{\ev}{\mathrm{ev}}
\newcommand{\val}{ {\mathrm{val}} }
\newcommand{\vir}{ {\mathrm{vir}} }
\newcommand{\Boxs}{\mathrm{Box}(\si)}
\newcommand{\BoxS}{\mathrm{Box}(\bSi^{\can})}

\newcommand{\Int}{\mathrm{Int}}

\newcommand{\Eff}{{\mathrm{Eff}}}
\newcommand{\Nef}{{\mathrm{Nef}}}
\newcommand{\NE}{{\mathrm{NE}}}
\newcommand{\Area}{\mathrm{Area}}

\newcommand{\Picst}{\Pic^\mathrm{st}}

\newcommand{\be}{\mathbf{e}}
\newcommand{\bk}{\mathbf{k}}
\newcommand{\bp}{\mathbf{p}}
\newcommand{\bu}{\mathbf{u}}

\newcommand{\one}{\mathbf{1}}
\newcommand{\bmu}{\boldsymbol{\mu}}
\newcommand{\btau}{\boldsymbol{\tau}}
\newcommand{\bsi}{{\boldsymbol{\si}}}
\newcommand{\brho}{{\boldsymbol{\rho}}}
\newcommand{\bgamma}{{{\boldsymbol{\gamma}}}}

\newcommand{\bSi}{\mathbf{\Si}}

\newcommand{\fg}{\mathfrak{g}}
\newcommand{\fl}{\mathfrak{l}}
\newcommand{\fm}{\mathfrak{m}}
\newcommand{\fM}{\mathfrak{M}}
\newcommand{\fn}{\mathfrak{n}}
\newcommand{\fp}{\mathfrak{p}}
\newcommand{\fr}{\mathfrak{r}}
\newcommand{\fs}{\mathfrak{s}}

\newcommand{\fC}{\mathfrak{C}}

\newcommand{\rA}{{\mathscr{A}}}
\newcommand{\rB}{{\mathscr{B}}}
\newcommand{\rP}{{\mathscr{P}}}

\newcommand{\w}{\mathsf{w}}
\newcommand{\su}{\mathsf{u}}
\newcommand{\sv}{\mathsf{v}}
\newcommand{\sw}{\mathsf{w}}

\newcommand{\hX}{\hat{X}}
\newcommand{\hY}{\hat{Y}}

\newcommand{\hx}{\hat{x}}
\newcommand{\hy}{\hat{y}}
\newcommand{\hxi}{\hat{\xi}}
\newcommand{\htheta}{\hat{\theta}}
\newcommand{\fX}{{\mathfrak{X}}}

\newcommand{\tGa}{\widetilde{\Gamma}}
\newcommand{\txi}{\widetilde{\xi}}
\newcommand{\tPhi}{\widetilde{\Phi}}
\newcommand{\tSi}{\widetilde{\Sigma}}
\newcommand{\trho}{\widetilde{\rho}}
\newcommand{\tcL}{\widetilde{\cL}}
\newcommand{\tbK}{\widetilde{\bK}}
\newcommand{\tbH}{\widetilde{\bH}}
\newcommand{\tfr}{\widetilde{\fr}}
\newcommand{\tfs}{\widetilde{\fs}}

\newcommand{\tsu}{\widetilde{\su}}
\newcommand{\tsw}{\widetilde{\sw}}

\newcommand{\tA}{\widetilde{A}}
\newcommand{\tB}{\widetilde{B}}
\newcommand{\tC}{\widetilde{C}}
\newcommand{\tD}{\widetilde{D}}
\newcommand{\tF}{\widetilde{F}}
\newcommand{\tH}{\widetilde{H}}

\newcommand{\tL}{\widetilde{L}}
\newcommand{\tM}{\widetilde{M}}
\newcommand{\tN}{\widetilde{N}}
\newcommand{\tQ}{\widetilde{Q}}

\newcommand{\tS}{\widetilde{S}}

\newcommand{\tU}{\widetilde{U}}
\newcommand{\tV}{\widetilde{V}}
\newcommand{\tX}{\widetilde{X}}
\newcommand{\tY}{\widetilde{Y}}
\newcommand{\tZ}{\widetilde{Z}}

\newcommand{\tb}{\widetilde{b}}
\newcommand{\tc}{\widetilde{c}}

\newcommand{\tp}{\widetilde{p}}

\newcommand{\tu}{\widetilde{u}}
\newcommand{\tw}{\widetilde{w}}
\newcommand{\tx}{\widetilde{x}}
\newcommand{\ty}{\widetilde{y}}

\newcommand{\tGamma}{\widetilde{\Gamma}}
\newcommand{\tbsi}{{\widetilde{\bsi}}}

\newcommand{\tmu}{\widetilde{\mu}}
\newcommand{\tnu}{\widetilde{\nu}}
\newcommand{\tTheta}{\widetilde{\Theta}}
\newcommand{\tgamma}{\widetilde{\gamma}}

\newcommand{\tbT}{\widetilde{\bT}}

\newcommand{\tNef}{\widetilde{\Nef}}
\newcommand{\tNE}{\widetilde{\NE}}

\newcommand{\vGa}{\vec{\Gamma}}
\newcommand{\bGa}{\mathbf{\Gamma}}
\newcommand{\Mbar}{\overline{\cM}}

\newcommand{\spa}{ {\ \ \,} }

\newcommand{\Cbar}{\overline{C}}

\newcommand{\ST}{ {S_{\bT}} }
\newcommand{\RT}{ {R_{\bT}} }
\newcommand{\bST}{ {\bar{S}_{\bT}} }
\newcommand{\bSTQ}{ \bST[\![ \tQ,\tau'' ]\!] }

\newcommand{\Rt}{ {R_{\bT'}} }
\newcommand{\bSt}{ {\bar{S}_{\bT'}} }

\newcommand{\nov}{\Lambda_{\mathrm{nov}}}
\newcommand{\novT}{\bar{\Lambda}^{\bT}_{\mathrm{nov}} }

\newcommand{\XX}{X_{(\tau,\si)}}
\newcommand{\YY}{Y_{(\tau,\si)}}
\newcommand{\lra}{\longrightarrow}

\newcommand{\chX}{\check{X}}
\newcommand{\chY}{\check{Y}}

\newtheorem{dummy}{dummy}[section]
\newtheorem{lemma}[dummy]{Lemma}
\newtheorem{theorem}[dummy]{Theorem}

\newtheorem{corollary}[dummy]{Corollary}
\newtheorem{proposition}[dummy]{Proposition}

\theoremstyle{definition}
\newtheorem{definition}[dummy]{Definition}
\newtheorem{example}[dummy]{Example}
\newtheorem{notation}[dummy]{Notation}
\newtheorem{convention}[dummy]{Convention}

\theoremstyle{remark}
\newtheorem{remark}[dummy]{Remark}

\numberwithin{equation}{section}

\begin{document}

\title{On the remodeling conjecture for toric Calabi-Yau 3-orbifolds}

\author{Bohan Fang}
\address{Bohan Fang, Beijing International Center for Mathematical
  Research, Peking University, 5 Yiheyuan Road, Beijing 100871, China}
\email{bohanfang@gmail.com}

\author{Chiu-Chu Melissa Liu}
\address{Chiu-Chu Melissa Liu, Department of Mathematics, Columbia University, 2990 Broadway, New York, NY 10027}
\email{ccliu@math.columbia.edu}

\author{Zhengyu Zong}
\address{Zhengyu Zong, Yau Mathematical Sciences Center, Tsinghua University, Jin Chun Yuan West Building,
Tsinghua University, Haidian District, Beijing 100084, China}
\email{zyzong@mail.tsinghua.edu.cn}

\subjclass[2010]{Primary 14N35, Secondary 14J33}

\date{}

\begin{abstract}
The Remodeling Conjecture proposed by Bouchard-Klemm-Mari\~{n}o-Pasquetti (BKMP) relates the A-model open
and closed topological string amplitudes (the all genus open and closed Gromov-Witten invariants) of a semi-projective toric
Calabi-Yau 3-manifold/3-orbifold  to the Eynard-Orantin invariants of its mirror curve.
It is an all genus open-closed mirror symmetry for toric Calabi-Yau 3-manifolds/3-orbifolds.
In this paper, we present a proof of the BKMP Remodeling Conjecture for all genus open-closed orbifold Gromov-Witten invariants of an arbitrary semi-projective toric Calabi-Yau 3-orbifold relative to an outer framed Aganagic-Vafa Lagrangian brane. We also prove
the conjecture in the closed string sector at all genera.
\end{abstract}

\maketitle

{\small \tableofcontents}

\section{Introduction}

\subsection{Background and motivation}
Mirror symmetry is a duality from string theory originally discovered by physicists. It says two
dual string theories -- type IIA and type IIB -- on different Calabi-Yau 3-folds give rise to the same
physics. Mathematicians became interested in this relationship around
1990 when Candelas, de la Ossa, Green, and Parkes \cite{CdGP} obtained
a conjectural formula of the number of rational curves of arbitrary
degree in the quintic 3-fold by relating it to period integrals
of the quintic mirror. By late 1990s mathematicians had
established the foundation of Gromov-Witten (GW) theory as a mathematical theory of A-model topological closed strings. In this context, the genus $g$ free energy of the topological A-model on a Calabi-Yau 3-fold $\cX$ is
defined as a generating function $F_g^{\cX}$ of genus $g$
Gromov-Witten invariants of $\cX$, which is a function on a (formal)
neighborhood around the large radius limit in the complexified K\"{a}hler moduli of $\cX$. The genus $g$ free energy of the topological B-model on the mirror Calabi-Yau 3-fold $\check{\cX}$
is a section of $\mathcal{V}^{2g-2}$, where $\mathcal{V}$ is the
Hodge line bundle over the complex moduli $\check{\cM}$ of $\check{\cX}$, whose fiber over
$\check{\cX}$ is $H^0(\check{\cX},\Omega^3_{\check{\cX}})$. Locally it is a function
$\check{F}_g^{\check{\cX}}$ near the large radius limit on the complex moduli $\check{\cX}$.
Mirror symmetry predicts that $\check{F}_g^{\check{\cX}}= F_g^{\cX} + \delta_{g,0}a_0 + \delta_{g,1} a_1$ under the mirror map,
where $a_0$  (resp. $a_1$) is a cubic (resp. linear) function in K\"{a}hler parameters.
The mirror map and $\check{F}_0^{\check{\cX}}$ are determined by period integrals of a holomorphic 3-form on $\check{\cX}$.
Period integrals on toric manifolds and complete intersections in them
can be expressed in terms of  explicit hypergeometric functions. This
conjecture has been proved in many cases to various degrees. Roughly
speaking, our result is about (a much more generalized version of) this
conjecture when $\cX$ is a toric Calabi-Yau $3$-orbifold.

\subsubsection{Mirror symmetry for compact Calabi-Yau manifolds}
Givental \cite{G96} and Lian-Liu-Yau \cite{LLY97} independently proved
the genus zero mirror formula for the quintic Calabi-Yau 3-fold $Q$; later they
extended their results to Calabi-Yau complete intersections in
projective toric manifolds \cite{G98, LLY99, LLY3}.
Bershadsky-Cecotti-Ooguri-Vafa (BCOV) conjectured the genus-one and genus-two mirror formulae
for the quintic 3-fold \cite{BCOV}. The BCOV genus-one  mirror formula was first proved
by A. Zinger  in \cite{Zi09} using genus-one reduced Gromov-Witten theory, and later reproved in \cite{KimL, CFKim} via quasimap theory and
in \cite{CGLZ} via MSP theory. The BCOV genus-two mirror formula was recently proved by Guo-Janda-Ruan \cite{GJR} and Chang-Guo-Li \cite{CGL}.
Combining the techniques of BCOV, results of
Yamaguchi-Yau \cite{YY04}, and boundary conditions,
Huang-Klemm-Quackenbush \cite{HKQ} proposed a mirror
conjecture on $F_g^Q$ up to $g=51$. The mirror conjecture
on $F_g^Q$ is open for $g> 2$. One difficulty is that mathematical
theory of higher genus B-model on a general compact
Calabi-Yau manifold has not been developed until recently. In 2012, Costello and Li initiated
a mathematical analysis of the BCOV theory \cite{CoLi1} based on the effective
renormalization method developed by Costello \cite{Co}. One essential
idea in their construction  \cite{CoLi2} is to introduce open
topological strings on the  B-model. The higher genus B-model potentials are then uniquely determined by the genus-zero open B-model potentials. We will see later that this phenomenon also arises in the BKMP Remodeling Conjecture, where the higher genus B-model potentials are determined by the genus-zero open B-model potentials via the Eynard-Orantin recursion.

\subsubsection{Gromov-Witten invariants of  toric Calabi-Yau 3-manifolds/3-orbifolds}\label{sec:orbiGW}
The technique of virtual localization \cite{GP99} reduces all genus Gromov-Witten invariants of
toric orbifolds to Hodge integrals. When the toric orbifold $\cX$ is a
smooth Calabi-Yau 3-fold, the Topological
Vertex \cite{AKMV,LLLZ,MOOP} provides an efficient algorithm to
compute these integrals, and thus to compute Gromov-Witten invariants
of $\cX$ as well as open Gromov-Witten invariants of $\cX$
relative to an Aganagic-Vafa Lagrangian brane $\cL$ (defined in
\cite{KL, DF, Liu02, LLLZ} in several ways), in all genera and
degrees. The algorithm of the topological vertex is equivalent to the Gromov-Witten/Donaldson-Thomas
correspondence for smooth toric Calabi-Yau threefolds \cite{MNOP1}; it provides a combinatorial formula for
a generating function of all genus Gromov-Witten invariants of a fixed degree.
Recently, this effective algorithm has been generalized to toric Calabi-Yau 3-orbifolds
with transverse $A_n$ singularities \cite{Zo15, RZ13, RZ14, R14}, but not for more general toric Calabi-Yau 3-orbifolds.

\subsubsection{Mirror symmetry for toric Calabi-Yau 3-manifolds/orbifolds}
The topological B-model for the mirror
$\check{\cX}$ of a semi-projective toric Calabi-Yau 3-manifold/3-orbifold
$\cX$ can be reduced to a theory on the mirror curve of $\cX$ \cite{HoVa}.
Under mirror symmetry, $F_0^{\cX}$ corresponds to integrals of 1-forms on the mirror curve along loops,
whereas the generating function $F_{0,1}^{\cX,\cL}$ of genus-zero open Gromov-Witten invariants
(counting holomorphic disks in $\cX$ bounded by $\cL$) corresponds to
integrals of 1-forms on the mirror curve along paths \cite{AV,
  AKV}. Based on the work of Eynard-Orantin \cite{EO07} and Mari\~{n}o
\cite{Ma}, Bouchard-Klemm-Mari\~{n}o-Pasquetti  (BKMP) \cite{BKMP09, BKMP10} proposed a new formalism of the
topological B-model on the Hori-Vafa mirror $\check{\cX}$ of $\cX$ in terms of the Eynard-Orantin invariants $\omega_{g,n}$ of the mirror curve; $\omega_{g,n}$ are mathematically defined and can be effectively computable for all $g,n$. Eynard-Mari\~{n}o-Orantin \cite{EMO} showed that the non-holomorphic $\check{F}_g(\btau,\bar{\btau})$ defined by the Eynard-Orantin topological recursion satisfy the BCOV holomorphic anomaly equation, and also derived holomorphic anomaly equations in the open string sector. BKMP conjectured a precise correspondence, known as the BKMP Remodeling Conjecture, between $\omega_{g,n}$ (where $n>0$) and the generating function $F_{g,n}^{\cX,\cL}$ of
open Gromov-Witten invariants counting holomorphic maps from bordered Riemann surfaces with
$g$ handles and $n$ holes to $\cX$ with boundaries in $\cL$. In the closed string sector, BKMP conjectured that
the A-model genus $g$ Gromov-Witten potential
$F_g^\cX$ is equal to the B-model genus $g$ free energy 
$\check{F}_g := \displaystyle{ \lim_{\bar{\btau}\to \sqrt{-1}\infty} \check{F}_g(\btau,\bar{\btau}) }$ under the closed mirror map.
These conjectures, known as the BKMP Remodeling Conjecture in both open string and closed string sectors,
are all genus open-closed mirror symmetry, and provide an effective algorithm of
computing $F_{g,n}^{\cX,\cL}$ and $F_g^{\cX}$ recursively, for general semi-projective toric Calabi-Yau 3-orbifolds.

The open string sector of the Remodeling Conjecture for $\bC^3$ was proved independently
by  L. Chen \cite{Ch09} and J. Zhou \cite{Zh09}; the closed string sector of the Remodeling Conjecture for $\bC^3$ was proved
independently by Bouchard-Catuneanu-Marchal-Su{\l}kowski
\cite{BCMS} and S. Zhu \cite{Zhu}.  Eynard and Orantin provided a proof of the Remodeling Conjecture for general
smooth semi-projective toric Calabi-Yau 3-folds in \cite{EO15}. The authors proved the Remodeling Conjecture for
all semi-projective affine toric Calabi-Yau 3-orbifolds $[\bC^3/G]$ \cite{FLZ}.

\subsection{Statement of the main result and outline of the proof}
In this paper, we prove the BKMP Remodeling Conjecture for a general semi-projective toric Calabi-Yau 3-orbifold $\cX$,
in both the open string sector and the closed string sector. We consider a framed Aganagic-Vafa brane $(\cL,f)$ on an
outer leg of $\cX$, where $\cL \cong [(S^1\times \bC)/\bmu_\fm]$ for a finite abelian group
$\bmu_\fm\cong \bZ_\fm$ and $f\in \bZ$. This outer leg may be gerby with
a non-trivial isotropy group $\bmu_\fm$.  We define generating functions of open-closed Gromov-Witten invariants:
$$
F_{g,n}^{\cX,(\cL,f)}(\btau;\tilde{X}_1,\ldots, \tilde{X}_n)
$$
as $H^*_{\CR}(B\bmu_\fm;\bC)^{\otimes n}$-valued formal
power series in A-model closed string coordinates $\btau=(\tau_1,\ldots, \tau_{\fp})$
and A-model open string coordinates $\tilde{X}_1,\ldots, \tilde{X}_n$;
here $H^*_{\CR}(B\bmu_\fm;\bC)\cong \bC^\fm$ is the Chen-Ruan orbifold cohomology
of the classifying space $B\bmu_\fm$ of $\bmu_\fm$. (The precise definition of
$F_{g,n}^{\cX,(\cL,f)}$ is given in Section \ref{sec:open-closed-GW}.)

On the other hand, we use the Eynard-Orantin invariants $\omega_{g,n}$ of the framed mirror curve
to define B-model potentials
$$
\check{F}_{g,n}(q;\hat{X}_1,\ldots,\hat{X}_n)
$$
as $H^*_{\CR}(\cB\bmu_\fm;\bC)^{\otimes n}$-valued
functions in B-model closed string coordinates (complex parameters)
$q=(q_1,\ldots, q_{\fp})$ and B-model open string
coordinates $\hat{X}_1,\ldots,\hat{X}_n$. (The precise definitions of
$\omega_{g,n}$ and $\check{F}_{g,n}$ are given in Section \ref{sec:eynard-orantin} and Section \ref{sec:B-potential}, respectively.)
They are analytic in an open neighborhood of the origin in $\bC^{\fp}\times\bC^n$. The closed mirror map
relates the flat coordinates $(\tau_1,\ldots,\tau_{\fp})$
to the complex parameters $(q_1,\ldots, q_{\fp})$ of the mirror curve and the open mirror map relates
the A-model open string coordinates $\tilde{X}_1,\ldots, \tilde{X}_n$ to the B-model open string
coordinates $\hat{X}_1,\ldots,\hat{X}_n$.


Our first main result is the BKMP Remodeling Conjecture for the
open string sector:

\medskip
\paragraph{\bf Theorem \ref{main} (BKMP Remodeling Conjecture: open string sector)}
{\it For any $g\in \bZ_{\geq 0}$ and $n\in \bZ_{>0}$, under the open-closed mirror map $\btau=\btau(q)$ and $\widetilde X=\widetilde
X(q,\hat X)$,
\begin{equation}\label{eqn:main-eqn}
\check{F}_{g,n}(q;\hX_1,\ldots, \hX_n) = (-1)^{g-1+n}  F_{g,n}^{\cX, (\cL,f)}(\btau;\tX_1,\ldots, \tX_n).
\end{equation}
}

This is more general than the original conjecture in
\cite{BKMP10}, which covers the $m=1$ case, i.e. when $\cL$ is
on an effective leg.

In the closed string sector, we prove the BKMP Remodeling
Conjecture for free energies. We have the
following theorems under the closed mirror map $\btau=\btau(q)$.

\medskip
\paragraph{\bf Theorem \ref{thm:stable-Fg} (free energies at genus $g>1$)}
{\it When $g>1$, we have,}
\begin{equation}\label{eqn:stable-Fg}
F^{\cX}_g(\btau)= (-1)^{g-1} \check{F}_g(q).
\end{equation}
\medskip
\paragraph{\bf Theorem \ref{F1} (genus one free energy) }
{\it When $g=1$, we have,}
\begin{equation}\label{eqn:F1}
dF^\cX_1(\btau)=d\check{F}_1(q).
\end{equation}
\medskip
\paragraph{\bf Theorem \ref{F0} (genus zero free energy)}
{\it For any $i,j,k\in\{1,\cdots,\fp\}$, we have,}
\begin{equation}\label{eqn:F0}
\frac{\partial^3 F^\cX_0}{\partial\tau_i\partial\tau_j\partial\tau_k}(\btau)
=-\frac{\partial^3 \check{F}_0}{\partial\tau_i\partial\tau_j\partial\tau_k}(q).
\end{equation}

The key idea in the proof of the BKMP Remodeling Conjecture is that we
can realize the A-model and B-model higher genus potentials as
quantizations on two isomorphic semi-simple Frobenius structures. On
the A-model side, we use the Givental quantization formula to express
the higher genus GW potential of $\cX$ in terms of the Frobenius
structure of the quantum cohomology of $\cX$ (genus-zero data). On the
B-model side, the Eynard-Orantin recursion determines the higher genus
B-model potential by the genus-zero initial data. The bridge
connecting these two formalisms on A-model and B-model is the graph
sum formula. The quantization formula on the A-model is a formula involving the exponential of a quadratic differential operator. By the classical Wick formula, it can be rewritten as a graph sum formula:
$$
F_{g,n}^{\cX,(\cL,f)}  =\sum_{\vGa\in \bGa_{g,n}(\cX)}\frac{w^O_A(\vGa)}{|\Aut(\vGa)|}
$$
where $\bGa_{g,n}(\cX)$ is certain set of decorated graphs, $\Aut(\vGa)$ is the automorphism group of the decorated graph $\vGa$,
and $w^O_A(\vGa)$ is the A-model weight of the decorated graph $\vGa$.

On the B-model side, by the result in \cite{DOSS}, the Eynard-Orantin recursion is equivalent to a graph sum formula. So the B-model potential $\check{F}_{g,n}$ can also be expressed as a graph sum:
$$
\check{F}_{g,n} = \sum_{\vGa\in \bGa_{g,n}(\cX)}\frac{w^O_B(\vGa)}{|\Aut(\vGa)|}
$$
where $w^O_B(\vGa)$ is the B-model weight of the decorated graph
$\vGa$.  Then we reduce the BKMP Remodeling Conjecture to
$$
w^O_A(\vGa)=w^O_B(\vGa).
$$

The weights $w^O_A(\vGa)$ and $w^O_B(\vGa)$ are determined by the
A-model and B-model $R-$matrices (information extracted from the
Frobenius structures) together with the A-model and B-model disk
potentials. The disk mirror theorem in \cite{FLT} is precisely what we
need to match the disk potentials. The genus-zero mirror theorem \cite{CCIT, CCK} identifies the
equivariant quantum cohomology ring of $\cX$ with the
equivariant Jacobian ring of its Landau-Ginzburg B-model. In
particular, the quantum differential equations on the A-model and on
the Landau-Ginzburg B-model are identified. By the dimensional
reduction, we can show that the B-model $R-$matrix is indeed the
$R-$matrix in the fundamental solution of the B-model quantum
differential equation. The fundamental solution of the quantum
differential equation is unique up to a constant matrix. We
identify the A-model and B-model $R$-matrices by matching them in
degree zero. Putting these pieces together, we have
$$
w^O_A(\vGa)=w^O_B(\vGa).
$$

\subsection{Some remarks}
\label{sec:remarks}

We have the following remarks about our proof:

\begin{itemize}

\item Our proof does not rely on the equivalence of
the orbifold Gromov-Witten vertex (a generating function of Hurwitz Hodge
integrals \cite{Ro14})  and the orbifold Donaldson-Thomas vertex (a generating
function of colored 3d partitions). As mentioned in Section \ref{sec:orbiGW} above,
the equivalence is known for toric Calabi-Yau 3-orbifolds with transverse $A_n$-singularities
\cite{Zo15, RZ13, RZ14, R14}. It is not clear how to formulate the equivalence
for toric Calabi-Yau 3-orbifolds which do not satisfy the Hard Lefschetz condition.
Moreover, the structure of the algorithm for the Topological Vertex is very different
from the structure of the Eynard-Orantin topological recursion on the B-model. Roughly speaking,
the vertex algorithm comes from degeneration of the target, whereas the topological recursion
comes from degeneration of the domain. It seems very difficult, if not
impossible, to derive the Remodeling Conjecture from the Topological Vertex.

\item Instead, we study the GW theory of $\cX$ by Givental's
 quantization formula, which expresses the higher genus GW potential
 of $\cX$ in terms of the abstract Frobenius structure of the quantum
cohomology. In \cite{Zo}, the Givental quantization formula for
general GKM orbifolds is proved. So we can apply the result in
\cite{Zo} to the case of toric Calabi-Yau 3-orbifolds. It turns out
that the Givental quantization formula on the A-model matches the
Eynard-Orantin recursion on the B-model perfectly. In particular,
we provide new proofs of the BKMP Remodeling Conjecture in
the smooth case and the affine case.

\item The Remodeling Conjecture provides
a very effective recursive algorithm to compute closed and open-closed
Gromov-Witten invariants of all semi-projective toric Calabi-Yau
3-orbifolds at all genera (see \cite{BKMP09,BKMP10} for the numerical
computation). Before the introduction of this algorithm, these invariants were very difficult to
compute in the non-Hard-Lefschetz orbifold cases where the Topological Vertex was not applicable.

\item One key ingredient in Eynard-Orantin's recursive algorithm is the open
topological string. Only by including the open topological strings can
we determine the higher genus topological strings from the genus zero
data by the Eynard-Orantin recursion. This philosophy is in line with
the method of Costello-Li \cite{CoLi1,CoLi2}, and may be enlightening
for further study of mirror symmetry.

\item The disk mirror theorem proved in \cite{FLT} covers both outer and inner branes, so the statement and the proof of Theorem \ref{main} can be extended to the case where $\cL$ is an inner brane.  When $\cL$ is an outer brane, the left (resp. right) hand side of  Equation \eqref{eqn:main-eqn} involves only positive powers of $\hat{X}_i$ (resp. $\tX_i$); when $\cL$ is an inner brane, the left (resp. right) hand side of Equation \eqref{eqn:main-eqn} involves both positive and negative powers of
$\hat{X}_i$ (resp. $\tX_i$).

\end{itemize}

\subsection{Future work}
The BKMP Remodeling Conjecture has many interesting applications. We discuss two of them:
all genus open-closed Crepant Transformation Conjecture for toric Calabi-Yau 3-orbifolds and modularity for all genus open-closed GW potentials of toric Calabi-Yau 3-orbifolds.

\subsubsection{The all genus open-closed Crepant Transformation Conjecture for toric Calabi-Yau 3-orbifolds}
The Crepant Transformation Conjecture, proposed by Ruan \cite{Ruan1, Ruan2}
and later generalized by others in various situations, relates GW theories of
K-equivalent smooth varieties, orbifolds, or Deligne-Mumford stacks.
To establish this equivalence, one may need to do change of variables, analytic continuation, and symplectic
transformation for the GW potential. In
general, the higher genus Crepant Transformation Conjecture is difficult
to formulate and prove. Coates-Iritani introduced the Fock sheaf formalism and proved
all genus Crepant Transformation Conjecture for compact toric orbifolds \cite{CI}.
The Remodeling Conjecture leads to simple formulation and proof of  all genus Crepant Transformation Conjecture for
semi-projective toric CY 3-orbifolds (which are always non-compact). The key point here is that our higher genus
B-model, defined in terms of Eynard-Orantin invariants of the mirror curve, is \emph{global} and \emph{analytic}. One can use the secondary fan to construct a global
B-model closed string moduli space, over which we construct a {\em global} family of mirror curves.

\subsubsection{Modularity for all genus open-closed GW potentials of toric Calabi-Yau 3-orbifolds}
The modularity of the GW potentials of Calabi-Yau 3-folds has been studied in
\cite{ABK, JieZhou}. In these works, the modularity of the GW potentials
provides a powerful tool to construct higher genus B-models. It also
produces closed formulae for some GW potentials in terms of
quasi-modular forms \cite{JieZhou}. The mathematical
proof of the modularity for GW potentials remains a difficult problem
in general. For toric Calabi-Yau 3-orbifolds, the Remodeling
Conjecture relates the GW potential to the Eynard-Orantin invariants
of the mirror curve. Eynard and Orantin studied the modularity of the
Eynard-Orantin invariants of any spectral curves \cite{EO07}. This
modularity follows from the modularity of the fundamental differential
of the spectral curve. Therefore, the Remodeling Conjecture should imply the modularity for all genus open-closed GW potentials of toric Calabi-Yau 3-orbifolds.


\subsection{Overview of the paper}

In Section \ref{sec:A-geometry}, we fix the notation of toric
varieties and orbifolds. We also discuss the geometry of toric Calabi-Yau $3$-orbifolds
and Aganagic-Vafa branes in them.

In Section \ref{sec:A-string}, we introduce the equivariant GW invariants, as well as open-closed GW invariants relative to Aganagic-Vafa branes.
Section \ref{sec:bigQH} to \ref{sec:Agraph} are on the quantization of the Frobenius manifolds from
big equivariant quantum cohomology; the graph sum formula from
\cite{Zo} expressing all genus descendant potential for toric
orbifolds is stated in Section \ref{sec:Agraph}. In Section \ref{sec:I-J} to Section \ref{sec:open-closed-GW},
we consider restriction to the small phase space. We recall the genus zero mirror theorem from \cite{CCIT} in Section \ref{sec:I-J}
and  define A-model open potentials $F_{g,n}^{\cX,(\cL,f)}$ in Section \ref{sec:open-closed-GW}.

Section \ref{sec:HV-LG} defines three different
mirrors of a toric Calabi-Yau $3$-orbifold: the Hori-Vafa mirror, the equivariant Landau-Ginzburg mirror,
and the mirror curve. Section \ref{sec:HV}  describes dimensional reduction from  the genus-zero B-model on the 3-dimensonal  Hori-Vafa mirror  to a theory on the mirror curve in terms of period integrals. Section \ref{sec:LG} describes dimensional reduction from  the 3-dimensional equivariant Landau-Ginzburg model to a Landau-Ginzburg model on the mirror curve, in terms of Frobenius algebras and oscillatory integrals.

In Section \ref{sec:mirror-curve-geometry}, we study geometry and topology of mirror
curves. In particular, we construct a family of mirror curves near the
limit point in the B-model moduli space in Section \ref{sec:toric-degeneration}. In Section \ref{sec:B-flat}, we
introduce flat coordinates on the B-model moduli space,  and identify each B-model flat coordinate  with a specific solution to a Picard-Fuch equation which is a component of the mirror map (expressible in terms of explicit hypergeometric series).

In Section \ref{sec:b-string}, we recall the Eynard-Orantin's topological recursion \cite{EO07}, and the graph sum formula of Eynard-Orantin
invariants $\omega_{g,n}$ derived in \cite{DOSS}. Using the disk mirror theorem \cite{FLT}, we expand this graph sum formula  around suitable puncture(s) on the
mirror curve and obtain a graph sum formula of the B-model potential
$\check{F}_{g,n}$. In Section \ref{sec:mirror-symmetry}, we finish the proof of the
Remodeling Conjecture by comparing the A-model and B-model graph sums.

\subsection{List of notations}
\begin{center}
\begin{tabular}{| c | p{6.5cm} | p{5.5cm}|}
\hline
Notations    & Explanation & Remark\\
\hline
$\cX$ & a toric CY $3$-orbifold, defined by a fan $\Si$ &  fan $\Si$ is the cone over a triangulated polytope $P$\\
\hline
$\fp'$ & $h^2(\cX)$ & $\fp'+3$ = number of $1$-cones in $\Si$ \\
\hline
$\fp$ & $h^2_\CR(\cX)$, number of extended K\"ahler parameters & $\fp+3$ = number of lattice points in $P$\\
\hline
$\fg$ & $h^4_\CR(\cX)$, also the genus of the compactified mirror curve $\overline C_q$ & also number of lattice points in the interior of $P$
\\
\hline
$\fn$ & number of punctures in the mirror curve, $\#(\overline C_q \setminus C_q)$ & also number of lattice points on the boundary of $P$, $\fn =\fp +3 -\fg$ \\
\hline
$\bT$  & torus acting on the toric CY 3-fold $\cX$ & $\bT\cong (\bC^*)^3$ \\
\hline
$\bT'$ & Calabi-Yau torus $\subset \bT$ preserving the CY form & $\bT'\cong (\bC^*)^2$\\
\hline
$\cL$ & a fixed outer Aganagic-Vafa brane & location (phase) given by $\si_0\in \Si(3), \tau_0\in \Si(2)$\\
\hline
$H^*_{\CR,\bT}$, $H^*_{\CR,\bT'}$ & equivariant Chen-Ruan  cohomology & \\
\hline
$QH_{\bT}^*$, $QH_{\bT'}^*$ & equivariant quantum cohomology  & \\
\hline
$I_\Si$ & index set of canonical basis for $QH_{\bT}^*(\cX), H^*_{\CR,\bT}(\cX)$, etc. & $I_\Si=\{\bsi=(\si,\gamma): \si\in \Si(3),\gamma\in G^*_\si\}$\\
\hline
$\phi_\bsi(t)$, $\hat \phi_\bsi(t)$ & canonical and normalized canonical basis of equivariant quantum cohomlogy of $\cX$  &\\
\hline
$\phi_\bsi$, $\hat\phi_\bsi$ & canonical and normalized canonical basis of equivariant Chen-Ruan  cohomology  of $\cX$ & $\phi_\bsi=\phi_\bsi(0)$, $\hat\phi_\bsi=\hat\phi_\bsi(0)$\\
\hline
$\Delta^\bsi(\tau)$ & $\sqrt{\frac{1}{\Delta^\bsi(\tau)}}$ is the length of $\phi_\bsi(\tau)$ & $h^\bsi_1=\sqrt{\frac{-2}{\Delta^\bsi(\tau)}}$\\
\hline
$H(X,Y,q)=0$ & mirror curve equation & defines the mirror curve $C_q\subset (\bC^*)^2$\\
\hline
$W^{\bT'}$ & $\bT'$-equivariant LG superpotential  & $W^{\bT'}=H(X,Y,q)Z +\su_1 x+\su_2 y$\\
\hline
$\bold{p}_\bsi$ & critical points of $W^{\bT'}$, in $(\bC^*)^3$ & labeled by $\bsi \in I_\Si$ \\
\hline
$p_\bsi$ & critical points of $\hat x=\su_1 x+\su_2 y$ on the  mirror curve $C_q$ & labeled by $\bsi\in I_\Si$ \\

\hline
$\Jac(W^{\bT'})$ & Jacobian ring of $W^{\bT'}$ & $\cong QH_{\bT'}^*(\cX)$ under the mirror map, as a Frobenius algebra\\
\hline
$V_\bsi$ & canonical basis of $\Jac(W^{\bT'})$&  $V_\bsi(\bold{p}_{\bsi'})=\delta_{\bsi\bsi'}$ \\
\hline
$q_a,\ a=1\dots \fp$ & complex parameters mirror to extended K\"ahler parameters & depend on a choice of extended K\"ahler basis $H_1,\dots, H_\fp$\\
\hline
$\overline C_q$ & Compactified mirror curve at parameter $q$ & \\

\hline
$B(p_1,p_2)$ & the fundamental differential of the second kind a.k.a. the Bergman kernel & depends on the choice of $A$-cycles: $A_1,\dots,A_\fg\in H_1(\overline C_q;\bC)$\\
\hline
$F_{g,n}^{\cX,(\cL,f)}$ & A-model open GW potential & depends on $\cX$, $\cL$ and the framing $f$
\\
\hline
$\omega_{g,n}$ & B-model higher genus invariants from the Eynard-Orantin topological recursion& symmetric meromorphic form on $(\overline C_q)^n$.
\\
\hline
$\check F_{g,n}$ & B-model open potential & defined as the indefinite integral of $\omega_{g,n}$  \\
\hline
\end {tabular}
\end{center}

\subsection*{Acknowledgements}
We wish to thank Charles Doran, Bertrand Eynard, and Nicolas Orantin for helpful conversations. We wish to thank Motohico Mulase and Yongbin Ruan for encouragement.
We wish to thank Sheldon Katz, Jie Zhou, and Shengmao Zhu for their comments on earlier versions of this paper.
The first author is partially supported by a start-up grant at
Peking University. The second author is partially supported by NSF grants DMS-1206667 and DMS-1159416. The third author is partially supported by the start-up grant at Tsinghua University.

\section{A-model Geometry and Topology}
\label{sec:A-geometry}

We work over $\bC$. In this section, we give a brief review of semi-projective toric Calabi-Yau 3-orbifolds.
We refer to \cite{Fu, CLS} for the theory of general toric varieties. We refer to
\cite{BCS05,FMN} for the theory of general smooth toric Deligne-Mumford (DM) stacks.
In Section \ref{sec:lattices} and Section \ref{sec:kahler}, we specialize the definitions
in \cite[Section 3.1]{Iri} to toric Calabi-Yau 3-orbifolds.

\subsection{The simplicial toric variety and the fan}\label{sec:fan}
Let $N\cong \bZ^3$ be a lattice of rank $3$. Let $X_\Si$ be a 3-dimensional
simplicial toric variety defined by a (finite) simplicial fan
$\Si$ in $N_\bR:=N\otimes \bR$.  Then $X_\Si$ contains the algebraic torus $\bT= N\otimes \bC^* \cong (\bC^*)^3$
as a open dense subset, and the action of $\bT$ on itself extends to $X_\Si$. We further assume that:
\begin{enumerate}
\item[(i)] $X_\Si$ is {\em Calabi-Yau}: the canonical divisor of $X_\Si$ is trivial;
\item[(ii)] $X_\Si$ is {\em semi-projective}: the $\bT$-action on $X_\Si$ has at least one fixed
point, and the morphism from $X_\Si$ to its affinization $X_0=\Spec H^0(X_\Si,\cO_{X_\Si})$
is projective.
\end{enumerate}
We introduce some notation:
\begin{itemize}
\item Let $\Si(d)$ be the set of $d$-dimensional cones in $\Si$.
\item Let $\Si(1)=\{\rho_1,\ldots,\rho_{3+\fp'}\}$ be the set of 1-dimensional cones
in $\Si$, where $\fp'\in \bZ_{\geq 0}$, and let $b_i\in N$ be characterized by $\rho_i\cap N = \bZ_{\geq 0} b_i$.
\end{itemize}

The lattice $N$ can be canonically identified with $\Hom(\bC^*,\bT)$, the cocharacter lattice of $\bT$; the dual
lattice $M=\Hom(N,\bZ)$ can be canonically identified with $\Hom(\bT,\bC^*)$, the character lattice of $\bT$.
Given $m\in M$, let $\chi^m \in \Hom(\bT,\bC^*)$ denote the
corresponding character of $\bT$. Let $M_\bR := M\otimes \bR$ be the
dual real vector space of $N_\bR$. The Calabi-Yau condition (i) implies
that, there exists a vector $e_3^*\in M$ such that
$\langle e_3^*, b_i\rangle =1$ for $i=1,\ldots, 3+\fp'$. We may choose
$e_1^*, e_2^*$ such that $\{ e_1^*, e_2^*, e_3^*\}$ is a $\bZ$-basis of $M$. Let
$\{e_1, e_2,e_3\}$ be the dual $\bZ$-basis of $N$, which defines an isomorphism
$N\cong \bZ^3$ given by  $n_1 e_1 + n_2 e_2 + n_3 e_3 \longmapsto (n_1,n_2,n_3)$; under this
isomorphism, $b_i = (m_i,n_i,1)$ for some $(m_i,n_i)\in \bZ^2$. We define
the Calabi-Yau subtorus of $\bT$ to be $\bT':= \Ker(\chi^{e_3^*}:\bT\to \bC^*)\cong (\bC^*)^2$.
Then  $N':= \Ker(e_3^*:N\to \bZ)\cong \bZ^2$ can be canonically identified with $\Hom(\bC^*, \bT')$,
the cocharacter lattice of the Calabi-Yau torus $\bT'$.
Let $P\subset N'_\bR:= N'\otimes_\bZ \bR \cong \bR^2$ be the convex hull of $\{(m_i,n_i):i=1,\ldots,3+\fp'\}$, and
let $\underline{\si}\subset N_\bR$ be the cone over $P\times \{1\} \subset N_\bR'\times \bR=N_\bR$.
Then
$$
\underline{\si}=\bigcup_{\si\in \Si(3)} \si
$$
is a 3-dimensional strongly convex polyhedral cone. We have
$$
H^0(X_\Si,\cO_{X_\Si})=\bC[M\cap \underline{\si}^\vee]
$$
where $\underline{\si}^\vee \subset M_\bR$ is the dual cone of $\underline{\si}\subset N_\bR$. Therefore,
the affine toric variety defined by the cone $\underline{\si}$ is the affinization
$X_0$ of $X_\Si$.

There is a group homomorphism
$$
\phi':\tN':=\bigoplus_{i=1}^{3+\fp'} \bZ \tb_i \to N,\quad \tb_i\longmapsto b_i
$$
with finite cokernel. Applying $-\otimes_{\bZ}\bC^*$, we obtain an exact sequence of abelian groups
$$
1\to G_\Si \to \tbT'\to \bT \to 1,
$$
where $\tbT'=\tN'\otimes\bC^* \cong (\bC^*)^{3+\fp'}$, and $G_\Si$ can be disconnected.  The action of
$\tbT'$ on itself extends to $\bC^{3+\fp'}=\Spec\bC[Z_1,\ldots, Z_{3+\fp'}]$. Let
$Z_\si:=\prod_{\rho_i\not\subset \si }Z_i$, and let $Z(\Si)\subset \bC^{3+\fp'}$ be the closed
subvariety defined by the ideal generated by $\{Z_\si: \si\in \Si\}$.
Then $\tbT'$ acts on $U_\Si:=\bC^{3+\fp'}- Z(\Si)$, and
the simplicial toric variety $X_\Si$  is the geometric quotient
$$
X_\Si = U_\Si /G_\Si.
$$

\subsection{The toric orbifold and the stacky fan} \label{sec:stacky}
In general, a smooth toric DM stack is defined by a stacky fan
$\bSi=(N,\Si,\beta)$ \cite{BCS05}, and a toric orbifold is a smooth toric DM stack
with trivial generic stabilizer \cite[Section 5]{FMN}.
The canonical stacky fan associated to the simplicial fan $\Si$
in Section \ref{sec:fan} is
$$
\bSi^\can =(N,\Si,\beta^\can=(b_1,\ldots, b_{3+\fp'})).
$$
The toric orbifold $\cX$ defined by $\bSi^\can$ is the stacky quotient
$$
\cX =[U_\Si/G_\Si].
$$
In this paper, we consider semi-projective toric Calabi-Yau 3-orbifolds $\cX$
constructed as above. A more general toric orbifold $\cX_{\bSi}$ (with the same
coarse moduli $X_\Si$) is obtained by
a more general choice of $\beta=(a_1 b_1,\ldots, a_{3+\fp'}b_{3+\fp'})$, where
$a_1,\ldots, a_{3+\fp'}$ are positive integers.

We will also need an alternative description of $\cX$ in terms of an extended
stacky fan introduced by Y. Jiang \cite{J}. For any toric Calabi-Yau 3-orbifold, there
is a canonical extended stacky fan
$$
\bSi^\ext=(N,\Si, \beta^\ext=(b_1,\ldots,b_{3+\fp}))
$$
where $b_i=(m_i,n_i,1)$ and
$$
\{ (m_i,n_i): i=1\ldots, 3+\fp\} = P\cap \bZ^2.
$$
There is a surjective group homomorphism
$$
\phi:\tN:=\bigoplus_{i=1}^{3+\fp}\bZ \tb_i \to N,\quad \tb_i\longmapsto b_i
$$
Let $\bL = \Ker(\phi)\cong \bZ^\fp$. Then we have a short exact sequence of free $\bZ$-modules:
\begin{equation}\label{eqn:tNN}
0\to \bL \stackrel{\psi}{\to} \tN \stackrel{\phi}{\to} N\to 0.
\end{equation}
Applying $-\otimes_{\bZ}\bC^*$, we obtain an exact sequence of abelian Lie groups:
\begin{equation}\label{eqn:tTT}
1\to G \to \tbT\to \bT \to 1,
\end{equation}
where $\tbT=\tN\otimes\bC^* \cong \bC^{3+\fp}$, and $G=\bL\otimes \bC^* \cong (\bC^*)^{\fp}$ .  The action of
$\tbT$ on itself extends to $\bC^{3+\fp}=\Spec\bC[Z_1,\ldots, Z_{3+\fp}]$
and preserves the Zariski open dense subset  $U_{\bSi^\ext}=U_\Si\times (\bC^*)^{\fp-\fp'}\subset \bC^{3+\fp}$.
Then
$$
X_\Si = U_{\bSi^\ext}/G,\quad
\cX =[U_{\bSi^\ext}/G].
$$
The action of $\tbT$ on $\bC^{3+\fp}$ restricts to an action of the subgroup $G$ on $\bC^{3+\fp}$. Let $\chi_i\in \Hom(G,\bC^*)$ be the
character of the $G$-action on the $i$-th coordinate $Z_i$ of $\bC^{3+\fp}$.

\begin{example}
  \label{exp:polytope}
A polytope $P$ is illustrated in Figure \ref{fig:defining-polytope}. It is triangulated, and the vertices are $(0,0)$, $(0,2)$, $(1,0)$, $(2,-1)$. One regards $P\subset N'_\bR\times\{1\}\subset N_\bR\cong \bR^3$, and can define a fan whose $3$-cones (resp. $2$ and $1$-cones) of the fan $\Si$ are cones at $(0,0,0)\in N_\bR$ over the  triangles (resp. solid segments, vertices) in $P$. In this example, $\fp=2$, $\fp'=1$, and $b_1=(1,0,1),b_2=(0,2,1),b_3=(0,0,1),b_4=(2,-1,1)$. In the extended fan $\bSi^\ext$, we also have $b_5=(0,1,1)$.
\begin{figure}[h]
\begin{center}
\psfrag{si0}{\small $\si_0'$}
\psfrag{si1}{\small $\si_1'$}
\psfrag{si2}{\small $\si_2'$}
\psfrag{tau0}{\small $\tau_0'$}
\includegraphics[scale=0.6]{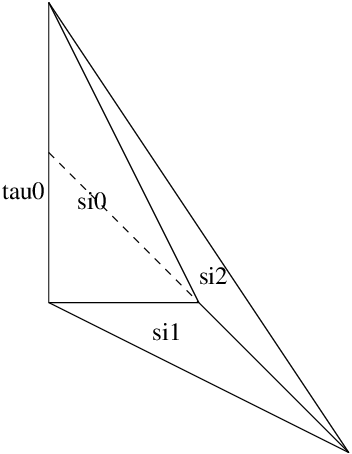}
\caption{The defining polytope of a toric Calabi-Yau $3$-orbifold $\cX$. It defines a fan $\Si$ where $3$-cones $\si_i$ are cones over $\si_i'$.}
\label{fig:defining-polytope}
\end{center}
\end{figure}
\end{example}

\subsection{Character lattices and integral second cohomology groups} \label{sec:lattices}
Let $\tM:=\Hom(\tN,\bZ)$ be the dual lattice of $\tN$, which can be canonically identified
with the character lattice $\Hom(\tbT,\bC^*)$ of $\tbT$. The $\tbT$-equivariant
inclusion $U_\Si\hookrightarrow \bC^{3+\fp}$ induces a surjective group homomorphism
$$
\kappa_{\bT}:\tM\cong H^2_{\tbT}(\bC^{3+\fp};\bZ)\cong H^2_{\bT}([\bC^{3+\fp}/G];\bZ)
\longrightarrow  H^2_{\tbT}(U_\Si;\bZ)\cong H^2_{\bT}(\cX;\bZ).
$$
Let $D_i^{\bT}\in H^2_{\tbT}(\bC^{3+\fp};\bZ)$ be the $\tbT$-equivariant Poincar\'{e} dual of the divisor $\{Z_i=0\}\subset\bC^{3+\fp}$.
Then $\{ D_i^{\bT}:i=1,\ldots,{3+\fp}\}$ is the $\bZ$-basis of $\tM \cong H_{\tbT}^2(\bC^{3+\fp};\bZ)$ which is
dual to the $\bZ$-basis $\{\tb_i:i=1,\ldots,{3+\fp}\}$ of $\tN$. We have
$$
\Ker(\kappa_{\bT})= \bigoplus_{i=4+\fp'}^{3+\fp}\bZ D_i^\bT.
$$

Let $\bL^\vee=\Hom(\bL,\bZ)$ be the dual lattice of $\bL$, which can be canonically
identified with the character lattice $\Hom(G,\bC^*)$ of $G$. The $G$-equivariant inclusion
$U_\Si\hookrightarrow \bC^{3+\fp}$ induces a surjective group homomorphism
$$
\kappa:\bL^\vee\cong H^2_G(\bC^{3+\fp};\bZ)\cong H^2([\bC^{3+\fp}/G];\bZ)  \longrightarrow  H^2_G(U_\Si;\bZ)\cong H^2(\cX;\bZ).
$$
We have
$$
\Ker(\kappa)=\bigoplus_{i=4+\fp'}^{3+\fp} \bZ D_i
$$
where $D_i\in H^2_G(\bC^{3+\fp};\bZ)$ is the $G$-equivariant Poincar\'{e} dual of the divisor
$\{Z_i=0\}\subset \bC^{3+\fp}$.

Applying $\Hom(-;\bZ)$ to \eqref{eqn:tNN}, we obtain the following
short exact sequence of free $\bZ$-modules:
\begin{equation}\label{eqn:MtM}
0 \to M  \stackrel{\phi^\vee}{\to} \tM  \stackrel{\psi^\vee}{\to} \bL^\vee \to 0.
\end{equation}

To summarize, we have the following commutative diagram:
\begin{equation}\label{eqn:diagram}
\begin{CD}
&& 0 && 0 && 0 \\
& & @VVV  @VVV  @VVV \\
0 @>>> 0 @>>> \bigoplus_{i=4+\fp'}^{3+\fp} \bZ D_i^{\bT} @>{\cong}>> \bigoplus_{i= 4+\fp'}^{3+\fp} \bZ D_i @>>> 0\\
& & @VVV  @VVV  @VVV \\
0 @>>> M  @>{\phi^\vee}>> \tM  @>{\psi^\vee}>> \bL^\vee @>>> 0\\
& & @VV{\cong}V  @VV{\kappa_\bT}V  @VV{\kappa}V & \\
0 @>>> M  @>{\bar{\phi}^\vee}>> H^2_{\bT}(\cX;\bZ)  @>{\bar{\psi}^\vee}>> H^2(\cX;\bZ) @>>> 0\\\
&&  @VVV @VVV @VVV & \\
&& 0 && 0 && 0 &  \\
\end{CD}
\end{equation}
In the above diagram, the rows and columns are short exact sequences of
abelian groups.  The map $\psi^\vee$ sends $D_i^\bT$ to $D_i$. For $i=1,\ldots,3+\fp$,
we define
$$
\bar{D}_i^{\bT}:= \kappa_{\bT}(D_i^{\bT})\in H^2_{\bT}(\cX;\bZ),\quad
\bar{D}_i :=\kappa(D_i)\in H^2(\cX;\bZ).
$$
Then $\bar{\psi}^\vee(\bar{D}_i^{\bT})= \bar{D}_i$, and
$\bar{D}_i^{\bT}=\bar{D}_i =0$ for $4+\fp'\leq i\leq 3+\fp$.

Finally, we have the following isomorphisms
$$
c_1:\Pic(\cX)\stackrel{\cong}{\to} H^2(\cX;\bZ), \quad
c_1^{\bT}:\Pic_{\bT}(\cX)\stackrel{\cong}{\to} H^2_{\bT}(\cX;\bZ),
$$
where $c_1$ and $c_1^\bT$ are the first Chern class and
the $\bT$-equivariant first Chern class, respectively.

\subsection{Torus invariant closed substacks} \label{sec:substack}
 Given $\si\in \Si$, define
$$
I_{\si}':=\{ i\in\{1,\ldots,3+\fp'\}:\rho_i\in \si\},\quad
\quad I_\si = \{1,\ldots, 3+\fp\}\setminus I_{\si}'.
$$
Then $I_{\si}'\subset \{1,\ldots,3+\fp'\}$ and
$\{4+\fp',\ldots, 3+\fp\} \subset I_\si\subset \{1,\ldots,3+\fp\}$.
If $\si\in \Si(d)$ then $|I_{\si}'|=d$ and $|I_\si|=3+\fp-d$.

Let $\tV(\si)\subset U_{\bSi^{\ext}}$ be the closed subvariety defined by the ideal of $\bC[Z_1,\ldots, Z_{3+\fp}]$ generated by
$\{ Z_i=0\mid i \in I'_\si\}$. Then $\cV(\si) := [\tV(\si)/G]$
is a codimension $d$ $\bT$-invariant closed substack of $\cX= [U_{\bSi^{\ext}}/G]$.

The generic stabilizer of $\cV(\si)$ is
$G_\si= \{ g\in G\mid g\cdot z = z \textup{ for all } z\in \tV(\si)\}$,
which is a finite subgroup of $G\cong (\bC^*)^\fp$. If $\tau\subset \si$ then $\cV(\tau)\supset \cV(\si)$ so $G_\tau$ is a subgroup  of $G_\si$. When $\si\in \Si(3)$, we denote $\fp_\si=\cV(\si)$, which is a $\bT$-fixed point.

\subsection{Extended nef, K\"{a}hler, and Mori cones} \label{sec:kahler}

We first introduce some notation. Given a lattice $\Lambda\cong \bZ^m $ and
$\bF=\bQ,\bR, \bC$, let $\Lambda_\bF$ denote the $\bF$ vector space
$\Lambda\otimes_{\bZ}\bF\cong \bF^m$.

Given a maximal cone $\si\in \Si(3)$, we define
$$
\bK_\si^\vee := \bigoplus_{i\in I_\si }\bZ D_i
$$
which is a sublattice of $\bL^\vee$ of finite index, and define
the {\em extended  $\si$-nef cone} to be
$$
\tNef_\si := \sum_{i\in I_\si}\bR_{\geq 0} D_i,
$$
which is a top dimensional cone in $\bL^\vee_\bR$.
The {\em extended nef cone} of
the extended stacky fan $\bSi^{\ext}$ is
$$
\Nef(\bSi^{\ext}):=\bigcap_{\si\in \Si(3)} \tNef_\si.
$$
The {\em extended $\si$-K\"{a}hler cone} $\tC_\si$ is defined to be the interior of
$\tNef_\si$; the {\em extended K\"{a}hler cone} $C(\bSi^{\ext})$ of $\bSi^{\ext}$
is defined to be the interior of the extended nef cone $\Nef(\bSi^{\ext})$.
We have an exact sequence of $\bR$-vector spaces:
$$
0\to \bigoplus_{i=\fp' +4}^{\fp+3} \bR D_i\longrightarrow \bL^\vee_\bR
\stackrel{\kappa}{\longrightarrow} H^2(\cX;\bR)
=H^2(X_\Si;\bR) \to 0
$$
The {\em K\"{a}hler cone} of $\cX$ is
$C(\Si)= \kappa(C(\bSi^{\ext}))\subset H^2(\cX;\bR)$.

For $i=\fp'+4,\dots,\fp+3$, let $\si$ be the smallest cone containing $b_i$. Then
$$
b_i=\sum_{j\in I_\si'}c_j(b_i)b_j,
$$
where $c_j(b_i)\in (0,1)$ and $\sum_{j\in I_\si'}c_j(b_i)=1$. There exists a unique $D_i^\vee \in \bL_\bQ$ such that
$$
\langle D_j, D_i^\vee\rangle =\begin{cases}
1, & j=i,\\
-c_j(b_i), &  j\in I'_\si,\\
0, & j\in I_\si\setminus \{i\}.
\end{cases}
$$

By \cite[Equation (39) and Lemma 3.2]{Iri}, the space $\bL^\vee_\bF$ decomposes as below
\begin{equation}
\label{eqn:split}
\bL^\vee_\bF=\Ker((D^\vee_{\fp'+4},\dots,D^\vee_{\fp+3}):\bL_\bF \to \bF^{\fp-\fp'})\oplus \bigoplus_{j=\fp'+4}^{\fp+3} \bF D_j;
\end{equation}
The first factor is identified with $H^2(\cX;\bF)$ under $\kappa$. The extended K\"ahler cone splits as
\[
C(\bSi^\ext)=C(\Si)\times  (\sum_{j=\fp'+4}^{\fp+3} \bR_{>0} D_j ),
\]
where $C(\Si)\subset H^2(\cX;\bR)\subset \bL^\vee_\bR$. For any $H\in \bL^\vee_\bF$ we denote $\bar H=\kappa(H)\in H^2(\cX;\bF)\subset \bL^\vee_\bF$, and we have
\[
\bar D_i=D_i+\sum_{j=\fp'+4}^{\fp+3} c_i(b_j) D_j.
\]

Let $\bK_\si $ be the dual lattice of $\bK_\si^\vee$; it can be viewed as an additive subgroup of $\bL_\bQ$:
$$
\bK_\si =\{ \beta\in \bL_\bQ \mid \langle D, \beta\rangle \in \bZ \  \forall D\in \bK_\si^\vee \},
$$
where $\langle-, -\rangle$ is the natural pairing between $\bL^\vee_\bQ$ and $\bL_\bQ$.
Define
$$
\bK:= \bigcup_{\si\in \Si(3)} \bK_\si.
$$
Then $\bK$ is a subset (which is not necessarily a subgroup) of $\bL_\bQ$, and $\bL\subset \bK$.

We define the {\em extended $\si$-Mori cone} $\tNE_\si\subset \bL_\bR$ to be the dual cone of
$\tNef_\si\subset \bL_\bR^\vee$:
$$
\tNE_\si=\{ \beta \in \bL_\bR\mid \langle D,\beta\rangle \geq 0 \ \forall D\in \tNef_\si\}.
$$
The {\em extended Mori cone} of the extended stacky fan $\bSi^{\ext}$ is
$$
\NE(\bSi^{\ext}):= \bigcup_{\si\in \Si(3)} \tNE_\si.
$$
We define
$$
\bK_{\eff,\si}:= \bK_\si\cap \tNE_\si,\quad \bK_{\eff}:=
\bK\cap \NE(\bSi^{\ext})= \bigcup_{\si\in \Si(3)} \bK_{\eff,\si}.
$$

\subsection{Anticones and stability}
\label{sec:anticones}

There is an alternative way to define the toric variety $\cX$ (see
\cite[Section 3.1]{Iri}). Given
$D_i\in \bL^\vee\cong \bZ^\fp,$ for $i=1,\dots, \fp+3$, one chooses a
stability vector $\eta\in \bL^\vee_\bR$. Define the anticone
\[
\cA_\eta=\{I\subset \{1,\dots, \fp+3\}: \eta\in \sum_{i\in I} \bR_{\geq
    0} D_i\}.
\]
Then the associated toric orbifold is $[(\bC^{\fp+3}\setminus
\bigcup_{I\notin \cA_\eta} \bC^I)/G]$, where $G=\bL\otimes_\bZ\bC^*$. The
definition is equivalent to the one in Section \ref{sec:stacky}. For
the stacky fan $\bSi^\ext$, one can always choose such $\eta$ -- for
any $\eta$ in $C(\bSi^\ext)$ this construction will produce $\cX$. Then $\mathcal A_\eta$ is the collection of $I_\si$ for all $\si\in \Si$.

\subsection{The inertia stack and the Chen-Ruan orbifold cohomology group} \label{sec:CR}
Given $\si\in \Si$, define $I_\si$ and $I'_\si$ as in Section \ref{sec:substack}, and define
$$
\Boxs:=\{ v\in N: v=\sum_{i\in I'_\si} c_i  b_i, \quad 0\leq c_i <1\}.
$$
If $\tau \subset \sigma$ then
$I'_\tau \subset I'_\si$, so $\mathrm{Box}(\tau)\subset \Boxs$.

Let $\si\in \Si(3)$ be a maximal cone in $\Si$. We
have a short exact sequence of abelian groups
$$
0\to \bK_\si/\bL\to \bL_\bR/\bL\to \bL_\bR/\bK_\si\to 0,
$$
which can be identified with the following short exact sequence of multiplicative abelian groups
$$
1\to G_\si\to G_\bR \to (G/G_\si)_\bR\to 1
$$
where $(G/G_\si)_\bR\cong U(1)^{\fp}$ is the maximal compact subgroup of $(G/G_\si)\cong(\bC^*)^{\fp}$.

Given a real number $x$, we recall some standard notation:
$\lfloor x \rfloor$ is the greatest integer less than or equal to  $x$,
$\lceil x \rceil$ is the least integer greater or equal to $x$,
and $\{ x\} = x-\lfloor x \rfloor$ is the fractional part of $x$.
Define $v: \bK_\si\to N$ by
$$
v(\beta)= \sum_{i=1}^{3+\fp} \lceil \langle D_i,\beta\rangle\rceil b_i.
$$
Then
$$
v(\beta) = \sum_{i\in I'_\si} \{ -\langle D_i,\beta\rangle \}b_i,
$$
so $v(\beta)\in \Boxs$. Indeed, $v$ induces a bijection $\bK_\si/\bL\cong \Boxs$.

For any $\tau \in \Si$ there exists $\si\in \Si(3)$ such that
$\tau \subset \si$. The bijection $G_\si \to \Boxs$ restricts
to a bijection $G_\tau\to \mathrm{Box}(\tau)$.

Define
$$
\BoxS:=\bigcup_{\si\in \Si}\Boxs =\bigcup_{\si\in\Si(3)}\Boxs.
$$
Then there is a bijection $\bK/\bL\to \BoxS$.

Given $v \in \BoxS$ and let $\si$ be the smallest cone containing $v$, define $c_i(v)\in [0,1)\cap \bQ$ by
$$
v= \sum_{i\in I'_\si} c_i(v) b_i.
$$
Suppose that  $k \in G_\si$  corresponds to $v\in \Boxs$ under the bijection $G_\si\cong\Boxs$, then
$k$ acts on $(Z_1,\ldots, Z_{3+\fp})\in \bC^{3+\fp}$ by
$$
k\cdot Z_i = \begin{cases}
Z_i, & i\in I_\si,\\
e^{2\pi\sqrt{-1} c_i(v)} Z_i,& i \in I'_\si.
\end{cases}
$$
Define
$$
\age(k)=\age(v)= \sum_{i\notin I_\si} c_i(v).
$$

Let $IU=\{(z,k)\in U_{\bSi^\ext}\times G\mid k\cdot z = z\}$,
and let $G$ act on $IU$ by $h\cdot(z,k)= (h\cdot z,k)$. The
inertia stack $\cI\cX$ of $\cX$ is defined to be the quotient stack
$$
\cI\cX:= [IU/G].
$$
Note that $(z=(Z_1,\ldots,Z_{3+\fp}), k)\in IU$ if and only if
$$
k\in \bigcup_{\si\in \Si}G_\si \textup{ and }  Z_i=0 \textup{ whenever } \chi_i(k) \neq 1.
$$
So
$$
IU=\bigcup_{v\in \BoxS} U_v,
$$
where
$$
U_v:= \{(Z_1,\ldots, Z_{3+\fp})\in U_{\bSi^{\ext}}: Z_i=0 \textup{ if } c_i(v) \neq 0\}.
$$
The connected components of $\cI\cX$ are
$$
\{ \cX_v:= [U_v/G] : v\in \BoxS\}.
$$

Let $\bF =\bQ, \bR$ or $\bC$.
The Chen-Ruan orbifold cohomology group \cite{CR04} with coefficient $\bF$ is defined to be
$$
H^*_\CR (\cX;\bF)=\bigoplus_{v\in \BoxS}  H^*(\cX_v;\bF)[2\age(v)],
$$
where $[2\age(v)]$ is the degree shift by $2\age(v)$. The Chen-Ruan orbifold cup product is denoted by $\star_\cX$. It is not the component-wise cup product.
Denote $\mathbf 1_v$ to be the unit in $H^*(\cX_v;\bF)$. Then $\mathbf 1_v\in
H^{2\age(v)}_\CR(\cX;\bF)$.

We have
$$
H^2_\CR(\cX;\bF) = H^2(\cX;\bF)\oplus \bigoplus_{i=4+\fp'}^{3+\fp} \bF\one_{b_i}
\cong H^2(\cX;\bF)\oplus \bigoplus_{i=4+\fp'}^{3+\fp} \bF D_i \cong \bL_{\bF}^{\vee},
$$
where we identify $\one_{b_i}=D_i$.

\begin{convention}
In the rest of this paper, when the coefficient is $\bC$, we omit the coefficient.
For example, $H^*_{\CR}(\cX)=H^*_{\CR}(\cX;\bC)$, $H^*(X_\Si)= H^*(X_\Si;\bC)$, etc.
\end{convention}

Let $\fg :=|\Int(P)\cap N'|$
be the number of lattice points in $\Int(P)$, the interior of the polytope $P$, and
let $\fn:=|\partial P\cap N'|$ be the number of lattice points on $\partial P$, the boundary of the polytope $P$. Then
\begin{eqnarray*}
\dim_\bC H^2(X_\Si)   &=& \fp'= |\Si(1)|-3\\
\dim_\bC H^2_\CR(\cX) &=& \fp= |P \cap N'|-3 = \fg+\fn-3,\\
\dim_\bC H^4_\CR(\cX) &=& |\Int(P)\cap N'| = \fg,\\
\dim_\bC H^*_\CR(\cX)  &=& 2\Area(P) = 1+ \fp +\fg = 2\fg-2+\fn.
\end{eqnarray*}

Let $H^2_{\CR,\mathrm{c}}(\cX)$ be the subspace of $H^2_{\CR}(\cX)$ generated by
$$
\{ \bar{D}_i: i\in \{1,\ldots, 3+\fp'\},\cV(\rho_i)\textup{ is proper}  \}\cup
\{ \one_{b_i}: i\in \{4+\fp',\ldots, 3+\fp\}, \cX_{b_i} \textup{ is proper} \}
$$
Then $\dim_\bC H^2_{\CR,\mathrm{c}}(\cX)=\fg$, and there is a perfect pairing
\begin{equation}\label{eqn:non-equivariant-pairing}
H^2_{\CR,\mathrm{c}}(\cX)\times H^4_{\CR}(\cX)\longrightarrow \bC.
\end{equation}

\begin{example}[Example \ref{exp:polytope}, continued]
  \label{exp:number-of-parameters}
  Let $\cX$ be the toric Calabi-Yau $3$-orbifold defined in Example \ref{exp:polytope}. We have
  \begin{align*}
    &\fp=2,\ \fp'=1,\ \fg=1,\ \fn=4,\ \dim_\bC H^*_\CR(\cX)=4,\\
    &D_1=-2D_4,\ D_4=2D_3+D_5=2D_2+D_5,\ \bar D_1=D_1,\ \bar D_2= D_2+\frac{1}{2}D_5,\ \bar D_3=D_3+\frac{1}{2}D_5,\ \bar D_4=D_4,\ \bar D_5=0,\\
    &\text{K\"ahler cone}\ C(\Sigma)=\bR_{>0} \bar D_4,\quad \text{extended K\"ahler cone}\ C(\bSi^\ext)=\bR_{>0} D_4\times \bR_{>0}D_5.
  \end{align*}
\end{example}

\subsection{Equivariant Chen-Ruan cohomology}\label{sec:equivariantCR}

Let $\RT:= H_{\bT}^*(\mathrm{pt})= H^*(B\bT)$, and let $\ST$ be the fractional field of $\RT$.
Then
$$
\RT =\bC[\su_1,\su_2,\su_3],\quad \ST=\bC(\su_1,\su_2,\su_3).
$$

As a graded $\bC$ vector space and an $R_\bT$-module, the $\bT$-equivariant Chen-Ruan orbifold cohomology
with coefficient $\bC$ is defined to be
$$
H^*_{\CR,\bT} (\cX)=\bigoplus_{v\in \BoxS}  H^*_{\bT}(\cX_v)[2\age(v)].
$$
By slight abuse of notation, the Chen-Ruan orbifold cup product in the equivariant setting is also denoted by $\star_\cX$.

Given $\si\in \Si(3)$, let $\cX_\si=[\bC^3/G_\si]$ be the affine toric Calabi-Yau 3-orbifold defined
by the cone $\si$; its coarse moduli is the affine simplicial toric variety $X_\si$
defined by $\si$: $X_\si=\mathrm{Spec}\bC[\si^\vee\cap M]\cong \bC^3/G_\si$.
The inertia stack of
$\cX_\si$ is
$$
\cI \cX_\si =\bigcup_{h\in G_\si} \cX_h,
\quad
\textup{where } \cX_h=[(\bC^3)^h/G].
$$
As a graded vector space over $\bC$,
$$
H^*_\CR(\cX_\si;\bC) = \bigoplus_{h\in G_\si} H^*(\cX_h;\bC)[2\age(h)] =\bigoplus_{h\in G_\si}\bC \one_h,
$$
where $\one_h$ is unit of the ring $H^*(\cX_h;\bC)\cong \bC$,  so $\deg(\one_h)=2\age(h)$.  Let $\sw_{1,\si},\sw_{2,\si},\sw_{3,\si}$ be the weights of the $\bT$ action along the 3 coordinate axes of $\cX_\si\cong [\bC^3/G_\si]$.
Then the $\bT$-equivariant Poincar\'{e} pairing is given by
$$
\langle \one_h,\one_{h'}\rangle_{\cX_\si} = \frac{\delta_{hh',1}}{\displaystyle{|G_\si| \prod_{i=1}^3 \sw_{i,\si}^{\delta_{c_i(h),0}} } },
$$
and the $\bT$-equivariant
orbifold cup product is given by
$$
\one_h \star_{\chi_\si} \one_{h'} =\Big(\prod_{i=1}^3 \sw_{i,\si}^{c_i(h)+c_i(h')-c_i(hh')}\Big) \one_{hh'}
$$
Define
$$
\bar{\one}_h:=\frac{\one_h}{\prod_{i=1}^3 \sw_{i,\si}^{c_i(h)}} \in H^*_{\bT,\CR}(\cX_\si;\bC)\otimes_{\RT}\bST,
$$
where $\bST$ is the minimal extension of $\ST$ which contains $\{\sw_{i,\si}^{c_i(h)}: i\in \{1,2,3\}, h\in G_\si, \si\in \Si(3) \}$.
Then
$$
\langle \bar{\one}_h,\bar{\one}_{h'}\rangle_{\cX_\si} = \frac{\delta_{hh',1}}{|G_\si|\prod_{i=1}^3 \sw_{i,\si} },\quad
\bar{\one}_h \star_{\cX_\si}\bar{\one}_{h'} =\bar{\one}_{hh'}.
$$
For any $\gamma\in G_\si^*$, define
$$
\bar{\phi}_\gamma:=\frac{1}{|G_\si|}\sum_{h\in G_\si}\chi_\gamma(h^{-1})\bar{\one}_h.
$$
Then
$$
\langle \bar{\phi}_\gamma,\bar{\phi}_{\gamma'}\rangle_{\cX_\si} = \frac{\delta_{\gamma \gamma'}}{|G_\si|^2 \prod_{i=1}^3 \sw_{i,\si} },\quad
\bar{\phi}_\gamma\star_{\cX_\si} \bar{\phi}_{\gamma'} = \delta_{\gamma\gamma'}\bar{\phi}_{\gamma}.
$$
So
$\{\bar{\phi}_\gamma:\gamma\in G_\si^*\}$
is a canonical basis of the semisimple Frobenius algebra
$$
\big(H^*_{\CR,\bT}(\cX_\si)\otimes_\RT \bST, \star_{\cX_\si}, \langle \  ,\  \rangle_{\cX_\si}\big)
$$
over the field $\bST$.

The Frobenius algebra  $H^*_{\CR,\bT}(\cX)\otimes_{\RT}\bST$,
equipped with the $\bT$-equivariant orbifold cup product and the $\bT$-equivariant Poincar\'{e} pairing,
is isomorphic to a direct sum of Frobenius algebras:
\begin{equation}\label{eqn:direct-sum}
\bigoplus_{\si\in \Si(3)} \iota_\si^*: H^*_{\CR,\bT}(\cX;\bC)\otimes_\RT\bST
\stackrel{\cong}{\longrightarrow}  \bigoplus_{\si\in \Si(3)} H^*_{\CR,\bT}(\cX_\si;\bC)\otimes_{\RT} \bST,
\end{equation}
where  $\iota_\si^*: H^*_{\CR,\bT}(\cX)\otimes_{\RT} \bST \to H^*_{\CR,\bT}(\cX_\si)\otimes_{\RT} \bST$ is induced
by the $\bT$-equivariant open embedding $\iota_\si: \cX_\si\hookrightarrow \cX$.
There exists a unique $\phi_{\si,\gamma}\in H^*_{\CR,\bT}(\cX)\otimes_{\RT} \bST$
such that $\phi_{\si,\gamma}|_{\cX_\sigma}= \bar{\phi}_\gamma$
and $\phi_{\si,\gamma}|_{\fp_{\si'}}=0$ for $\si'\in \Si(3)$, $\si'\neq \si$, where $\fp_{\si'}$ is the $\bT$-fixed point corresponding to $\si'$.
Define $I_\Si:=\{(\si,\gamma): \si\in \Si(3),\gamma\in G_\si^*\}$. Then
$$
\{\phi_{\si,\gamma}: (\si,\gamma)\in I_\Si\}
$$
is a canonical basis of the semisimple $\bST$-algebra  $H^*_{\CR,\bT}(\cX;\bC)\otimes_{\RT}\bST$:
$$
\phi_{\si,\gamma}\star_{\cX} \phi_{\si',\gamma'}
=\delta_{\si,\si'}\delta_{\gamma,\gamma'} \phi_{\si,\gamma}.
$$
We have
$$
\sum_{(\si,\gamma)\in I_\Si}\phi_{\si,\gamma} =1.
$$
The $\bT$-equivariant Poincar\'{e} pairing is given by
$$
( \phi_{\si,\gamma},\phi_{\si',\gamma'})_{\cX,\bT}
=\frac{\delta_{\si,\si'}\delta_{\gamma,\gamma'}}{\Delta^{\si,\gamma}},\quad
\Delta^{\si,\gamma}=|G_\si|^2 \be(\si),
$$
where $\be(\si)\in H^6_{\bT}(\fp_\si)=H^6(B\bT)$ is the $\bT$-equivariant
Euler class of the tangent space $T_{\fp_\si}\cX$ (viewed as a $\bT$-equivariant vector bundle over $\fp_\si\cong \cB G_\si$).

In the rest of this paper, we sometimes use the bold letter $\bsi$ for the pair $(\si,\gamma)$ for simplicity.
Define
$$
\hat{\phi}_{\bsi}=\sqrt{\Delta^{\bsi}}\phi_{\bsi},\quad \bsi\in I_\Si.
$$
(By a finite field extension, we may assume $\sqrt{\Delta^{\bsi}}\in \bST$ for all $\bsi\in I_\Si$.) Then
$$
\hat{\phi}_{\bsi}\star_{\cX}\hat{\phi}_{\bsi'} = \delta_{\bsi,\bsi'} \sqrt{\Delta^{\bsi}}\hat{\phi}_{\bsi},\quad
(\hat{\phi}_{\bsi}, \hat{\phi}_{\bsi'})_{\cX,\bT}= \delta_{\bsi,\bsi'}.
$$
We call $\{\hat{\phi}_\bsi:\bsi\in I_\Si\}$ the {\em classical normalized canonical basis}. It is
a normalized canonical basis of the $\bT$-equivariant Chen-Ruan orbifold cohomology
ring $H^*_{\CR,\bT}(\cX)\otimes_{\RT}{\bST}$.

\subsection{The symplectic quotient and the toric graph}\label{sec:toric-graph}
Let $\{\ep_a: a=1,\ldots, \fp\}$ be
a $\bZ$-basis of the lattice $\bL$ and let
$\{\ep_a^*: a=1,\ldots, \fp\}$ be the dual $\bZ$-basis
of the dual lattice $\bL^\vee$. Then $\psi:\bL\to \tN$ and
$\psi^\vee:\tM\to \bL^\vee$ are  given by
$$
\psi(\ep_a)= \sum_{i=1}^{3+\fp} l^{(a)}_i\tb_i,\quad
\psi^\vee(D_i^{\bT})=\sum_{a=1}^\fp l^{(a)}_i \ep_a^*.
$$
for some $l^{(a)}_i \in \bZ$, where $1\leq a\leq \fp$ and $1\leq i\leq 3+\fp$.
The $\fp$ vectors $l^{(a)}:= (l^{(a)}_1,\ldots, l^{(a)}_{3+\fp})\in \bZ^{3+\fp}$ are
known as the charge vectors in the physics literature.

Let $G_\bR \cong U(1)^\fp$ be the maximal compact torus of  $G\cong (\bC^*)^\fp$.
Then $\bL^\vee_\bR$ can be canonically identified with the dual of the Lie algebra of $G_\bR$.
The $G$-action on $\bC^{3+\fp}$ restricts to a Hamiltonian  $G_\bR$-action on
the symplectic manifold $(\bC^{3+\fp},\sqrt{-1}\sum_{i=1}^{3+\fp} dZ_i\wedge d\bar{Z}_i)$. A moment
map of this Hamiltonian $G_\bR$-action is given by
$$
\tmu:\bC^{3+\fp} \longrightarrow \bL_\bR^\vee,\quad
\tmu(Z_1,\ldots, Z_{3+\fp})= \sum_{a=1}^\fp \Big(\sum_{i=1}^{3+\fp} l_i^{(a)} |Z_i|^2\Big) \ep_a^*.
$$
Given a point $\eta$ in the extended K\"{a}hler cone $C(\Si^{\ext}) \subset \bL_\bR^\vee$,
the symplectic quotient $[\tmu^{-1}(\eta)/G_{\bR}]$ is a K\"{a}hler orbifold which is isomorphic
to $\cX$ as a complex orbifold (c.f. Section \ref{sec:anticones}). The symplectic structure $\omega(\eta)$ depends on $\eta$. The map
$$
\kappa|_{C(\bSi^\ext)}: C(\Si^\ext)\longrightarrow  C(\Si)\subset H^2(X_\Si;\bR)
$$
can be identified with $\bk\mapsto [\omega(\eta)]$, where $[\omega(\eta)]$ is the K\"{a}hler
class of the K\"{a}hler form $\omega(\eta)$.

Let $\bT_\bR\cong U(1)^3$ (resp. $\bT'_\bR\cong U(1)^2$) be the maximal compact torus
of $\bT\cong (\bC^*)^3$ (resp. $\bT'\cong (\bC^*)^2$).
Then $M_\bR$ (resp. $M_\bR'$) is canonically identified with the dual of the Lie algebra of $\bT_\bR$
(resp. $\bT'_\bR$). The $\bT$-action on $\cX$ restricts to a Hamiltonian $\bT_\bR$-action on the K\"{a}hler orbifold $(\cX,\omega(\eta))$. The map $\kappa(\eta)$ determines
a moment map $\mu_{\bT_\bR}: \cX \longrightarrow M_\bR$ up to translation by a vector in $M_\bR$.
The image $\mu_{\bT_\bR}(\cX)$ is a convex polyhedron.
The moment map $\mu_{\bT'_\bR}:\cX \longrightarrow M'_\bR$ is the composition $\pi\circ \mu_{T_\bR}$,
where $\pi:M_\bR\cong \bR^3 \to M'_{\bR}\cong \bR^2$ is the projection.
The map $\mu_{\bT'_\bR}$ is surjective. Let $\cX^1 \subset \cX$ be the union of
0-dimensional and 1-dimensional $\bT$-orbits in $\cX$. The {\em toric graph} of $(\cX,\omega(\eta))$ is defined by
$\Gamma:= \mu_{T_\bR'}(\cX^1)\subset M'_\bR \cong \bR^2$. It is determined by $\kappa(\eta)$ up to translation
by a vector in $M'_\bR$. The vertices (resp. edges) of $\Gamma$ are in one-to-one correspondence
to $3$-dimensional (resp. $2$-dimensional) cones in $\Si$.

\subsection{Aganagic-Vafa branes}
\label{sec:av-branes}

An Aganagic-Vafa brane in $\cX=[\tilde{\mu}^{-1}(\eta)/G_\bR]$ is a Lagrangian suborbifold
of the form $\cL=[\tL/G_\bR]$, where
$$
\tL=\{ (Z_1,\ldots, Z_{3+\fp})\in \tmu^{-1}(\eta): \sum_{i=1}^{3+\fp} \hat{l}^1_i |Z_i|^2 = c_1,
\quad \sum_{i=1}^{3+\fp}\hat{l}^2_i|Z_i|^2=c_2,\quad \arg(\prod_{i=1}^{3+\fp} Z_i)= c_3\}.
$$
for some $\hat{l}^1, \hat{l}^2\in \bZ^{3+\fp}$ satisfying
$\sum_{i=1}^{3+\fp} \hat{l}^1_i =\sum_{i=1}^{3+\fp} \hat{l}^2_i =0$ and $c_1,c_2,c_3\in \bR$.
The constants $c_1, c_2$ are chosen such that $\mu_{\bT_\bR}(\cL)$ is a ray ending
at a point in the interior of an edge of the moment polytope. Then $\mu_{\bT'_\bR}(\cL)$
is a point in $\Gamma$ which is not a vertex.

Given a 2-dimensional cone $\tau \in \Si(2)$ such that $I'_\tau=\{i,j\}$,
where $1\leq i<j \leq 3+\fp'$, let $\fl_\tau:=\mathcal V(\tau)$ be defined by
$Z_i=Z_j=0$, and let $\ell_\tau$ be the coarse moduli space of
$\fl_\tau$. There are two cases:
\begin{enumerate}
 \item $\tau$ is the intersection of two 3-dimensional cones
and $\ell_\tau \cong \bP^1$.
\item There is a unique 3-dimensional cone $\si$ containing $\tau$ and $\ell_\tau\cong \bC$.
\end{enumerate}

An Aganagic-Vafa brane $\cL$ intersects a unique 1-dimensional $\bT$-orbit closure
$\fl_\tau$, where $\tau\in \Si(2)$. We say $\cL$ is an inner (resp. outer) brane
if $\ell_\tau\cong \bP^1$ (resp. $\ell_\tau \cong \bC$). In this paper we will consider
a fixed $\tau_0\in \Si(2)$ such that $\ell_{\tau_0}\cong \bC$, and consider
an Aganagic-Vafa (outer) brane intersecting $\fl_{\tau_0}$. Let $\sigma_0$ be the unique
3-dimensional cone containing $\tau_0$. By permuting $b_1,\ldots, b_{3+\fp'}$ we may
assume $\si_0$ is spanned by $b_1,b_2,b_3$ and $\tau_0$ is spanned by
$b_2, b_3$. We have a short exact sequence of finite abelian groups
$$
1\to G_{\tau_0}\cong \bmu_{\fm} \longrightarrow G_{\si_0} \longrightarrow  \bmu_{\fr}\to 1,
$$
where $\fm$ and $\fr$ are positive integers and $\bmu_{\fm}$ and $\bmu_{\fr}$ are finite cyclic groups of orders $\fm$ and $\fr$ respectively.
We say $\fl_{\tau_0}$ is an effective leg (resp. a gerby leg)
if $\fm=1$ (resp. $\fm>1$).
By choosing suitable $\bZ$-basis $(e_1,e_2,e_3)$ of $N$ we may
assume
$$
b_1 = \fr e_1-\fs e_2+e_3,\quad b_2 = \fm e_2+e_3,\quad b_3 = e_3,
$$
where $\fs\in \{0,1\ldots,\fr-1\}$.

\begin{figure}[h]
\begin{center}
\psfrag{L}{\tiny $\mathcal L$}
\psfrag{lt0}{\tiny $\fl_{\tau_0}$}
\psfrag{ls0}{\tiny $\fl_{\si_0}$}
\includegraphics[scale=0.3]{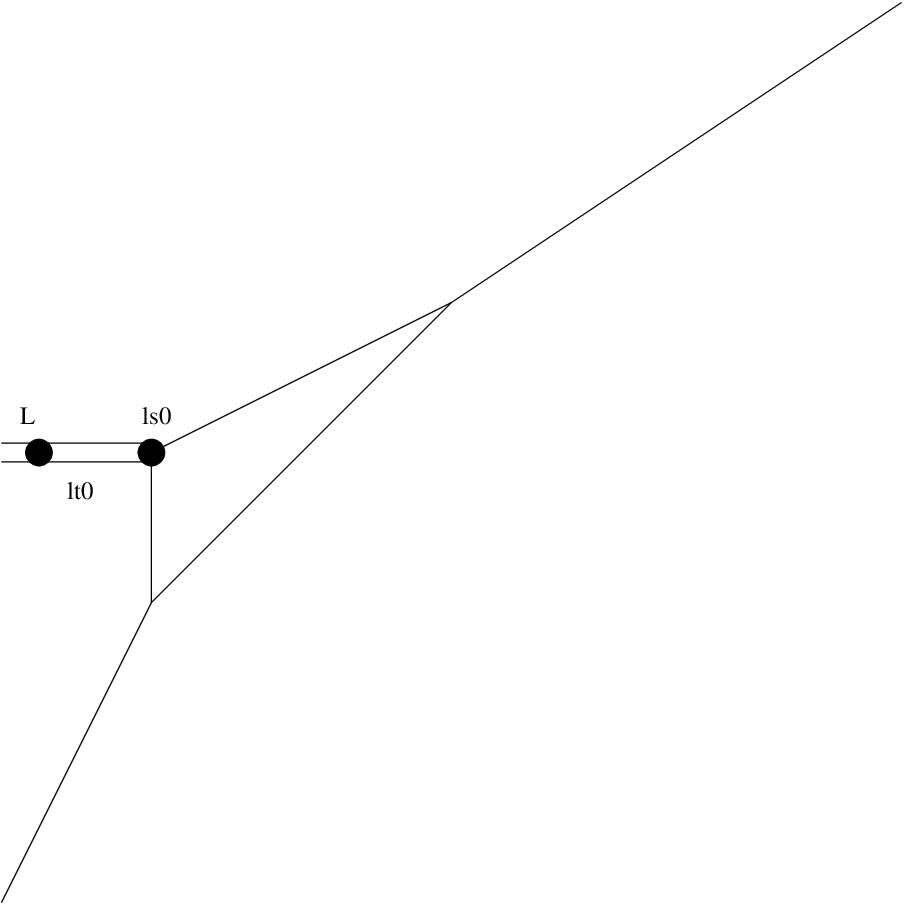}
\caption{The toric graph for $\cX$ in Example \ref{exp:polytope}. The gerby leg $\fl_{\tau_0}$'s image is a line, but we draw a double line to denote it is gerby.}
\label{fig:graph}
\end{center}
\end{figure}

\begin{example}[Example \ref{exp:number-of-parameters}, continued]
  \label{exp:graph}
  Given the choice of $\si_i$ and $\tau_0$ as in Figure \ref{fig:defining-polytope}, the toric graph for $\cX$ and the phase of the Aganagic-Vafa brane $\cL$ is in Figure \ref{fig:graph}. We have
  \begin{align*}
    &G_{\tau_0}\cong \bmu_2,\ \fm=2,\ G_{\si_0}\cong \bmu_2,\ \fr=1,\\
    & \fl_{\tau_0}\cong \bC\times \cB \bmu_2,\ \fl_{\si_0}=[\mathrm{pt}/\bmu_2],\\
    &G_{\si_1}=G_{\si_2}=\{1\}.
  \end{align*}
\end{example}

\section{A-model Topological Strings}
\label{sec:A-string}

In Section \ref{sec:TGW}-\ref{sec:Agraph}, we consider
the 3-dimensional torus $\bT$ and the big phase space.
In Section \ref{sec:I-J} we specialize to the small  phase space.
After that, we further specialize to the Calabi-Yau torus $\bT'$.

\subsection{Equivariant Gromov-Witten invariants}\label{sec:TGW}
Given nonnegative integers $g$, $n$ and an effective curve class
$d\in H_2(X_\Si;\bZ)$, let $\Mbar_{g,n}(\cX,d)$ be the
moduli stack of genus $g$, $n$-pointed, degree $d$ twisted
stable maps to $\cX$ (\cite{AGV02, AGV08}, \cite[Section 2.4]{Ts10}).
Let $\ev_i:\Mbar_{g,n}(\cX,d)\to \cI\cX$ be the evaluation
at the $i$-th marked point. The $\bT$-action on $\cX$ induces
$\bT$-actions on $\Mbar_{g,n}(\cX,d)$ and on the inertia stack $\cI\cX$, and
the evaluation map $\ev_i$ is $\bT$-equivariant.

Let $\Mbar_{g,n}(X_\Si,d)$ be the moduli stack of
genus $g$, $n$-pointed, degree $d$ stable maps to the coarse
moduli $X_\Si$ of $\cX$. Let
$\pi:\Mbar_{g,n+1}(X_\Si,d)\to \Mbar_{g,n}(X_\Si,d)$
be the universal curve, and let $\omega_\pi$
be the relative dualizing sheaf. Let $s_j:\Mbar_{g,n}(X_\Si,d)\to
\Mbar_{g,n+1}(X_\Si,d)$ be the section associated to
the $j$-th marked point. Then
$$
\bL_j :=s_j^*\omega_\pi
$$
is the line bundle over $\Mbar_{g,n}(X_\Si,d)$ formed
by the cotangent line at the $j$-th marked point.
The descendant classes on $\Mbar_{g,n}(X_\Si,d)$ are defined by
$$
\psi_j:= c_1(\bL_j)\in H^2(\Mbar_{g,n}(X_\Si,d);\bQ), \quad j=1,\ldots,n.
$$
The $\bT$-action on $X_\Si$ induces a $\bT$-action on
$\Mbar_{g,n}(X_\Si,d)$, and we choose a $\bT$-equivariant
lift $\psi_j^{\bT}\in H^2_{\bT}(\Mbar_{g.n}(X_\Si,d);\bQ)$
of $\psi_j\in H^2(\Mbar_{g,n}(X_\Si,d);\bQ)$ as in \cite[Section 5.1]{Liu}.

The map $p:\Mbar_{g,n}(\cX,d)\to \Mbar_{g,n}(X_\Si,d)$ is
$\bT$-equivariant. Following \cite[Section 2.5.1]{Ts10}, we define
$$
\hat{\psi}_j := p^*\psi_j  \in H^2(\Mbar_{g,n}(\cX,d);\bQ)
$$
to be the pullback of $\psi_j \in H^2(\Mbar_{g,n}(X_\Si,d);\bQ)$.
(Note that $\hat{\psi}_j$ is denoted  $\bar{\psi}_j$ in \cite{CCIT}
and $\psi_j$ in \cite{Zo}.) Then
$$
\hat{\psi}_j^{\bT}:=p^*\psi_j^{\bT} \in H^2_{\bT}(\Mbar_{g,n}(\cX,d);\bQ)
$$
is a $\bT$-equivariant lift of $\hat{\psi}_j$.

Since $X_\Si$ is not projective, the moduli stack $\Mbar_{g,n}(\cX,d)$ is
not proper in general, but the $\bT$-fixed substack
$\Mbar_{g,n}(\cX,d)^{\bT}$ is. Given $a_1,\ldots, a_n\in \bZ_{\geq 0}$ and
$\gamma_1,\ldots,\gamma_n\in H^*_{\CR,\bT}(\cX;\bQ)$, we define
$\bT$-equivariant genus $g$ degree $d$ Gromov-Witten invariants
of $\cX$ by
$$
\langle \gamma_1\hat{\psi}^{a_1}, \cdots,\gamma_n\hat{\psi}^{a_n}\rangle^{\cX,\bT}_{g,n,d}
= \int_{[\Mbar_{g,n}(\cX,d)^{\bT}]^{w,\bT}} \frac{\iota^*\big(\prod_{j=1}^n \ev_j^*(\gamma_j)(\hat{\psi}_j^{\bT})^{a_j}\big)}{e_{\bT}(N^\vir)}
\in \ST=\bQ(\su_1,\su_2,\su_3).
$$
where the weighted virtual fundamental class $[\Mbar_{g,n}(\cX,d)^{\bT}]^{w,\bT}$ \cite{AGV02, AGV08} (resp.
the virtual normal bundle $N^\vir$ of $\Mbar_{g,n}(\cX,d)^{\bT}$ in $\Mbar_{g,n}(\cX,d)$) is
defined by the fixed (resp. moving) part of the restriction to $\Mbar_{g,n}(\cX,d)^{\bT}$ of the $\bT$-equivariant
perfect obstruction theory on $\Mbar_{g,n}(\cX,d)$ \cite{GP99},
and $\iota^*:H^*_{\bT}(\Mbar_{g,n}(\cX,d);\bQ)\rightarrow H^*_{\bT}(\Mbar_{g,n}(\cX,d)^{\bT};\bQ)$ is induced
by the inclusion map $\iota: \Mbar_{g,n}(\cX,d)^{\bT}\hookrightarrow \Mbar_{g,n}(\cX,d)$.
More generally, if we insert $\gamma_1,\ldots,\gamma_n\in H^*_{\CR,\bT}(\cX)\otimes_{\RT}\bST$
then we obtain
$$
\langle \gamma_1 \hat{\psi}^{a_1}, \cdots , \gamma_n \hat{\psi}^{a_n} \rangle^{\cX,\bT}_{g,n,d} \in \bST.
$$

\subsection{Generating functions}\label{sec:generating-functions}
Let $\NE(\Si)\subset H_2(\cX;\bR)=H_2(X_\Si;\bR)$ be the Mori cone
generated by effective curve classes in $X_\Si$ (see Section \ref{sec:kahler}).
Let $E(\cX)$ denote the semigroup $\NE(\Si)\cap H_2(X_\Si;\bZ)$.
The Novikov ring $\nov$ of $\cX$ is defined to be the completion of the
semigroup ring $\bC[E(\cX)]$.
$$
\nov:=\widehat{ \bC[E(\cX)]}
= \{ \sum_{d\in E(\cX)} c_d Q^d:c_d\in \bC\}.
$$

Given $a_1,\ldots, a_n\in \bZ_{\geq 0}$,
$\gamma_1,\ldots,\gamma_n\in H^*_{\CR,\bT}(\cX)\otimes_\RT \bST$, we define
the following double correlator:
\begin{eqnarray*}
\llangle \gamma_1\hat{\psi}^{a_1},  \cdots, \gamma_n\hat{\psi}^{a_n} \rrangle^{\cX,\bT}_{g,n}
:=\sum_{m=0}^\infty \sum_{d\in E(\cX)}\frac{Q^d}{m!}\langle
\gamma_1\hat{\psi}^{a_1}, \cdots, \gamma_n\hat{\psi}^{a_n}, t^m \rangle^{\cX,\bT}_{g,n+m,d}
\end{eqnarray*}
where $Q^d\in \bC[E(\cX)]\subset \nov$ is the Novikov variable corresponding to
$d\in E(\cX)$, and $t\in H^*_{\CR,\bT}(\cX)\otimes_{R_\bT}\bST$.

We introduce two types of generating functions of genus $g$, $n$-point $\bT$-equivariant
Gromov-Witten invariants of $\cX$.

\begin{enumerate}
\item For $j=1,\ldots,n$, introduce formal variables
$$
\bu_j =\bu_j(z)= \sum_{a\geq 0}(u_j)_a z^a
$$
where $(u_j)_a \in H^*_{\CR,\bT}(\cX)\otimes_\RT\bST$.
Define
$$
\llangle \bu_1,\ldots, \bu_n \rrangle_{g,n}^{\cX,\bT} =
\llangle \bu_1(\hat{\psi}),\ldots, \bu_n(\hat{\psi}) \rrangle_{g,n}^{\cX,\bT}
=\sum_{a_1,\ldots,a_n\geq 0}
\llangle (u_1)_{a_1}\hat{\psi}^{a_1}, \cdots, (u_n)_{a_n}\hat{\psi}^{a_n}\rrangle_{g,n}^{\cX,\bT}.
$$
\item Let $z_1,\ldots,z_n$ be formal variables. Given $\gamma_1,\ldots,\gamma_n
\in H^*_{\CR,\bT}(\cX)\otimes_\RT \bST$, define
$$
\llangle \frac{\gamma_1}{z_1-\hat{\psi}_1},\ldots, \frac{\gamma_n}{z_n-\hat{\psi}_n}\rrangle^{\cX,\bT}_{g,n}
=\sum_{a_1,\ldots,a_n\in \bZ_{\geq 0}}
\llangle \gamma_1 \hat{\psi}^{a_1} ,\ldots, \gamma_n \hat{\psi}^{a_n} \rrangle^{\cX,\bT}_{g,n}\prod_{j=1}^n z_j^{-a_j-1}.
$$
\end{enumerate}
The above two generating functions are related by
\begin{equation}\label{eqn:gamma-u}
\llangle \frac{\gamma_1}{z_1-\hat{\psi}_1},\ldots, \frac{\gamma_n}{z_n-\hat{\psi}_n}\rrangle^{\cX,\bT}_{g,n}
=\llangle \bu_1,\ldots, \bu_n \rrangle_{g,n}^{\cX,\bT}\Big|_{\bu_j(z)=\frac{\gamma_j}{z_j-z}}.
\end{equation}

\subsection{The equivariant big quantum cohomology}\label{sec:bigQH}
Let
$$
\chi= \dim_{\bC}H^*_{\CR}(\cX) =\dim_{\bST} H^*_{\CR,\bT}(\cX;\bST).
$$
We choose an $\bST$-basis of $H^*_{\CR,\bT}(\cX;\bST)$ $\{ T_i: i=0,1,\ldots,\chi-1\}$
such that
$$
T_0=1,\quad T_a=\bar{D}_{3+a}^{\bT} \textup{ for } a=1,\ldots,\fp',\quad
T_a=\one_{b_{3+a}}\textup{  for }a=\fp'+1,\ldots, \fp,
$$
and for $i=\fp+1,\ldots,\chi-1$, $T_i$ is of the form $T_aT_b$ for some $a,b\in \{1,\ldots, \fp\}$.
Write $t=\sum_{a=0}^{\chi-1}\tau_a T_a$, and let $\tau'=(\tau_1,\ldots, \tau_{\fp'})$,
$\tau''=(\tau_0, \tau_{\fp'+1},\ldots, \tau_{\chi-1})$.
By the divisor equation,
$$
\llangle T_i,T_j,T_k\rrangle^{\cX,\bT}_{0,3} \in \ST[\![ \tQ, \tau'']\!],\quad
\llangle \hat{\phi}_{\bsi}, \hat{\phi}_{\bsi'}, \hat{\phi}_{\bsi''}\rrangle^{\cX,\bT}_{0,3}\in \bSTQ,
$$
where $\tQ^d = Q^d \exp(\sum_{a=1}^{\fp'}\tau_a \langle T_a, d\rangle)$. Let $S:= \bSTQ$. Given
$a,b\in H_{\CR,\bT}^*(\cX;\bST)$, define the \emph{quantum product}
$$
a\star_t b:= \sum_{\bsi\in I_\Si} \llangle a, b, \hat{\phi}_\bsi\rrangle \hat{\phi}_\bsi
\in H_{\CR,\bT}^*(\cX;\bST)\otimes_{\bST} S.
$$
Then $A:= H_{\CR,\bT}^*(\cX;\bST)\otimes_{\bST}S$ is a free $S$-module of rank $\chi$, and
$(A,*_t)$ is a commutative, associative algebra over $S$.  Let $I\subset S$
be the ideal generated by $\tQ,\tau''$, and define
$$
S_n:= S/I^n,\quad A_n:= A\otimes_{S}S_n
$$
for $n\in \bZ_{\geq 0}$. Then $A_n$ is a free $S_n$-module of rank $\chi$, and the ring structure
$*_t$ on $A$ induces a ring structure $*_{\underline{n}}$ on $A_n$. In particular,
$$
S_1=\bST,\quad A_1=H^*_{\CR,\bT}(\cX;\bST),
$$
and $*_{\underline{1}} = *_\cX$ is the orbifold cup product. So
$$
\{ \phi_{\bsi}^{(1)}:=\phi_{\bsi}: \bsi\in I_\Si\}
$$
is an idempotent basis of $(A_1,\star_{\underline{1}})$.
For $n\geq 1$, let  $\{ \phi_{\bsi}^{(n+1)}:\bsi\in I_\Si\}$
be the unique idempotent basis of $(A_{n+1},\star_{\underline{n+1}})$
which is the lift of the idempotent basis $\{ \phi_{\bsi}^{(n)}:\bsi\in I_\Si\}$
of $(A_n,\star_{\underline{n}})$ \cite[Lemma 16]{LP}. Then
$$
\{ \phi_{\bsi}(t):=\lim \phi_{\bsi}^{(n)}: \bsi\in I_\Si\}
$$
is an idempotent basis of $(A,\star_t)$. The ring $(A,\star_t)$ is called the \emph{equivariant big quantum cohomology ring}.

Set
$$
\novT:= \bST\otimes_{\bC}\nov =\bST[\![ E(\cX)]\!].
$$
Then $H:=H^*_{\CR,\bT}(\cX;\novT)$ is a free $\novT$-module of rank $\chi$.
Any point $t\in H$ can be written as  $t=\sum_{\bsi\in I_\Si}t^{\bsi} \hat{\phi}_{\bsi}$. We have
$$
H=\mathrm{Spec}(\novT [ t^{\bsi}:\bsi\in I_\Si]).
$$
Let $\hat{H}$ be the formal completion of $H$ along the origin:
$$
\hat{H} :=\mathrm{Spec}(\novT[\![ t^{\bsi}:\bsi\in I_\Si ]\!]).
$$
Let $\cO_{\hat{H}}$ be the structure sheaf on $\hat{H}$, and let $\cT_{\hat{H}}$ be the tangent sheaf on
$\hat{H}$.
Then $\cT_{\hat{H}}$ is a sheaf of free $\cO_{\hat{H}}$-modules of rank $\chi$.
Given an open set in $\hat{H}$,
$$
\cT_{\hat{H}}(U)  \cong \bigoplus_{\bsi\in I_\Si}\cO_{\hat{H}}(U) \frac{\partial}{\partial t^{\bsi}}.
$$
The big quantum product and the $\bT$-equivariant Poincar\'{e} pairing defines the structure of a formal
Frobenius manifold on $\hat{H}$:
$$
\frac{\partial}{\partial t^{\bsi}} \star_t \frac{\partial}{\partial t^{\bsi'}}
=\sum_{\brho\in I_\Si} \llangle \hat{\phi}_{\bsi},\hat{\phi}_{\bsi'},\hat{\phi}_{\brho}\rrangle_{0,3}^{\cX,\bT}
\frac{\partial}{\partial t^{\brho}}
\in \Gamma(\hat{H}, \cT_{\hat{H}}).
$$
$$
( \frac{\partial}{\partial t^{\bsi}},\frac{\partial}{\partial t^{\bsi'}})_{\cX,\bT} =\delta_{\bsi,\bsi'}.
$$

\subsection{The A-model canonical coordinates and the $\Psi$-matrix}\label{sec:A-canonical}
The canonical coordinates $\{ u^{\bsi}=u^{\bsi}(t):\bsi\in I_\Si\}$ on the formal Frobenius
manifold $\hat{H}$ are characterized by
\begin{equation}\label{eqn:partial-u}
\frac{\partial}{\partial u^{\bsi}} = \phi_{\bsi}(t).
\end{equation}
up to additive constants in $\novT$. We choose canonical coordinates
such that they lie in $\bST[\tau'][\![\tQ,\tau'']\!]$ and vanish when $Q=0$,
$\tau'=\tau''=0$. Then $u^{\bsi}-\sqrt{\Delta^{\bsi}}t^{\bsi} \in \bST[\tau'][\![\tQ,\tau'']\!]$
and vanish when $\tQ=0$, $\tau''=0$.

We define $\Delta^{\bsi}(t)\in \bSTQ$ by the following equation:
$$
(\phi_{\bsi}(t), \phi_{\bsi'}(t))_{\cX,\bT} =\frac{\delta_{\bsi,\bsi'}}{\Delta^{\bsi}(t)}.
$$
Then $\Delta^{\bsi}(t) \to \Delta^{\bsi}$  in the large radius limit
$\tQ,\tau''\to 0$. The normalized canonical basis of $(\hat{H},\star_t)$ is
$$
\{ \hat{\phi}_{\bsi}(t):= \sqrt{\Delta^{\bsi}(t)}\phi_{\bsi}(t): \bsi\in I_\Si\}.
$$
They satisfy
$$
\hat{\phi}_{\bsi}(t)\star_t \hat{\phi}_{\bsi'}(t) =\delta_{\bsi, \bsi'}\sqrt{\Delta^{\bsi}(t)}\hat{\phi}_{\bsi}(t),\quad
(\hat{\phi}_{\bsi}(t), \hat{\phi}_{\bsi'}(t))_{\cX,\bT}=\delta_{\bsi,\bsi'}.
$$
(Note that $\sqrt{\Delta^{\bsi}(t)}= \sqrt{\Delta^{\bsi}} \cdot
\sqrt{\frac{ \Delta^{\bsi}(t) }{\Delta^{\bsi}} }$,
where
$\sqrt{\Delta^{\bsi}}\in \bST$ and
$\sqrt{\frac{ \Delta^{\bsi}(t) }{\Delta^{\bsi}} }\in \bSTQ$,
so $\sqrt{\Delta^{\bsi}(t)}\in \bSTQ$.)
We call $\{\hat{\phi}_{\bsi}(t):t\in I_\Si\}$ the {\em quantum} normalized canonical basis
to distinguish it from the  {\em classical} normalized canonical basis $\{ \hat{\phi}_{\bsi}:\bsi \in I_\Si\}$.
The quantum canonical basis tends to the classical canonical
basis in the large radius limit: $\hat{\phi}_{\bsi}(t)\to \hat{\phi}_{\bsi}$ as $\tQ,\tau''\to 0$.

Let $\Psi=(\Psi_{\bsi'}^{\spa  \bsi})$ be the transition matrix between the classical and quantum
normalized canonical bases:
\begin{equation}\label{eqn:Psi-phi}
\hat{\phi}_{\bsi'}=\sum_{\bsi\in I_\Si} \Psi_{\bsi'}^{\spa \bsi} \hat{\phi}_\bsi(t).
\end{equation}
Then $\Psi$ is an $\chi\times \chi$ matrix with entries
in $\bSTQ$, and $\Psi\to \one$ (the identity matrix)
in the large radius limit $\tQ,\tau''\to 0$. Both
the classical and quantum normalized canonical bases are orthonormal with respect
to the $\bT$-equivariant Poincar\'{e} pairing $(\ ,\ )_{\cX,\bT}$, so $\Psi^T\Psi= \Psi \Psi^T= \one$,
where $\Psi^T$ is the transpose of $\Psi$, or equivalently
$$
\sum_{\brho\in I_\Si} \Psi_{\brho}^{\spa \bsi} \Psi_{\brho}^{\spa \bsi'} =\delta_{\bsi,\bsi'}
$$
Equation \eqref{eqn:Psi-phi} can be rewritten as
$$
\frac{\partial}{\partial t^{\bsi'}} =\sum_{\bsi\in I_\Si} \Psi_{\bsi'}^{\spa \bsi}
\sqrt{\Delta^{\bsi}(t)} \frac{\partial}{\partial u^{\bsi}}
$$
which is equivalent to
\begin{equation}
\label{eqn:Psi-matrix}
\frac{du^{\bsi}}{\sqrt{\Delta^{\bsi}(t)}} =
\sum_{\bsi'\in I_\Si}
dt^{\bsi'} \Psi_{\bsi'}^{\spa \bsi}.
\end{equation}

\subsection{The equivariant big quantum differential equation}
We consider the Dubrovin connection $\nabla^z$, which is a family
of connections parametrized by $z\in \bC\cup \{\infty\}$, on the tangent bundle
$T_{\hat{H}}$ of the formal Frobenius manifold $\hat{H}$:
$$
\nabla^z_{\bsi}=\frac{\partial}{\partial t^{\bsi}} -\frac{1}{z} \hat{\phi}_{\bsi}\star_t
$$
The commutativity (resp. associativity)
of $*_t$ implies that $\nabla^z$ is a torsion
free (resp. flat) connection on $T_{\hat{H}}$ for all $z$. The equation
\begin{equation}\label{eqn:qde}
\nabla^z \mu=0
\end{equation}
for a section $\mu\in \Gamma(\hat{H},\cT_{\hat{H}})$ is called the {\em $\bT$-equivariant
big quantum differential equation} ($\bT$-equivariant big QDE). Let
$$
\cT_{\hat{H}}^{f,z}\subset \cT_{\hat{H}}
$$
be the subsheaf of flat sections with respect to the connection $\nabla^z$.
For each $z$, $\cT_{\hat{H}}^{f,z}$ is a sheaf of
$\novT$-modules of rank $\chi$.

A section $L\in \End(T_{\hat{H}})=\Gamma(\hat{H},\cT_{\hat{H}}^*\otimes\cT_{\hat{H}})$
defines an $\cO_{\hat{H}}(\hat{H})$-linear map
$$
L: \Gamma(\hat{H},\cT_{\hat{H}})= \bigoplus_{\bsi\in I_{\Si}} \cO_{\hat{H}}(\hat{H})
\frac{\partial}{\partial t^{\bsi}}
\to \Gamma(\hat{H},\cT_{\hat{H}})
$$
from the free $\cO_{\hat{H}}(\hat{H})$-module $\Gamma(\hat{H},\cT_{\hat{H}})$ to itself.
Let $L(z)\in \End(T_{\hat H})$ be a family of endomorphisms of the tangent bundle $T_{\hat{H}}$
parametrized by $z$. $L(z)$ is called a {\em fundamental solution} to the $\bT$-equivariant QDE if
the $\cO_{\hat{H}}(\hat{H})$-linear map
$$
L(z): \Gamma(\hat{H},\cT_{\hat{H}}) \to \Gamma(\hat{H},\cT_{\hat{H}})
$$
restricts to a $\novT$-linear isomorphism
$$
L(z): \Gamma(\hat{H},\cT_H^{f,\infty})=\bigoplus_{\bsi\in I_{\Si}} \novT \frac{\partial}{\partial t^{\bsi}}
\to \Gamma(\hat{H},\cT_H^{f,z}).
$$
between rank $\chi$ free $\novT$-modules.


\subsection{The $\cS$-operator}\label{sec:A-S}
The $\cS$-operator is defined as follows.
For any cohomology classes $a,b\in H_{\CR,\bT}^*(\cX;\bST)$,
$$
(a,\cS(b))_{\cX,\bT}=(a,b)_{\cX,\bT}
+\llangle a,\frac{b}{z-\hat{\psi}}\rrangle^{\cX,\bT}_{0,2}
$$
where
$$
\frac{b}{z-\hat{\psi}}=\sum_{i=0}^\infty b\hat{\psi}^i z^{-i-1}.
$$
The $\cS$-operator can be viewed as an element in $\End(T_{\hat{H}})$ and is a fundamental solution to the $\bT$-equivariant
big QDE \eqref{eqn:qde}.  The proof for $\cS$ being a fundamental solution can be found in  \cite[Section 10.2]{CK}
for the smooth case and in \cite{Iri} for the orbifold case which is a direct generalization of the smooth case.

\begin{remark}
One may notice that since there is a formal variable $z$ in the definition of
the $\bT$-equivariant big QDE \eqref{eqn:qde}, one can consider its solution space over different rings. Here the operator
$\cS= \one+ \cS_1/z+ \cS_2/z^2+\cdots$ is viewed as a formal power series in $1/z$ with operator-valued coefficients.
\end{remark}

\begin{remark}\label{Novikov}
Given $t\in H^*_{\CR,\bT}(\cX)\otimes_\RT \bST$, let  $t=t'+t''=\sum_{a=0}^{\chi-1}\tau_aT_a$ where $t'=\sum_{a=1}^{\fp'}\tau_a  T_a\in H^2_{\bT}(\cX)\otimes_\RT \bST$ and $t''$ is a linear combination of elements in $H^{\neq 2}_{\CR,\bT}(\cX)\otimes_\RT \bST$ and elements in degree 2 twisted sectors. Define $\tau'=(\tau_1,\dots, \tau_{\fp'})$ and $\tau''=(\tau_0,\tau_{\fp'+1},\dots,\tau_{\chi-1})$ as in Section \ref{sec:bigQH}. Then by divisor equation, we have
$$
(a,b)_{\cX,\bT}+\llangle a,\frac{b}{z-\hat{\psi}}\rrangle^{\cX,\bT}_{0,2}=(a,be^{t'/z})_{\cX,\bT}+
\sum_{m=0}^\infty \sum_{\substack{d\in E(\cX)\\ (d,m) \neq (0,0) } }\frac{Q^de^{\int_dt'}}{m!}\langle a,\frac{be^{t'/z}}{z-\hat{\psi}},(t'')^m\rangle^{\cX,\bT}_{0,2+m,d}.
$$
In the above expression, if we fix the power of $z^{-1}$, then only finitely many terms in the expansion of $e^{t'/z}$ contribute. Therefore, the factor $e^{\int_dt'}$ can play the role of $Q^d$ and hence the restriction $\llangle a,\frac{b}{z-\hat{\psi}}\rrangle^{\cX,\bT}_{0,2}|_{Q=1}$ is well-defined. So $(a,\cS(b))_{\cX,\bT}\in \bST[\tau'][\![\tQ,\tau'',z^{-1}]\!]$ and  the operator $\cS|_{Q=1}$ is well-defined: $(a,\cS\vert_{Q=1}(b))_{\cX,\bT}\in \bST[\tau'][\![e^{\tau'},\tau'',z^{-1}]\!]$.
\end{remark}

\begin{definition}[$\bT$-equivariant big $J$-function] \label{big-J}
The $\bT$-equivariant big $J$-function $J_{\bT}^{\mathrm{big}}(z)$  is characterized by
$$
(J_{\bT}^{\mathrm{big}}(z),a)_{\cX,\bT} = (1,\cS(a))_{\cX,\bT}
$$
for any $a\in H_{\bT}^*(\cX;\bST)$. Equivalently,
$$
J_{\bT}^{\mathrm{big}}(z)= 1+ \sum_{\bsi\in I_\Si}\llangle 1, \frac{\hat{\phi}_{\bsi}}{z-\hat{\psi}}\rrangle_{0,2}^{\cX,\bT}\hat{\phi}_{\bsi}.
$$
\end{definition}

We consider several different (flat) bases for $H_{\CR,\bT}^*(\cX;\bST)$:
\begin{enumerate}
\item The classical canonical basis $\{ \phi_{\bsi}:\bsi\in I_\Si \}$ defined in Section \ref{sec:equivariantCR}.
\item The basis dual to the classical canonical basis with respect to the $\bT$-equivariant Poincare pairing:
$\{ \phi^{\bsi} =\Delta^{\bsi} \phi_{\bsi}: \bsi \in I_\Si \} $.
\item The classical normalized canonical basis
$\{ \hat{\phi}_{\bsi}=\sqrt{\Delta^{\bsi}}\phi_{\bsi} :\bsi\in I_\Si\}$ which is self-dual: $\{ \hat{\phi}^{\bsi}=\hat{\phi}_{\bsi}: \bsi \in I_\Si \}$.
\end{enumerate}

For $\bsi, \bsi'\in I_\Si$, define
$$
S^{\bsi'}_{\spa \bsi}(z) := (\phi^{\bsi'}, \cS(\phi_{\bsi})).
$$
Then $(S^{\bsi'}_{\spa \bsi}(z))$ is the matrix  of the $\cS$-operator with respect to the canonical basis
$\{\phi_{\bsi}:\bsi\in I_\Si \}$:
\begin{equation}\label{eqn:S}
\cS(\phi_{\bsi}) =\sum_{\bsi'\in I_\Si}
\phi_{\bsi'} S^{\bsi'}_{\spa \bsi}(z).
\end{equation}

For $\bsi,\bsi'\in I_\Si$, define
$$
S_{\bsi'}^{\spa \widehat{\bsi} }(z) := (\phi_{\bsi'}, \cS(\hat{\phi}^{\bsi})).
$$
Then $(S_{\bsi'}^{\spa  \widehat{\bsi}})$ is the matrix of the $\cS$-operator
with respect to the basis $\{\hat{\phi}^{\bsi}:\bsi\in I_\Si\}$ and
$\{\phi^{\bsi}: \bsi\in I_{\Si}\}$:
\begin{equation}\label{eqn:barS}
\cS(\hat{\phi}^{\bsi})=\sum_{\bsi'\in I_{\Si}} \phi^{\bsi'}
 S_{\bsi'}^{\spa \widehat{\bsi}}(z).
\end{equation}

Introduce
\begin{align*}
S_z(a,b)&=(a,\cS(b))_{\cX,\bT},\\
V_{z_1,z_2}(a,b)&=\frac{(a,b)_{\cX,\bT}}{z_1+z_2}+\llangle \frac{a}{z_1-\psi_1},
                  \frac{b}{z_2-\psi_2}\rrangle^{\cX,\bT}_{0,2}.
\end{align*}
The following identity is known (see e.g. \cite{G97}, \cite{GT}):
\begin{equation}
\label{eqn:two-in-one}
V_{z_1,z_2}(a,b)=\frac{1}{z_1+z_2}\sum_i S_{z_1}(T_i,a)S_{z_2}(T^i,b),
\end{equation}
where $T_i$ is any basis of $H^*_{\CR,\bT}(\cX;\bST)$ and $T^i$ is its dual basis.
In particular,
$$
V_{z_1,z_2}(a,b)=\frac{1}{z_1+z_2}\sum_{\bsi\in I_\Si} S_{z_1}(\hat{\phi}_{\bsi},a)S_{z_2}(\hat{\phi}_{\bsi},b).
$$

\subsection{The A-model $R$-matrix}


Let $U$ denote the diagonal matrix whose diagonal entries are the canonical coordinates.
The results in \cite{G97, G98} and \cite{Zo} imply the following statement.
\begin{theorem}\label{R-matrix}
There exists a unique matrix power series $R(z)= \one + R_1z+R_2 z^2+\cdots$
satisfying the following properties.
\begin{enumerate}
\item The entries of $R_d$ lie in $\bSTQ$.
\item $\tS=\Psi R(z) e^{U/z}$  is a fundamental solution to the $\bT$-equivariant
big QDE \eqref{eqn:qde}.
\item $R$ satisfies the unitary condition $R^T(-z)R(z)=\one$.
\item
\begin{equation}\label{eqn:R-at-zero}
\lim_{\tQ,\tau''\to 0} R_{\rho,\delta}^{\spa\si,\gamma}(z)
= \frac{\delta_{\rho,\si}}{|G_\si|}\sum_{h\in G_\si}\chi_\rho(h) \chi_\gamma(h^{-1})
\prod_{i=1}^3 \exp\Big( \sum_{m=1}^\infty \frac{(-1)^m}{m(m+1)}B_{m+1}(c_i(h))
(\frac{z}{\w_i(\si)})^m \Big),
\end{equation}
where $B_m(x)$ is the $m$-th Bernoulli polynomial, defined by the following identity:
$$
\frac{te^{tx}}{e^t-1}=\sum_{m\geq 0}\frac{B_m(x)t^m}{m!}.
$$
\end{enumerate}
\end{theorem}

Each matrix in (2) of Theorem \ref{R-matrix} represents an operator with respect to the classical canonical basis  $\{ \hat{\phi}_{\bsi}: \bsi\in I_\Si\}$.
So $R^T$ is the adjoint of $R$ with respect to the $\bT$-equivariant
Poincar\'{e} pairing $(\ , \ )_{\cX,\bT}$.
The matrix $(\tS_{\bsi'}^{\spa \widehat{\bsi}})(z)$ is of the form
$$
\tS_{\bsi'}^{\spa\widehat{\bsi}}(z)
= \sum_{\brho\in I_\Si} \Psi_{\bsi'}^{\spa \brho}
R_{\brho}^{\spa \bsi}(z) e^{u^{\bsi}/z}
=(\Psi R(z))_{\bsi'}^{\spa \bsi} e^{u^{\bsi}/z}
$$
where $R(z)= (R_{\brho}^{\spa\bsi}(z)) = \one + \sum_{k=1}^\infty R_k z^k$.

We call the unique $R(z)$ in Theorem \ref{R-matrix} the {\em A-model $R$-matrix}.
The A-model $R$-matrix plays a central role in the quantization formula of the descendent potential of $\bT$-equivariant Gromov-Witten
theory of $\cX$. We will state this formula in terms of graph sum in the the next subsection.


\subsection{The A-model graph sum} \label{sec:Agraph}
In \cite{Zo}, the third author generalizes Givental's formula for the total descendant
potential of equivariant Gromov-Witten theory of GKM manifolds to GKM orbifolds; recall that
a complex manifold/orbifold is GKM if it is equipped with a torus action with finitely many 0-dimensional
and 1-dimensional orbits.
In order to state this formula, we need to introduce some definitions.

\begin{itemize}
\item We define
$$
S^{\widehat{\underline{\bsi}}}_{\spa \widehat{\underline{\bsi'}} }(z)
:= (\hat{\phi}_{\bsi}(t), \cS(\hat{\phi}_{\bsi'}(t))).
$$
Then $(S^{ \widehat{\underline{\bsi}}  }_{\spa \widehat{\underline{\bsi'}} }(z))$ is the matrix of the $\cS$-operator with
respect to the normalized canonical basis  $\{ \hat{\phi}_{\bsi}(t): \bsi\in I_\Si\} $:
\begin{equation}
\cS(\hat{\phi}_{\bsi'}(t))=\sum_{\bsi\in I_\Si} \hat{\phi}_{\bsi}(t)
S^{\widehat{\underline{\bsi}} }_{\spa \widehat{\underline{\bsi'}} }(z).
\end{equation}
\item We define
$$
S^{\widehat{\underline{\bsi}}}_{\spa \bsi'}(z)
:= (\hat{\phi}_{\bsi}(t), \cS(\phi_{\bsi'})).
$$
Then $(S^{ \widehat{\underline{\bsi}}  }_{\spa \bsi'}(z))$ is the matrix of the $\cS$-operator with
respect to the  basis $\{\phi_{\bsi}:\bsi\in I_\Si\}$ and
$\{ \hat{\phi}_{\bsi}(t): \bsi\in I_\Si\} $:
\begin{equation}
\cS(\phi_{\bsi'})=\sum_{\bsi\in I_\Si} \hat{\phi}_{\bsi}(t)
S^{\widehat{\underline{\bsi}} }_{\spa \bsi'}(z).
\end{equation}
\end{itemize}

Given a connected graph $\Ga$, we introduce the following notation.
\begin{enumerate}
\item $V(\Ga)$ is the set of vertices in $\Ga$.
\item $E(\Ga)$ is the set of edges in $\Ga$.
\item $H(\Ga)$ is the set of half edges in $\Ga$.
\item $L^o(\Ga)$ is the set of ordinary leaves in $\Ga$. The ordinary
leaves are ordered: $L^o(\Ga)=\{l_1,\ldots,l_n\}$ where
$n$ is the number of ordinary leaves.
\item $L^1(\Ga)$ is the set of dilaton leaves in $\Ga$. The dilaton leaves are unordered.
\end{enumerate}

With the above notation, we introduce the following labels:
\begin{enumerate}
\item (genus) $g: V(\Ga)\to \bZ_{\geq 0}$.
\item (marking) $\bsi: V(\Ga) \to I_\Si$. This induces
$\bsi :L(\Ga)=L^o(\Ga)\cup L^1(\Ga)\to I_\Si$, as follows:
if $l\in L(\Ga)$ is a leaf attached to a vertex $v\in V(\Ga)$,
define $\bsi(l)=\bsi(v)$.
\item (height) $k: H(\Ga)\to \bZ_{\geq 0}$.
\end{enumerate}

Given an edge $e$, let $h_1(e),h_2(e)$ be the two half edges associated to $e$. The order of the two half edges does not affect the graph sum formula in this paper. Given a vertex $v\in V(\Ga)$, let $H(v)$ denote the set of half edges
emanating from $v$. The valency of the vertex $v$ is equal to the cardinality of the set $H(v)$: $\val(v)=|H(v)|$.
A labeled graph $\vGa=(\Ga,g,\bsi,k)$ is {\em stable} if
$$
2g(v)-2 + \val(v) >0
$$
for all $v\in V(\Ga)$.

Let $\bGa(\cX)$ denote the set of all stable labeled graphs
$\vGa=(\Gamma,g,\bsi,k)$. The genus of a stable labeled graph
$\vGa$ is defined to be
$$
g(\vGa):= \sum_{v\in V(\Ga)}g(v)  + |E(\Ga)|- |V(\Ga)|  +1
=\sum_{v\in V(\Ga)} (g(v)-1) + (\sum_{e\in E(\Gamma)} 1) +1.
$$
Define
$$
\bGa_{g,n}(\cX)=\{ \vGa=(\Gamma,g,\bsi,k)\in \bGa(\cX): g(\vGa)=g, |L^o(\Ga)|=n\}.
$$

We assign weights to leaves, edges, and vertices of a labeled graph $\vGa\in \bGa(\cX)$ as follows.
\begin{enumerate}
\item {\em Ordinary leaves.}
To each ordinary leaf $l_j \in L^o(\Ga)$ with  $\bsi(l_j)= \bsi\in I_\Si$
and  $k(l)= k\in \bZ_{\geq 0}$, we assign the following descendant  weight:
\begin{equation}\label{eqn:u-leaf}
(\cL^{\bu})^{\bsi}_k(l_j) = [z^k] (\sum_{\bsi',\brho\in I_\Si}
\left(\frac{\bu_j^{\bsi'}(z)}{\sqrt{\Delta^{\bsi'}(t)} }
S^{\widehat{\underline{\brho}} }_{\spa
  \widehat{\underline{\bsi'}}}(z)\right)_+ R(-z)_{\brho}^{\spa \bsi} ),
\end{equation}
where $(\cdot)_+$ means taking the nonnegative powers of $z$.

\item {\em Dilaton leaves.} To each dilaton leaf $l \in L^1(\Ga)$ with $\bsi(l)=\bsi
\in I_\Si$
and $2\leq k(l)=k \in \bZ_{\geq 0}$, we assign
$$
(\cL^1)^{\bsi}_k := [z^{k-1}](-\sum_{\bsi'\in I_\Si}
\frac{1}{\sqrt{\Delta^{\bsi'}(t)}}
R_{\bsi'}^{\spa \bsi}(-z)).
$$

\item {\em Edges.} To an edge connecting a vertex marked by $\bsi\in I_\Si$ and a vertex
marked by $\bsi'\in I_\Si$, and with heights $k$ and $l$ at the corresponding half-edges, we assign
$$
\cE^{\bsi,\bsi'}_{k,l} := [z^k w^l]
\Bigl(\frac{1}{z+w} (\delta_{\bsi\bsi'}-\sum_{\brho\in I_\Si}
R_{\brho}^{\spa \bsi}(-z) R_{\brho}^{\spa \bsi'}(-w)\Bigr).
$$
\item {\em Vertices.} To a vertex $v$ with genus $g(v)=g\in \bZ_{\geq 0}$ and with
marking $\bsi(v)=\bsi$, with $n$ ordinary
leaves and half-edges attached to it with heights $k_1, ..., k_n \in \bZ_{\geq 0}$ and $m$ more
dilaton leaves with heights $k_{n+1}, \ldots, k_{n+m}\in \bZ_{\geq 0}$, we assign
$$
 \Big(\sqrt{\Delta^{\bsi}(t)}\Big)^{2g(v)-2+\val(v)}\langle  \tau_{k_1}\cdots\tau_{k_{n+m}}\rangle_g,
$$
where $\langle  \tau_{k_1}\cdots\tau_{k_{n+m}}\rangle_{g}=\int_{\Mbar_{g,n+m}}\psi_1^{k_1} \cdots \psi_{n+m}^{k_{n+m}}$.
\end{enumerate}

We define the weight of a labeled graph $\vGa\in \bGa(\cX)$ to be
\begin{eqnarray*}
w_A^{\bu}(\vGa) &=& \prod_{v\in V(\Ga)} \Bigl(\sqrt{\Delta^{\bsi(v)}(t)}\Bigr)^{2g(v)-2+\val(v)} \langle \prod_{h\in H(v)} \tau_{k(h)}\rangle_{g(v)}
\prod_{e\in E(\Ga)} \cE^{\bsi(v_1(e)),\bsi(v_2(e))}_{k(h_1(e)),k(h_2(e))}\\
&& \cdot \prod_{l\in L^1(\Ga)}(\cL^1)^{\bsi(l)}_{k(l)}\prod_{j=1}^n(\cL^{\bu})^{\bsi(l_j)}_{k(l_j)}(l_j).
\end{eqnarray*}
With the above definition of the weight of a labeled graph, we have
the following theorem which expresses the $\bT$-equivariant descendent
Gromov-Witten potential of $\cX$ in terms of graph sum.

\begin{theorem}[{Zong \cite{Zo}}]
\label{thm:Zong}
 Suppose that $2g-2+n>0$. Then
$$
\llangle \bu_1,\ldots, \bu_n\rrangle_{g,n}^{\cX,\bT}=\sum_{\vGa\in \bGa_{g,n}(\cX)}\frac{w_A^{\bu}(\vGa)}{|\Aut(\vGa)|}.
$$
\end{theorem}

\begin{remark}\label{descendent Q}
In the above graph sum formula, we know that the restriction $S^{\widehat{\underline{\brho}} }_{\spa
  \widehat{\underline{\bsi'}}}(z)|_{Q=1}$ is well-defined by Remark \ref{Novikov}. Meanwhile by (1) in Theorem \ref{R-matrix}, we know that the restriction $R(z)|_{Q=1}$ is also well-defined. Therefore by Theorem \ref{thm:Zong}, we have $\llangle \bu_1,\ldots, \bu_n\rrangle_{g,n}^{\cX,\bT}|_{Q=1}$ is well-defined.
\end{remark}

%

We make the following observation.
\begin{lemma}\label{lm:dilaton}
$$
(\cL^{\bu})^{\bsi}_k(l_j)\Big|_{\bu_j(z)=\one z}
= -(\cL^1)^{\bsi}_k  + (\cL^{\bu})^{\bsi}_k(l_j)\Big|_{\bu_j(z)= t}
$$
\end{lemma}
\begin{proof} Let $(\ ,\ )$ denote the $\bT$-equivariant
Poincar\'{e} pairing $(\ , \ )_{\cX,\bT}$. Given $u\in H^*_{\CR,\bT}(\cX)\otimes_{\RT}\otimes\bST$, we define
$S^{\widehat{\underline{\brho}} }_{\spa u}(z) := (\hat{\phi}_\brho, \cS(u))$. Then
\begin{eqnarray*}
(\cL^{\bu})^{\bsi}_k(l_j)\Big|_{\bu_j(z)= \one z} &=& [z^k]
(\sum_{\brho\in I_\Si} \left(z
S^{\widehat{\underline{\brho}} }_{\spa
  \one}(z)\right)_+ R(-z)_{\brho}^{\spa \bsi} )\\
(\cL^{\bu})^{\bsi}_k(l_j)\Big|_{\bu_j(z)= t} &=& [z^k]
(\sum_{\brho\in I_\Si} \left(
S^{\widehat{\underline{\brho}} }_{\spa
  t}(z)\right)_+ R(-z)_{\brho}^{\spa \bsi} ).
\end{eqnarray*}
where
\begin{eqnarray*}
\left(zS^{\widehat{\underline{\brho}} }_{\spa
  \one}(z)\right)_+
&=& z(\hat{\phi}^{\brho}(t), 1)
+\llangle \hat{\phi}^{\brho}(t),1\rrangle^{\cX,\bT}_{0,2}
= \frac{z}{\sqrt{\Delta^\brho(t)}} + (\hat{\phi}^{\rho}(t), t)\\
\left(S^{\widehat{\underline{\brho}} }_{\spa
  t}(z)\right)_+ &=& (\hat{\phi}^\brho(t), t)
\end{eqnarray*}
So
\[
(\cL^{\bu})^{\bsi}_k(l_j)\Big|_{\bu_j(z)= \one z}
= [z^k] \big(\sum_{\brho\in I_\Si}\frac{z}{\sqrt{\Delta^\brho(t)}}
R_\brho^{\spa\bsi}(-z)\big)\
+(\cL^{\bu})^{\bsi}_k(l_j)\Big|_{\bu_j(z)= t}
= -(\cL^1)^{\bsi}_k  +(\cL^{\bu})^{\bsi}_k(l_j)\Big|_{\bu_j(z)= t}.
\]
\end{proof}

As a special case of Theorem \ref{thm:Zong}, if $g>1$ then
$$
\llangle \ \rrangle_{g,0}^{\cX,\bT}
=\sum_{\vGa\in \Gamma_{g,0}(\cX)}
\frac{w_A^{\bu}(\vGa)}{|\Aut(\vGa)|}.
$$
We end this subsection with the following alternative graph sum formula for $\llangle \ \rrangle_{g,0}^{\cX,\bT}$.
\begin{proposition}\label{prop:stable-Fg-A}
If $g>1$ then
$$
\llangle \ \rrangle_{g,0}^{\cX,\bT}
=\frac{1}{2-2g} \sum_{\vGa\in \Gamma_{g,1}(\cX)}
\frac{w_A^{\bu}(\vGa)\Big|_{(\cL^{\bu})^{\bsi}_k(l_1) = (\cL^1)^{\bsi}_k} }{|\Aut(\vGa)|}
$$
\end{proposition}
\begin{proof}
Theorem \ref{thm:Zong} and Lemma \ref{lm:dilaton} imply
\begin{equation}\label{eqn:psi-graph-sum}
\llangle \hat{\psi} \rrangle_{g,1}^{\cX,\bT} =
-\sum_{\vGa\in \Gamma_{g,1}(\cX)}
\frac{w_A^{\bu}(\vGa)\Big|_{(\cL^{\bu})^{\bsi}_k(l_1) = (\cL^1)^{\bsi}_k} }{|\Aut(\vGa)|} + \llangle t \rrangle_{g,1}^{\cX,\bT}.
\end{equation}

On the other hand,
\begin{eqnarray*}
\llangle \hat{\psi}\rrangle_{g,1}^{\cX,\bT} &=&
\sum_{m=0}^\infty\sum_{d\in E(\cX)}\frac{Q^d}{m!}\langle \hat{\psi}, t^m\rangle^{\cX,\bT}_{g,1+m,d} \\
&=& \sum_{m=0}^\infty\sum_{d\in E(\cX)}\frac{Q^d}{m!} (2g-2+m)
\langle t^m\rangle^{\cX,\bT}_{g,m,d} \\
&=& (2g-2) \sum_{m=0}^\infty\sum_{d\in E(\cX)}\frac{Q^d}{m!}
\langle t^m\rangle^{\cX,\bT}_{g,m,d}
+\sum_{m=1}^\infty \sum_{d\in E(\cX)}\frac{Q^d}{(m-1)!}
\langle t^m \rangle^{\cX,\bT}_{g,m,d}\\
&=& (2g-2) \llangle \ \rrangle_{g,0}^{\cX,\bT}
+ \llangle t\rrangle_{g,1}^{\cX,\bT},
\end{eqnarray*}
where the second equality follows from the dilaton equation. Therefore,
\begin{equation}\label{eqn:dilaton-equation}
\llangle\hat{\psi}\rrangle_{g,1}^{\cX,\bT} = (2g-2)\llangle \ \rrangle_{g,0}^{\cX,\bT} +\llangle t\rrangle_{g,1}^{\cX,\bT}.
\end{equation}
The proposition follows from Equation \eqref{eqn:psi-graph-sum}
and Equation \eqref{eqn:dilaton-equation}.
\end{proof}

\subsection{Genus zero mirror theorem over the small phase space} \label{sec:I-J}
In \cite{CCIT}, Coates-Corti-Iritani-Tseng proved a genus-zero mirror theorem
for toric Deligne-Mumford stacks. This theorem is also a consequence of genus-zero wall-crossing in orbifold
quasimap theory \cite{CCK}, and takes a particularly simple form when
the extended stacky fan satisfies the {\em weak Fano} condition \cite[Section 4.1]{Iri}.
We state this theorem for toric Calabi-Yau 3-folds over the small phase space.

\subsubsection{The small phase space}
\label{sec:choice-H}
So far we work with the big phase space $H_{\CR,\bT}^*(\cX;\bST)\cong \bST^{\oplus \chi}$.
In this paper, we define the small phase space to be  $H_{\CR,\bT}^2(\cX) \cong \bC^{3+\fp}$.

Following  \cite[Section 3.1]{Iri},
we choose $H_1,\ldots, H_\fp \in \bL^\vee\cap \Nef(\bSi^{\ext})$ (where
$H_a$ corresponds to the symbol $p_a$ in \cite{Iri}) such that
\begin{itemize}
\item $\{H_1,\ldots,H_\fp\}$ is a $\bQ$-basis of $\bL_\bQ^\vee$.
\item $\{\bar{H}_1,\ldots, \bar{H}_{\fp'}\}$ is a $\bQ$-basis
of $H^2(\cX;\bQ)=H^2(X_\Si;\bQ)$.
\item $H_a= D_{3+a}$ for $a= \fp'+1,\dots,\fp$.
\item $\bar H_a=H_a$ for $i=1,\dots,\fp'$, where we regard $H^2(\cX;\bQ)$ as a subspace of $\bL^\vee_\bQ$ as in Equation \eqref{eqn:split}, i.e. $D^\vee_i(H_a)=0$.
\item Given any $3$-cone $\si\in \Si(3)$, $\{ D_i: b_i\in I_\si\}$ is a $\bQ$-basis of $\bL_\bQ^\vee$ and
$H_a \in \widetilde{\Nef}_\si =\sum_{i\in I_\si} \bR_{\geq 0} D_i$, so
\begin{equation}\label{eqn:s}
H_a =\sum_{b_j\in I_\si} s^\si_{aj} D_j
\end{equation}
where $s^\si_{aj} \in \bQ_{\geq 0}$.  We choose $H_a$ such that $s_{aj}^\si \in \bZ_{\geq 0}$ for
all $\si\in \Si(3)$, $j\in \{ j: b_j\in I_\si\}$ and $a\in \{1,\ldots,\fp\}$.
\end{itemize}
Then $H_1,\ldots, H_\fp$ is a $\bC$-basis of $H_{\CR}^2(\cX)\cong \bC^\fp$, and
$\bar{H}_1,\ldots, \bar{H}_{\fp'}$ lie in the Kahler cone $C(\Si)$.

For $a=1,\ldots,\fp'$ let $\bar{H}_a^{\bT}\in H^2_{\bT}(\cX)$ be the unique
$\bT$-equivariant lifting of $\bar{H}_a\in H^2(\cX)$ such that
$\bar H_a^{\bT}|_{\fp_{\si_0}}=0$. Then
$$
H_{\CR,\bT}^2(\cX) = \bigoplus_{i=1}^{3+\fp}\bC D_i^{\bT}
=  \bC \su_1 \oplus \bC \su_2 \oplus \bC \su_3\oplus  \bigoplus_{a=1}^{\fp'} \bC \bar{H}_a^{\bT} \oplus \bigoplus_{a=\fp'+1}^\fp \bC \one_{b_{a+3}}
$$
Any $\btau\in H^2_{\CR,\bT}(\cX,\bC)$ can be written as
$$
\btau = \tau_0 +\sum_{a=1}^{\fp'} \tau_a \bar{H}^{\bT}_a \oplus \sum_{a=\fp'+1}^\fp \tau_a \one_{b_{a+3}}
$$
where $\tau_0 \in H^2(B\bT)=\bC u_1\oplus \bC u_2\oplus \bC u_3$ and $\tau_1,\ldots,\tau_\fp\in \bC$.
We write $\btau=\btau'+\btau''$, where
$$
\btau'\in H^2_{\bT}(\cX)= H^2_{\bT}(X_\Si), \quad \btau'' = \sum_{a=\fp'+1}^\fp \tau_a \one_{b_{a+3}}.
$$

\subsubsection{The small equivariant quantum cohomology ring}
In section \ref{sec:bigQH}, we defined the big equivariant quantum cohomology ring, where the quantum product $\star_t$ depends on the point $t$ in the big phase space $H_{\CR,\bT}^*(\cX;\bST)$. The small equivariant quantum cohomology ring is defined by restricting $t$ to the small phase space $H_{\CR,\bT}^2(\cX)$.

More concretely, let $QH^*_{\CR,\bT}(\cX):=H_{\CR,\bT}^*(\cX;\bST)\otimes_{\bST}
\bST[\![\tQ,\tau_a]\!]_{a=\fp'+1,\cdots,\fp}$. Then $QH^*_{\CR,\bT}(\cX)$ is a free $\bST[\![\tQ,\tau_a]\!]_{a=\fp'+1,\cdots,\fp}$-module of rank $\chi$. Define the small quantum product $\star_{\btau}$ to be $\star_{\btau}:=\star_t|_{t=\btau}$, where $\btau$ is in the small phase space $H_{\CR,\bT}^2(\cX)$. The pair $(QH^*_{\CR,\bT}(\cX),\star_{\btau})$ is called the \emph{small equivariant quantum cohomology ring} of $\cX$.

The small equivariant quantum cohomology ring $(QH^*_{\CR,\bT}(\cX),\star_{\btau})$ is still semisimple. In fact, let $\phi_\bsi(\btau):=\phi_\bsi(t)|_{t=\btau}$ be the restriction of $\phi_\bsi(t)$ to the small phase space. Then
$$
\{\phi_\bsi(\btau):\bsi\in I_\Si\}
$$
is a canonical basis of $(QH^*_{\CR,\bT}(\cX),\star_{\btau})$.

\subsubsection{The equivariant small $J$-function}
The $\bT$-equivariant small $J$-function $J_{\bT}(\btau,z)$ is the restriction
of the $\bT$-equivariant big $J$-function to the small phase space. Given
$\btau \in  H^2_{\CR,\bT}(\cX)$
$$
J_{\bT}(\btau,z) := J^{\mathrm{big}}_{\bT}(z)|_{t=\btau, Q=1},
$$
where $J^{\mathrm{big}}_{\bT}(z)$ is defined in Definition \ref{big-J}.
The restriction to $Q=1$ is well-defined by Remark \ref{Novikov}.
Therefore,
\begin{eqnarray*}
&& J_{\bT}(\btau,z) = 1+\sum_{m=0}^\infty \sum_{d\in E(\cX)}\sum_{\bsi\in I_\Si}\frac{1}{m!}\langle 1,\frac{\hat{\phi}_{\bsi}}{z-\hat{\psi}},
\btau^m\rangle_{0,2+m,d}^{\cX,\bT}\hat{\phi}_{\bsi}\\
&=&
e^{ \btau'/z }(1+\sum_{m=0}^\infty \sum_{d\in E(\cX)}\sum_{\bsi\in I_\Si}
\frac{e^{\langle \btau', d\rangle}}{m!}
\langle 1, \frac{\hat{\phi}_{\bsi}}{z-\hat{\psi}}, (\btau'')^m\rangle_{0,2+m,d}^{\cX,\bT}
\hat{\phi}_{\bsi})
\end{eqnarray*}

\subsubsection{The equivariant small $I$-function}
\label{sec:small-I}

We define \emph{charges} $m^{(a)}_i\in \bQ$ by
$$
D_i=\sum_{a=1}^\fp m^{(a)}_i H_a.
$$
Let $t_0,q=(q_1,\ldots, q_\fp)$ be formal variables, and define
$q^\beta = q_1^{\langle H_1,\beta\rangle} \cdots q_\fp^{\langle H_\fp,\beta\rangle}$
for $\beta\in \bK$. Given the choice of $H_a$ in Section \ref{sec:choice-H}, $q_i=q^{D_i^\vee}$ for $i=\fp'+1,\dots,\fp$. The limit point $q\to 0$ is a B-model large
complex structure/orbifold mixed-type limit point. Let $q_K=(q_1,\dots,q_{\fp'})$, and $q_\mathrm{orb}=(q_{\fp'+1},\dots,q_{\fp})$. Then we say \emph{one takes the large complex structure limit} by setting $q_K=0$. Under mirror symmetry, this corresponds to taking the large radius limit  (i.e. setting all K\"ahler classes to infinity) while preserving the twisted classes. When $\cX$ is a smooth toric
variety, we have $\fp'=\fp$.  Following \cite[Definition 4.1]{Iri}, and \cite[Definition 28]{CCIT},
we define $\bT$-equivariant small $I$-function as follows.
\begin{definition}\label{def-I}
\begin{eqnarray*}
I_{\bT}(t_0,q,z) &=& e^{ (t_0+\sum_{a=1}^{\fp'} \bar H_a^\bT \log q_a)/z}
\sum_{\beta\in \bK_\eff} q^\beta \prod_{i=1}^{3+\fp'}
\frac{\prod_{m=\lceil \langle D_i, \beta \rangle \rceil}^\infty
(\bar{D}^\bT_i  +(\langle D_i,\beta\rangle -m)z)}{\prod_{m=0}^\infty(\bar{D}^\bT_i+(\langle D_i,\beta\rangle -m)z)} \\
&& \quad \cdot \prod_{i=4+\fp'}^{3+\fp} \frac{\prod_{m=\lceil \langle D_i, \beta \rangle \rceil}^\infty (\langle D_i,\beta\rangle -m)z}
{\prod_{m=0}^\infty(\langle D_i,\beta\rangle -m)z} \one_{v(\beta)}.
\end{eqnarray*}
Note that $\langle H_a,\beta\rangle\geq 0$ for $\beta\in \bK_{\eff}$.
\end{definition}

\begin{remark}
$\bar{D}_i^{\bT}\in H^2_{\bT}(\cX)$ in this paper corresponds to $u_i$ in \cite[Definition 28]{CCIT},
and $H_a$ (resp. $\bar{H}_a$) in this paper corresponds to $p_a$ (resp. $\bar{p}_a$) in \cite{Iri}.
The $I$-function in \cite[Definition 28]{CCIT} depends on variables
$t_1,\ldots, t_{3+\fp'}$, which are related to the variables $t_0, \log q_1,\ldots, \log q_{\fp'}$
in Definition \ref{def-I} by
$$
\sum_{i=1}^{3+\fp'} t_i \bar{D}_i^{\bT} = t_0 +\sum_{a=1}^{\fp'} \log q_a \bar{H}_a^\bT.
$$
Equivalently,
$$
t_0 = \sum_{i=1}^3 t_i \sw_i,\quad
\log q_a =\sum_{i=1}^{3+\fp'} m_i^{(a)} t_i,
$$
where $\sw_i\in \bC \su_1\oplus \bC \su_2\oplus \bC \su_3$ is the restriction
of $D_i^{\bT}$ to the fixed point $\fp_{\si_0}$.
\end{remark}

We now study the expansion of $I_{\bT}(t_0,q,z)$ in powers of $z^{-1}$.
It can be rewritten as
\begin{eqnarray*}
I_{\bT}(t_0,q,z) &=& e^{ (t_0 +\sum_{a=1}^{\fp'} \bar{H}_a^\bT \log q_a )/z}
\sum_{\beta\in \bK_\eff} \frac{q^\beta}{z^{\langle \hat \rho ,\beta\rangle + \age(v(\beta))} } \prod_{i=1}^{3+\fp'}
\frac{\prod_{m=\lceil \langle D_i, \beta \rangle \rceil}^\infty ( \frac{\bar D^\bT_i}{z}  +\langle D_i,\beta\rangle -m)}
{\prod_{m=0}^\infty(\frac{\bar D^\bT_i}{z}+ \langle D_i,\beta\rangle -m)} \\
&& \quad \cdot \prod_{i=4+\fp'}^{3+\fp} \frac{\prod_{m=\lceil \langle D_i, \beta \rangle \rceil}^\infty
(\langle D_i,\beta\rangle -m)}{\prod_{m=0}^\infty(\langle D_i,\beta\rangle -m)} \one_{v(\beta)}
\end{eqnarray*}
where $\hat \rho = D_1+\cdots +D_{3+\fp} \in C(\Si^{\ext})$.

For $i=1,\ldots, 3+\fp$, we will define $\Omega_i \subset \bK_\eff-\{0\}$.
and $A_i(q)$ supported on $\Omega_i$. We observe that, if $\beta\in \bK_\eff$ and $v(\beta)=0$ then
$\langle D_i,\beta\rangle\in \bZ$ for $i=1,\ldots, 3+\fp$.
\begin{itemize}
\item For $i=1,\ldots,3+\fp'$, let
$$
\Omega_i =\left \{ \beta\in \bK_\eff: v(\beta)=0, \langle D_i,\beta\rangle <0 \textup{ and  }
\langle D_j, \beta \rangle \geq 0 \textup{ for } j\in \{1,\ldots, 3+\fp\}-\{i\} \right \}.
$$
Then $\Omega_i\subset \{ \beta\in \bK_\eff: v(\beta)=0, \beta\neq 0\}$. We define
$$
A_i(q):=\sum_{\beta \in \Omega_i} q^\beta
\frac{(-1)^{-\langle D_i,\beta\rangle-1}(-\langle D_i, \beta\rangle -1)! }{
\prod_{j\in\{1,\ldots, 3+\fp\}-\{i\}}\langle D_j,\beta\rangle!}.
$$
\item For $i=4+\fp',\ldots, 3+\fp$, let
$$
\Omega_i := \{ \beta\in \bK_\eff: v(\beta)= b_i, \langle D_j, \beta\rangle \notin \bZ_{<0} \textup{ for }j=1,\ldots, 3+\fp\},
$$
and define
$$
A_i(q) = \sum_{\beta\in \Omega_i} q^\beta
\prod_{j=1}^{3+\fp} \frac{\prod_{m=\lceil \langle D_j, \beta \rangle \rceil}^\infty (\langle D_j,\beta\rangle -m)}{\prod_{m=0}^\infty(\langle D_j,\beta\rangle -m)}.
$$
\end{itemize}
Note that  $A_i(q) =q_i+O(|q_{\mathrm{orb}}|^2)+O(q_K)$ for $i=4+\fp',\ldots, 3+\fp$.

$$
I(t_0,q,z)= 1 +\frac{1}{z}(t_0 +\sum_{a=1}^{\fp'} \log(q_a) \bar{H}^{\bT}_a +
\sum_{i=1}^{3+\fp'} A_i(q) \bar{D}_i^{\bT}+ \sum_{i=4+\fp'}^{3+\fp} A_i(q) \one_{b_i}) + o(z^{-1}).
$$
where $o(z^{-1})$ involves $z^{-k}$, $k\geq 2$ We have
$$
\bar{D}^{\bT}_i =
\begin{cases} \sum_{a=1}^{\fp'}m_i^{(a)} \bar{H}_a^{\bT} +  \sw_i, & i=1,2,3,\\
\sum_{a=1}^{\fp'} m_i^{(a)} \bar{H}_a^{\bT}, & 4\leq i\leq 3+\fp'.
\end{cases}
$$
Let $S_a(q):=\sum_{i=1}^{3+\fp'} m_i^{(a)} A_i(q)$. Then
$$
I_{\bT}(t_0,q,z)= 1 +\frac{1}{z}( (t_0 +\sum_{i=1}^3 \sw_i A_i(q))
 +\sum_{a=1}^{\fp'} (\log(q_a) + S_a(q)) \bar{H}^{\bT}_a
+ \sum_{i=4+\fp'}^{3+\fp} A_i(q) \one_{b_i}) + o(z^{-1}).
$$

\subsubsection{The mirror theorem}
The results in \cite{CCK, CCIT} imply the following $\bT$-equivariant
mirror theorem:
\begin{theorem} \label{thm:IJ}
$$
J_{\bT}(\btau,z) = I_{\bT}(t_0,q, z),
$$
where the equivariant closed mirror map  $(t_0,q)\mapsto  \btau(t_0,q)$
is determined by the first-order term in the asymptotic expansion of the $I$-function
$$
I(t_0,q,z)=1+\frac{ \btau(t_0,q)}{z}+o(z^{-1}).
$$
More explicitly, the equivariant closed mirror map is given by
$$
\btau =\tau_0(t_0,q) + \sum_{a=1}^{\fp'}\tau_a(q) \bar{H}^\bT_a +\sum_{a=\fp'+1}^\fp \tau_a(q) \one_{b_{a+3}},
$$
where
\begin{eqnarray}
\label{eqn:closed-mirror-map}
\tau_0(t_0,q) &=& t_0 +\sum_{i=1}^3 \w_i A_i(q), \nonumber\\
\tau_a(q) &=& \begin{cases}
\log(q_a)+ S_a(q)
, & 1\leq a\leq \fp',\\
A_{a+3}(q), & \fp'+1\leq a\leq \fp.
\end{cases}
\end{eqnarray}
\end{theorem}

Note $S_a(q)$ and $A_{a+3}(q)$ do not contain any equivaraint parameter $\w_i$ and is an element in $\bC[\![q_1,\dots,q_\fp]\!]$ by degree reason. Under this mirror map, the B-model large complex structure/orbifold mixed-type
limit $q\to 0$ corresponds to the A-model large radius/orbifold
mixed type limit $\tQ\to 0, \tau''\to 0$.

\subsection{Non-equivariant small $I$-function}
Choose a basis $\{ e_1, \ldots, e_\fp\}$ of $H^2_{\CR}(\cX)$ such
that $\{e_1,\ldots, e_{\fg} \}$ is a basis of $H^2_{\CR,\mathrm{c}}(\cX)$.
We choose a basis $\{e^1, \ldots, e^{\fg}\}$ of $H^4_{\CR}(\cX)$ which is
dual to $\{e_1,\ldots, e_{\fg}\}$ under the perfect pairing
$H^2_{\CR,\mathrm{c}}(\cX)\times H^4_{\CR}(\cX)\to \bC$.
Then
\begin{align*}
I(q,z):= &I_{\bT}(0,q,z)\Bigr|_{\bar{H}_a^{\bT}=\bar{H}_a, \bar{D}_i^{\bT}= \bar{D}_i} =  1 +\frac{1}{z} \sum_{a=1}^\fp T^a(q) e_a +
\frac{1}{z^2} \sum_{b=1}^{\fg} W_b(q) e^b\\
=&1 +\frac{1}{z} \sum_{a=1}^\fp \tau_a(q) H_a +
\frac{1}{z^2} \sum_{b=1}^{\fg} W_b(q) e^b
\end{align*}
for some generating functions $T^a(q), W_b(q)$ of $q$.

For $b=1,\ldots,\fg$, the non-equivariant limit of
$(I_\bT(0, q,z),e_b)$ exists, and is equal to $z^{-2}W_b(q)$;  the non-equivariant limit of
$\llangle e_b\rrangle_{0,1}^{\cX,\bT}$ exists, and is denote by $\llangle e_b\rrangle_{0,1}^{\cX}$.
By the mirror theorem
$$
z^2(I_\bT(0, q,z),e_b) = z^2(J_{\bT}(\btau(0,q),z), e_b) = z^2 \llangle \frac{e_b}{z(z-\psi)}\rrangle_{0,1}^{\cX,\bT}\Bigr|_{t=\btau(0,q), Q=1}.
$$
Taking the non-equivariant limit of the above equation, we obtain
$$
W_b(q) =  \llangle e_b \rrangle_{0,1}^{\cX}\Bigr|_{t=\btau, Q=1} 
$$
under the mirror map.  

When the coarse moduli space $X_\Si$ of $\cX$ is a smooth toric variety (so that
$X_\Si=\cX$),
$T^1,\ldots, T^\fp$ have logarithm singularities
and $W_1,\ldots, W_{\fg}$  have double logarithm singularities.

\subsection{Non-equivariant Picard-Fuchs System} \label{sec:PF}
Given $\beta\in \bL$, define
\begin{align*}
\bD_\beta :=  q^\beta  \prod_{i: \langle D_i,\beta\rangle <0} \prod_{m=0}^{-\langle D_i,\beta\rangle -1}
(\bD_i-m) -\prod_{i:\langle D_i,\beta\rangle >0} \prod_{m=0}^{\langle D_i,\beta\rangle -1}(\bD_i -m)
\end{align*}
where
$$
\bD_i =\sum_{a=1}^{\fp} m_i^{(a)}  q_a\frac{\partial}{\partial q_a}.
$$

\begin{proposition} \label{prop:PF}
The solution space to the non-equivariant Picard-Fuchs system
\begin{equation}
\bD_\beta F(q) =0, \quad \beta \in \bL.
\label{eqn:PF}
\end{equation}
is $1+\fp+\fg$-dimensional. It is spanned by the coefficients of the non-equivariant small $I$-function:
  \[
  \{1,\tau_1(q),\dots,\tau_\fp(q),W_1(q),\dots,W_\fg(q)\}
  \]
  or equivalently by
  \[
  \{1,T^1(q),\dots,T^\fp(q),W_1(q),\dots,W_\fg(q)\}.
  \]
\end{proposition}
The fact that the non-equivariant $I$-function is annihilated by $\bD_\beta$ is due to Givental \cite{G98} (see also \cite[Lemma 4.6]{Iri}). See \cite[Proposition 4.4]{Iri} (and more recently \cite[Section 5.2]{CCIT-hodge}) for the dimension of the solution space and the discussion of the GKZ-style $D$-module related to the operators $\bD_\beta$.

\subsection{Restriction to the Calabi-Yau torus}
The inclusion $\bT'\hookrightarrow \bT$ induces a surjective ring homomorphism
$$
R_{\bT}= H^*_{\bT}(\mathrm{pt})=\bC[\su_1,\su_2,\su_3]\longrightarrow
R_{\bT'}=H^*_{\bT'}(\mathrm{pt})=\bC[\su_1,\su_2]
$$
given by $\su_1\mapsto \su_1$, $\su_2\mapsto \su_2$, $\su_3\mapsto 0$.
The image of $\{ \phi_{\bsi}:\bsi\in I_\Si\}$ under
the surjective ring homomorphism
\begin{equation}\label{eqn:iota}
H^*_{\CR,\bT}(\cX)\otimes_\RT\bST\longrightarrow H^*_{\CR,\bT'}(\cX)\otimes_\Rt\bSt
\end{equation}
is a canonical basis of the semisimple $\bSt$-algebra
$H^*_{\CR,\bT'}(\cX)\otimes_\Rt \bSt$.
In the rest of this paper, we will often consider $\bT'$-equivariant
cohomology. By slight abuse of notation, we also use the symbol $\phi_{\bsi}$ to denote the image of
$\phi_{\bsi}$ under the ring homomorphism in \eqref{eqn:iota}.

\subsection{Open-closed Gromov-Witten invariants}\label{sec:open-closed-GW}
Let $\cL\subset \cX$ be an outer Aganagic-Vafa Lagrangian brane.
Let $G_0 := G_{\si_0}$ be the stabilizer of the stacky point $\fp_{\si_0}$.
Our notation is similar to that in \cite[Section 5]{FLZ}. In
particular, the integer $f$ is a framing, and $\bT_f:=\Ker
(\su_2-f\su_1)$. The morphism
$H^*(B\bT';\bQ)=\bQ[\su_1,\su_2]\to H^*(B\bT_f; \bQ)=\bQ[\sv]$
is given by $\su_1\mapsto \sv$, $\su_2\mapsto f\sv$.
The weights of $\bT'$-action on $T_{\fp_{\si_0}}\cX$ are
$$
\sw'_1 = \frac{1}{\fr}\su_1 ,\quad \sw'_2 = \frac{\fs}{\fr\fm}\su_1 +\frac{1}{\fm}\su_2, \quad
\sw'_3 = -\frac{\fs+\fm}{\fr\fm}\su_1 -\frac{1}{\fm}\su_2,
$$
so the weights of $\bT_f$-action on $T_{\fp_{\si_0}}\cX$ are $w_1\sv$, $w_2\sv$, $w_3\sv$, where
$$
w_1= \frac{1}{\fr},\quad w_2=\frac{\fs +\fr f }{\fr\fm} ,\quad
w_3=\frac{-\fm- \fs- \fr f}{\fr\fm}.
$$
Recall that the integers $\fm,\fs,\fr$ are defined in the end of Section \ref{sec:av-branes}.

Let the correlator $\langle \btau^\ell \rangle^{\cX,(\cL,f)}_{g,d,(\mu_1,k_1),\ldots, (\mu_n,k_n)}$ denote
the equivariant open-closed Gromov-Witten invariant defined in \cite[Section 3]{FLT}.
Define the open-closed Gromov-Witten potential
\begin{eqnarray*}
\tF_{g,n}^{\cX,(\cL,f)}(\btau,Q; \tX_1,\ldots, \tX_n)
&=&\sum_{d\in \Eff(\cX)}\sum_{\mu_1,\ldots,\mu_n>0}\sum_{k_1,\ldots,k_n=0}^{\fm-1}\sum_{\ell\geq 0}
\frac{\langle \btau^\ell \rangle^{\cX,(\cL,f)}_{g,d,(\mu_1,k_1),\ldots, (\mu_n,k_n)}}{\ell !}\\
&& \cdot Q^d \cdot
\prod_{j=1}^n (\tX_j)^{\mu_j}\cdot ((-1)^\frac{-k_1}{\fm}) \one'_\frac{-k_1}{\fm}\otimes \cdots \otimes
((-1)^\frac{-k_n}{\fm})\one'_\frac{-k_n}{\fm},
\end{eqnarray*}
which is an $H^*_{\CR}(\cB\bmu_\fm;\bC)^{\otimes n}$-valued function, where
$$
H^*_{\CR}(\cB\bmu_\fm;\bC)=\bigoplus_{k=0}^{\fm-1} \bC\one'_\frac{k}{\fm}.
$$
We introduce some notation.
\begin{enumerate}
\item Given $d_0\in \bZ_{\geq 0}$ and $k\in \{0,\dots,\fm-1\}$ let $D'(d_0,k)$ be the disk factor
defined by Equation (13) in \cite{FLZ}, and define
$$
h(d_0,k): = (e^{2\pi\sqrt{-1}d_0w_1},e^{2\pi\sqrt{-1}(d_0w_2-\frac{k}{\fm})},
e^{2\pi\sqrt{-1}(d_0w_3+\frac{k}{\fm})}) \in G_0\subset \bT=(\bC^*)^3.
$$

\item Given $h\in G_0$, define
\begin{eqnarray*}
\Phi^h_0(\tX)&:=&\frac{1}{|G_0|}
\sum_{\substack{(d_0,k)\in \bZ_{\geq 0}\times \{0,\dots,\fm-1\}\\ h(d_0,k)=h} } D'(d_0,k) \tX^{d_0} ((-1)^{-k/\fm})\one'_\frac{-k}{\fm} \\
&=& -\frac{1}{|G_0|}
\sum_{\substack{(d_0,k)\in \bZ_{\geq 0}\times \{0,\dots,\fm-1\} \\ h(d_0,k)=h} }
\frac{\fr}{\sv}e^{\sqrt{-1}\pi(d_0 w_3-c_3(h))}
\big(\frac{\sv}{d_0}\big)^{\age(h)-1} \\
&&  \cdot \frac{\Gamma(d_0 (w_1+w_2)+c_3(h)) }{
\Gamma(d_0 w_1-c_1(h)+1)\Gamma(d_0 w_2 -c_2(h)+1)}\tX^{d_0} \one'_\frac{-k}{\fm}.
\end{eqnarray*}
Then $\Phi^h_0(\tX)$ takes values in $\bigoplus_{k=0}^{\fm-1}\bC \sv^{\age(v)-2} \one'_\frac{k}{\fm}$.

For $a\in \bZ$ and $h\in G_0$,  we define
$$
\Phi^h_a(\tX) := \frac{1}{|G_0|}\sum_{\substack{d_0> 0 \\
    h(d_0,k)=h} } D'(d_0,k)(\frac{d_0}{\sv})^a \tX^{d_0}((-1)^{-k/\fm})
\one'_\frac{-k}{\fm}.
$$
Then $\Phi^h_a(\tX)$ takes values in $\bigoplus_{k=0}^{\fm-1}\bC \sv^{\age(v)-2-a} \one'_\frac{k}{\fm}$, and
$$
\Phi^h_{a+1}(\tX)=(\frac{1}{\sv}\tX\frac{d}{d\tX})\Phi^h_a(\tX).
$$

\item For $a\in \bZ$ and $\alpha \in G_0^*$, we define
$$
\txi^\alpha_a(\tX):= |G_0|\sum_{h\in G}\chi_\alpha(h^{-1})
\bigl(\prod_{i=1}^3 (w_i\sv)^{1-c_i(h)}\bigr) \Phi^h_a(\tX).
$$
Then $\txi^\alpha_a(\tX)$ takes values in $\bigoplus_{k=0}^{\fm-1}\bC
\sv^{1-a} \one'_\frac{k}{\fm}$. We introduce
$$
\txi^\alpha(z,\tX):=\sum_{a\in \bZ_{\geq-2}} z^a \txi^\alpha_a(\tX).
$$

\item Given $h\in G_0=G_{\si_0}$, let $\one_{\si_0,h}$ be characterized by
$\one_{\si_0,h}|_{\fp_{\si}} =\delta_{\si,\si_0} \one_h$. We define
$\one_{\si_0,h}^*=|G|\be_h\one_{\si_0, h^{-1}}$, where $\be_h = \prod_{i=1}^3 (w_i\sv)^{\delta_{0,c_i(h)}}$.
\end{enumerate}

With the above notation, we have:
\begin{proposition}\label{prop:open-descendant}
\begin{enumerate}
\item {\em (disk invariants)}
\begin{eqnarray*}
&& \tF_{0,1}^{\cX,(\cL,f)}(\btau,Q;\tX)\\
&=&\Phi^1_{-2}(\tX) +\sum_{a=1}^{\fp} \tau_a \Phi^{h_a}_{-1}(\tX)+ \sum_{a\in \bZ_{\geq 0}}
\sum_{h\in G_0} \big( \llangle \one_{\si_0,h}^*\hat{\psi}^a \rrangle^{\cX,\bT_f}_{0,1}\big|_{t=\btau}\big)
\Phi^h_a(\tX) \\
&=& \frac{1}{|G_0|^2 w_1 w_2 w_3}\Big( \sum_{\alpha\in G_0^*}\txi^\alpha_{-2}(\tX)
+ \sum_{a=1}^{\fp} \tau_a \prod_{i=1}^3 w_i^{c_i(h_a)}\sum_{\alpha \in G_0^*}
\chi_\alpha(h_a)\txi^\alpha_{-1}(\tX)\Big)\Big|_{\sv=1}\\
&&  +\sum_{a\in \bZ_{\geq 0}}\sum_{\alpha\in G_0^*}
\big(\llangle \phi_{\si_0,\alpha}\hat{\psi}^a \rrangle^{\cX,\bT_f}_{0,1}\big|_{t=\btau} \big) \txi^\alpha_a(\tX)\\
&=&[z^{-2}]\sum_{\alpha\in
    G_0^*}S_z(1,\phi_{\si_0,\alpha})\txi^\alpha(z,\tX).
\end{eqnarray*}
\item{\em (annulus invariants)}
\begin{eqnarray*}
&& \tF_{0,2}^{\cX,(\cL,f)}(\btau,Q; \tX_1,\tX_2)-F^{\cX,(\cL,f)}_{0,2}(0;\tX_1,\tX_2)\\
&=& \sum_{a_1,a_2\in \bZ_{\geq 0}}
\sum_{h_1,h_2\in G_0}
\big(  \llangle \one_{\si_0, h_1}^*\hat{\psi}^{a_1}, \one_{\si_0, h_2}^* \hat{\psi}^{a_1}\rrangle^{\cX,\bT_f}_{0,2}\big|_{t=\btau}\big)
\Phi^{h_1}_{a_1}(\tX_1)\Phi^{h_2}_{a_2}(\tX_2) \\
&=& \sum_{a_1, a_2\in \bZ_{\geq 0}} \sum_{\alpha_1, \alpha_2\in G_0^*}
\big( \llangle {\phi}_{\si_0, \alpha_1}\hat{\psi}^{a_1},
\phi_{\si_0,\alpha_2} \hat{\psi}^{a_1} \rrangle^{\cX,\bT_f}_{0,2}\big|_{t=\btau}\big)
\txi^{\alpha_1}_{a_1}(\tX_1)\txi^{\alpha_2}_{a_2}(\tX_2)\\
\end{eqnarray*}
where
\begin{equation}\label{eqn:annulus-zero}
\begin{aligned}
&(\tX_1\frac{\partial}{\partial \tX_1}+ \tX_2\frac{\partial}{\partial \tX_2}) \tF_{0,2}^{\cX,(\cL,f)}(0;\tX_1,\tX_2)\\
=&|G_0| \Big(\sum_{h\in G_0} \be_h \Phi^h_0(\tX_1) \Phi^{h^{-1}}_0(\tX_2)\Big)\Big|_{\sv=1}\\
=&\frac{1}{|G_0|^2 w_1w_2w_3} \Big(\sum_{\gamma\in G_0^*}\big(\txi^\gamma_0(\tX_1)\txi^\gamma_0(\tX_2)\Big)\Big|_{\sv=1}.
\end{aligned}
\end{equation}

So we have
\begin{align*}
&\tF_{0,2}^{\cX,(\cL,f)}(\btau, Q; \tX_1, \tX_2)\\
=&[z_1^{-1}z_2^{-1}]\sum_{\alpha_1,\alpha_2\in
   G_0^*}V_{z_1,z_2}(\phi_{\si_0, \alpha_1}, \phi_{\si_0,\alpha_2})\txi^{\alpha_1}(z_1,\tX_1)\txi^{\alpha_2}(z_2,\tX_2).
\end{align*}

\item For $2g-2+n>0$,
\begin{eqnarray*}
&& \tF_{g,n}^{\cX,(\cL,f)}(\btau,Q; \tX_1,\ldots, \tX_n)\\
&=&\sum_{a_1,\ldots, a_n\in \bZ_{\geq 0}}
\sum_{h_1,\ldots, h_n\in G_0}
\Big( \llangle \one_{\si_0, h_1}^*\hat{\psi}^{a_1},\ldots, \one_{\si_0, h_n}^*\hat{\psi}^{a_n} \rrangle^{\cX,\bT_f}_{g,n}\Big|_{t=\btau}\Big)
\prod_{j=1}^n \Phi^{h_j}_{a_j}(\tX_j) \\
&=&\sum_{a_1,\ldots, a_n\in \bZ_{\geq 0}}
\sum_{\alpha_1,\ldots, \alpha_n\in G_0^*}
\Big(\llangle {\phi}_{\si_0,\alpha_1}\hat{\psi}^{a_1},\ldots
{\phi}_{\si_0,\alpha_n} \hat{\psi}^{a_n}\rrangle^{\cX,\bT_f}_{g,n}\Big|_{t=\btau}\Big)
\prod_{j=1}^n \txi^{\alpha_j}_{a_j}(\tX_j).\\
&=&[z_1^{-1}\dots z_n^{-1}]\sum_{\alpha_1,\dots,\alpha_n\in G_0^*}
\Big( \llangle\frac{\phi_{\si_0,\alpha_1}}{z_1-\hat{\psi}_1},
    \frac{\phi_{\si_0,\alpha_2}}{z_2-\hat{\psi}_2}, \dots, \frac{\phi_{\si_0,\alpha_n}}{z_n-\hat{\psi}_n}\rrangle_{g,n}^{\cX,\bT_f}\Big|_{t=\btau}\Big)
    \prod_{j=1}^n\txi^{\alpha_j}(z_j,\tX_j).
\end{eqnarray*}
\end{enumerate}
\end{proposition}

\begin{remark}
$\tF_{0,2}^{\cX,(\cL,f)}(0;\tX_1,\tX_2)$ is an $H^*(\cB\bmu_\fm;\bC)^{\otimes 2}$-valued
power series in $\tX_1,\tX_2$ which vanishes at $(\tX_1,\tX_2)=(0,0)$,
so it is determined by \eqref{eqn:annulus-zero}.
\end{remark}

We now combine Section \ref{sec:Agraph} and the above Proposition \ref{prop:open-descendant} to obtain
a graph sum formula for $\tF_{g,n}^{\cX,(\cL,f)}$.  We use the notation in Section \ref{sec:Agraph},
and introduce the notation
\[
\txi^\bsi(z,\tX)=
\begin{cases}
\txi^\alpha(z,\tX), &\text{ if
    $\bsi=(\si_0,\alpha)$},\\
0,& \text{ if $\bsi=(\si,\alpha)$ and $\si\neq \si_0$}.
\end{cases}
\]

\begin{itemize}
\item Given a labeled graph $\vGa\in \bGa_{g,n}(\cX)$, to each ordinary leaf $l_j\in L^o(\Gamma)$
with $\bsi(l_j)=\bsi \in I_\Si$ and $k(l_j)\in \bZ_{\geq 0}$ we assign the following weight (open leaf)
\begin{equation}\label{eqn:O-leaf}
(\tcL^O)^{\bsi}_k(l_j) = [z^k]\big(\sum_{\brho,\bsi' \in I_\Si}
\Big( \txi^{\bsi'}(z,\tX_j) S^{\widehat{\underline{\brho}}}_{\spa  \bsi'}\bigg|_{\substack{t=\btau\\ \sw_i=w_i\sv}}
\Big)_+ R(-z)_{\brho}^{\spa \bsi}\bigg|_{\substack{t=\btau\\\sw_i=w_i \sv}} \big).
\end{equation}

\item Given a labeled graph $\bGa_{g,n}(\cX)$, we define a weight
\begin{eqnarray*}
 \tw_A^O(\vGa) &=& \prod_{v\in V(\Ga)} (\sqrt{\Delta^{\bsi(v)}(t)}\bigg|_{\substack{t=\btau\\\sw_i=w_i\sv} })^{2g(v)-2+\val(v)} \langle \prod_{h\in H(v)} \tau_{k(h)}\rangle_{g(v)} \\
&& \cdot \Big( \prod_{e\in E(\Ga)} \cE^{\bsi(v_1(e)),\bsi(v_2(e))}_{k(h_1(e)),k(h_2(e))}\cdot
\prod_{l\in L^1(\Ga)}(\cL^1)^{\bsi(l)}_{k(l)}\Big)\bigg|_{\substack{t=\btau\\ \sw_i = w_i \sv} } \prod_{j=1}^n(\tcL^O)^{\bsi(l_j)}_{k(l_j)}(l_j).\\
\end{eqnarray*}
\end{itemize}

Then we have the following graph sum formula for $\tF_{g,n}^{\cX,(\cL,f)}$.
\begin{theorem}\label{tF-graph-sum}
$$
\tF_{g,n}^{\cX,(\cL,f)} = \sum_{\vGa\in \bGa_{g,n}(\cX)}\frac{\tw_A^O(\vGa)}{|\Aut(\vGa)|}.
$$
\end{theorem}
\begin{proof}
This follows from Theorem \ref{thm:Zong} and Proposition \ref{prop:open-descendant}.
\end{proof}

\begin{definition}[Restriction to $Q=1$]
  \label{def:Fgn}
We define
$$
F_{g,n}^{\cX,(\cL,f)}(\btau; \tX_1,\ldots, \tX_n):=
\tF_{g,n}^{\cX,(\cL,f)}(\btau, 1;\tX_1,\ldots,\tX_n).
$$
\end{definition}

\noindent
By Remark \ref{descendent Q} and Proposition \ref{prop:open-descendant},
$F_{g,n}^{\cX,(\cL,f)}$ is well-defined. When $n=0$, it does not depend on the brane $(\cL,f)$, and we denote $F^\cX_g=F^\cX_{g,0}$. Theorem \ref{tF-graph-sum} implies
\begin{corollary}
$$
F_{g,n}^{\cX,(\cL,f)} = \sum_{\vGa\in \bGa_{g,n}(\cX)}\frac{w_A^O(\vGa)}{|\Aut(\vGa)|}.
$$
where $w_A^O(\vGa)= \tw_A^O(\vGa)\big|_{Q=1}$.
\end{corollary}

\section{Hori-Vafa mirror, Landau-Ginzburg mirror, and the mirror curve} \label{sec:HV-LG}

\subsection{Notation} \label{sec:H}
A pair $(\tau,\si)$ is called a flag if $\tau\in \Si(2)$, $\si\in \Si(3)$, and $\tau\subset \si$. Given a flag $(\tau,\si)$, there is
a short exact sequence of finite abelian groups
$$
1\to G_\tau \to G_\si \to \bmu_{\fr(\tau,\si)} \to 1
$$
where $G_\tau \cong \bmu_{\fm_\tau}$ for some positive integer $\fm_\tau$.
The flag $(\tau,\si)$ determines an ordered triple
$(i_1,i_2,i_3)$, where $i_1,i_2, i_3\in \{1,\ldots, \fp'+3\}$, characterized by the following three conditions (i) $I_\si' = \{ i_1, i_2, i_3\}$, (ii) $I_\tau'=\{ i_2, i_3\}$, and (iii)
 $(m_{i_1}, n_{i_1})$, $(m_{i_2}, n_{i_2})$, $(m_{i_3}, n_{i_3})$ are vertices of a triangle $P_\si$ in $\bR^2$ in counterclockwise order.
 Note that the 3-cone $\si$ is the cone over the triangle $P_\si\times \{1\}$.
 There exists an ordered  $\bZ$-basis $(e_1^{(\tau,\si)}, e_2^{(\tau,\si)},
e_3^{(\tau,\si)})$ of $N$ such that
$$
b_{i_1} = \fr(\tau,\si) e_1^{(\tau,\si)}   -\fs(\tau,\si)e_2^{(\tau,\si)}  + e_3^{(\tau,\si)} ,\quad
b_{i_2} = \fm_\tau e_2^{(\tau,\si)}   + e_3^{(\tau,\si)},\quad
b_{i_3} = e_3^{(\tau,\si)},
$$
where $\fs(\tau,\si) \in \{0,1,\ldots, \fr(\tau,\si)-1\}$. Then
$\{ e_1^{(\tau,\si)}, e_2^{(\tau,\si)} \}$ is a $\bZ$-basis of $\bZ e_1\oplus \bZ e_2$, and $e_1^{(\tau,\si)}\wedge
e_2^{(\tau,\si)} = e_1\wedge e_2$.
Define integers $a(\tau,\si), b(\tau,\si), c(\tau,\si), d(\tau,\si)$ by
$$
e_1^{(\tau,\si)} = a(\tau,\si) e_1 + b(\tau,\si) e_2,\quad
e_2^{(\tau,\si)} = c(\tau,\si)e_1 + d(\tau,\si) e_2.
$$
Then $a(\tau,\si) d(\tau,\si)-b(\tau,\si)c(\tau,\si)=1$.

For $i=1,\ldots, \fp+3$, we define $(m_i^{(\tau,\si) }, n_i^{(\tau,\si)})$ by
$$
b_i = m_i^{(\tau,\si)} e_1^{(\tau,\si)}  + n_i^{(\tau,\si)} e_2^{ (\tau,\si)}  + e_3^{(\tau,\si)}
$$
Note that $m_i^{(\tau,\si)}$ and $n_i^{(\tau,\si)}$ are determined by $\{ (m_i,n_i): i=1,\ldots,\fp+3\}$ and $(\tau,\si)$.

Given any 3-cone $\si\in \Si(3)$ and $i\in  \{1,\ldots, \fp+3\}$, we define a monomial $a^{\si}_i(q)$ in $q=(q_1,\ldots,q_\fp)$ as follows:
$$
a_i^\si(q):=
\begin{cases}
1, & i\in I'_\si,\\
\prod_{a=1}^{\fp} q_a^{s^\si_{ai}}, & i\in I_\si
\end{cases}
$$
where $s^\si_{ai}$ are non-negative integers defined by Equation \eqref{eqn:s}.  Observe that:
\begin{enumerate}
\item [(a)] In the large complex structure limit $q_K\to 0$,
$$
\lim_{q_K\to 0} a_i^{\si}(q)=0 \textup{ if } b_i \notin \si.
$$
\item [(b)] If $\fp'+1\leq a\leq \fp$ then $s_{ai}^\si = \delta_{i, a+3}$. So
$$
\lim_{q_{\mathrm{orb}}\to 0} a_i^{\si}(q)=0 \textup{ if }  i> \fp'+3.
$$
\end{enumerate}

Given a flag $(\tau,\si)$, we define
$$
H_{(\tau,\si)} (\XX , \YY, q):= \sum_{i=1}^{\fp+3} a_i^\si(q) (\XX) ^{m_i^{(\tau,\si)}} (\YY)^{n_i^{(\tau,\si)}}
$$
which is an element in the ring  $\bZ[q_1,\ldots,q_\fp] [ \XX, (\XX)^{-1}, \YY, (\YY)^{-1} ]$. Then
\begin{equation}\label{eqn:untwisted-affine}
H_{(\tau,\si)}(\XX, \YY,0) = (\XX) ^{\fr(\tau,\si)} (\YY) ^{-\fs(\tau,\si)} + (\YY)^{\fm_\tau} +1.
\end{equation}

In Section \ref{sec:av-branes}, by choosing the Lagrangian $\cL$, we fix a preferred flag $(\tau_0,\si_0)$, and choose $b_i$  such that  $I'_{\si_0}=\{1,2,3\}$ and
$I'_{\tau_0}=\{2,3\}$. Then
$$
\fr = \fr(\tau_0,\si_0),\quad \fm=\fm_{\tau_0}, \quad m_i = m_i^{(\tau_0,\si_0)}, \quad n_i = n_i^{(\tau_0,\si_0)}.
$$
Define $a_i(q):= a_i^{\si_0}(q)$, and define
$$
H(X,Y,q):=  \sum_{i=1}^{\fp+3} a_i(q) X^{m_i} Y^{n_i} = X^{\fr} Y^{-\fs} + Y^{\fm} + 1 + \sum_{a=1}^\fp a_{3+a}(q) X^{m_{3+a}} Y^{m_{3+a}}
$$
which is an element in the ring $\bZ[q_1,\ldots,q_{\fp}][ X, X^{-1}, Y, Y^{-1}]$.

\subsection{The mirror curve and its compactification} \label{sec:mirror-curve}
The mirror curve of $\cX$ is
\[
C_q=\{(X,Y)\in (\bC^*)^2: H(X,Y,q)=0\}.
\]
For fixed $q\in \bC^{\fp}$, $C_q$ is an affine curve in
$(\bC^*)^2$. Note that
$$
C_q\cong \{ (X_{(\tau,\si)}, Y_{(\tau,\si)})\in (\bC^*)^2: H_{(\tau,\si)}(X_{(\tau,\si)}, Y_{(\tau,\si)},q)=0\}
$$
for any flag $(\tau, \si) \in \Si(3)$. More explicitly, for fixed $q\in (\bC^*)^{\fp}$ the isomophism is induced by the following reparametrization of
$(\bC^*)^2$ (we use the notation in Section \ref{sec:H}):
\begin{eqnarray}
X_{(\tau,\si)} &=& X^{a(\tau,\si)} Y^{b(\tau,\si)} a_{i_1}(q)^{w_1(\tau,\si)} a_{i_2}(q)^{w_2(\tau,\si)} a_{i_3}(q)^{w_3(\tau,\si)}, \label{eqn:Xab} \\
Y_{(\tau,\si)} &=& X^{c(\tau,\si)} Y^{d(\tau,\si)} \left(\frac{a_{i_2}(q)}{a_{i_3}(q)}\right)^\frac{1}{\fm_\tau}, \label{eqn:Ycd}
\end{eqnarray}
where
\begin{equation}\label{eqn:ccc}
w_1(\tau,\si)=\frac{1}{\fr(\tau,\si)},\quad
w_2(\tau,\si)= \frac{\fs(\tau,\si)}{\fr(\tau,\si)m_\tau},  \quad
w_3(\tau,\si) = -w_1(\tau,\si)-w_2(\tau,\si).
\end{equation}
Under the above change of variables, we have
\begin{equation}\label{eqn:HH}
H(X,Y,q) = a_{i_3}(q) X^{m_{i_3}} Y^{m_{i_3}}H_{(\tau,\si)}(X_{(\tau,\si)}, Y_{(\tau,\si)}, q).
\end{equation}

The polytope $P$ determines a toric surface
$\bS_P$ with a polarization $L_P$, and $H(X,Y,q)$
extends to a section $s_q\in H^0(\bS_P,L_P)$. The compactified
mirror curve $\Cbar_q\subset \bS_P$ is the zero locus of $s_q$.
For generic $q\in \bC^{\fp}$, $\Cbar_q$ is a compact
Riemann surface of genus $\fg$ and $\Cbar_q$ intersects
the anti-canonical divisor $\partial \bS_P=\bS_P\setminus (\bC^*)^2$ transversally
at $\fn$ points, where $\fg$ and $\fn$ are the number of lattice
points in the interior and the boundary of $P$, respectively.

\subsection{Three mirror families}
\label{sec:three-mirror}
The symplectic toric orbifold $(\cX,\omega(\eta))$ has three mirror families:
\begin{itemize}
\item The Hori-Vafa mirror  $(\check{\cX}_q,\Omega_q)$, where
\begin{equation}
\check{\cX}_q=\{ (u,v,X,Y)\in \bC\times\bC\times \bC^*\times
\bC^*: uv=H(X,Y,q)\}
\label{eqn:mirror-3-fold}
\end{equation}
is a non-compact Calabi-Yau 3-fold, and
$$
\Omega_q :=\Res_{\check{\cX}_q}\Big(
\frac{1}{H(X,Y,q)-uv} du\wedge dv \wedge \frac{dX}{X}\wedge
\frac{dY}{Y}\Big)
$$
is a holomorphic 3-form on $\check{\cX}_q$.

\item The $\bT'$-equivariant Landau-Ginzburg mirror
$( (\bC^*)^3, W_q^{\bT'})$, where
\begin{equation}
\label{eqn:mirror-LG}
W^{\bT'}_q(X,Y,Z) = H(X,Y,q)Z -\su_1\log X-\su_2\log Y
\end{equation}
is the $\bT'$-equivariant superpotential.

\item The mirror curve $C_q = \{ (X,Y)\in (\bC^*)^2: H(X,Y,q)=0\}$ and its compactification $\overline{C}_q$.
 \end{itemize}

We define
\begin{equation}
  \label{eqn:U-ep}
 U_\ep:= \{ q=(q_1,\ldots, q_{\fp}) \in (\bC^*)^{\fp'}\times \bC^{\fp-\fp'}:  |q_a|<\ep \}.
\end{equation}
We choose $\ep$ small enough such that $C_q,\Cbar_q$ and $\check \cX_q$ are smooth.

In Section \ref{sec:HV} (resp. Section \ref{sec:LG}) below, we will reduce the genus zero B-model
on the  Hori-Vafa mirror (resp. equivariant Landau-Ginzburg mirror) to a theory on the mirror curve.

\subsection{Dimensional reduction of the Hori-Vafa mirror}\label{sec:HV}

In this subsection, we describe the precise relations among the following 3-dimensional, 2-dimensional, and 1-dimensional integrals when $q\in U_\epsilon$:
\begin{itemize}
\item[(3d)] period integrals of the holomorphic 3-form $\Omega_q$ over 3-cycles in
the Hori-Vafa mirror $\check{\cX}_q$,
\item[(2d)] integral of the holomorphic 2-form $\frac{dX}{X}\wedge \frac{dY}{Y}$ on $(\bC^*)^2$
over relative 2-cycles of the pair $( (\bC^*)^2, C_q)$, and
\item[(1d)] integrals of a Liouville form along 1-cycles in the mirror curve $C_q$,
\end{itemize}
The references of this subsection are \cite{DK11} and \cite{CLT13}; see also \cite{KM11}.

The inclusion $J: C_q\to \Cbar_q$ induces a surjective homomorphism
$$
J_*: H_1(C_q;\bZ)\cong \bZ^{\oplus{2\fg +\fn-1}} \to H_1(\Cbar_q;\bZ)\cong \bZ^{\oplus 2\fg}.
$$
Let $J_1(C_q;\bZ)$ denote the kernel of the above map. Then $J_1(C_q;\bZ)\cong \bZ^{\oplus (\fn-1)}$
is generated by  $\delta_1, \ldots, \delta_{\fn}$, where $\delta_i \in H_1(C_q;\bZ)$
is the class of a small loop around the puncture $\bar{p}_i$. They satisfy
$$
\delta_1 +\cdots + \delta_\fn = 0.
$$

The inclusion $I:C_q\to (\bC^*)^2$ induces a homomorphism
$$
I_*: H_1(C_q;\bZ)\cong \bZ^{2\fg+\fn-1} \to H_1( (\bC^*)^2;\bZ)=\bZ^2
$$
whose cokernel is finite (but not necessarily trivial).
Let $K_1(C_q;\bZ)\cong \bZ^{2\fg+\fn-3}$ denote the kernel of the above map.

For any flag $(\tau,\si)$,  let
$\alpha_{(\tau,\si)}$ and $\beta_{(\tau,\si)}$ be classes in $H^1((\bC^*)^2;\bC)$ represented
by the closed 1-forms
$$
\frac{dX_{(\tau,\si)} }{2\pi\sqrt{-1}X_{(\tau,\si)}}  \quad \textup{and} \quad
\frac{dY_{(\tau,\si)} }{2\pi\sqrt{-1} Y_{(\tau,\si)} }
$$
on $(\bC^*)^2$, respectively. Then $\alpha_{(\tau,\si)}$ and $\beta_{(\tau,\si)}$ lie in $H^1((\bC^*)^2;\bZ)\subset H^1((\bC^*)^2;\bC)$
and form a $\bZ$-basis of $H^1((\bC^*)^2,\bZ)\cong \bZ^2$.  Define
\begin{equation}\label{eqn:Mx}
\fM_{\XX}: H_1(C_q;\bZ)\lra \bZ,  \quad \gamma \mapsto   \langle \alpha_{(\tau,\si)}, I_* \gamma\rangle,
\end{equation}
\begin{equation}\label{eqn:My}
\fM_{\YY} : H_1(C_q;\bZ)\lra \bZ,  \quad \gamma \mapsto \langle  \beta_{(\tau,\si)}, I_* \gamma\rangle,
\end{equation}
where $\langle \ \ , \ \  \rangle: H^1( (\bC^*)^2;\bZ)\times H_1((\bC^*)^2;\bZ)\lra \bZ$ is the natural pairing.

Let $K_{\XX}(C_q;\bZ)$ and $K_{\YY}(C_q;\bZ)$ denote the kernels of  \eqref{eqn:Mx} and \eqref{eqn:My}, respectively. Then they
are isomorphic to $\bZ^{2\fg-2+\fn}$, and
$$
K_{\XX}(C_q;\bZ) \cap K_{\YY} (C_q;\bZ) = K_1(C_q;\bZ).
$$
Let
$\fM_X=\fM_{X_{(\tau_0,\si_0)}}$,
$\fM_Y= \fM_{Y_{(\tau_0,\si_0)}}$,
$K_X =K_{X_{(\tau_0,\si_0)}}$,
$K_Y=K_{Y_{(\tau_0,\si_0)}}$.

\medskip

Let $J_1(C_q;\bQ):= J_1(C_q;\bZ)\otimes_{\bZ}\bQ$ and $K_1(C_q;\bQ) := K_1(C_q;\bZ)\otimes_{\bZ}\bQ$. Then
we have the following diagram:
\begin{equation}
\begin{CD}
& & & &  0 \\
& & & & @VVV \\
& & & & K_1(C_q;\bQ) \\
& & & & @V{\iota}VV \\
0 @>>> J_1(C_q;\bQ) @>>> H_1(C_q;\bQ) @>{J_*}>> H_1(\overline{C}_q;\bQ) @>>> 0 \\
& & & & @V{I_*}VV \\
& & & &  H_1((\bC^*)^2;\bQ) \\
& & & & @VVV \\
& & & & 0
\end{CD}
\end{equation}
In the above diagram, the row and the column are short exact sequences of
vector spaces over $\bQ$. Let $\tC_q$ be the fiber product
of the inclusion $I:C_q\to (\bC^*)^2$ and the universal cover
$\bC^2\to (\bC^*)^2$. Then $p: \tC_q \to C_q$ is a regular covering
with fiber $\bZ^2$, and there is an injective group homomorphism
$$
p_*: H_1(\tC_q;\bQ)\to H_1(C_q;\bQ)
$$
whose image is $K_1(C_q;\bQ)$.

Since $I_*(J_1(C_q;\bQ))=H_1((\bC^*)^2;\bQ)$
(i.e. $I_*|_{J_1(C_q;\bQ)}$ is surjective),
$K_1(C_q;\bQ)+J_1(C_q;\bQ)=H_1(C_q;\bQ)$. Then $J_*\circ \iota$ is
surjective, and we can lift any element $\gamma \in H_1(\overline
C_q;\bQ)$ to $K_1(C_q;\bQ)$.

The long exact sequence of relative homology for the
pair $( (\bC^*)^2, C_q)$ is
$$
\begin{array}{ccccccccc}
\cdots &\to & H_2(C_q;\bZ) &\to &H_2((\bC^*)^2;\bZ) &\to & H_2( (\bC^*)^2, C_q;\bZ)&& \\
&\to & H_1(C_q;\bZ) &\to& H_1((\bC^*)^2;\bZ) &\to& H_1( (\bC^*)^2, C_q;\bZ) &\to& \cdots
\end{array}
$$
where $H_2(C_q;\bZ)=0$. So we have a short exact sequence
$$
0\to H_2((\bC^*)^2;\bZ) \to H_2((\bC^*)^2, C_q;\bZ) \stackrel{\partial}{\to} K_1(C_q;\bZ)\to 0.
$$
Varying $q$ in $U_\ep$, we obtain the following short exact sequence of  local systems of lattices over $U_\ep$:
$$
0\to \underline{\bZ} \to \bH_{\bZ} \to \bK_{\bZ} \to 0
$$
where $\underline{\bZ}$ is the trivial $\bZ$-bundle over $U_\ep$, and  the fibers of $\bH_{\bZ}$ and $\bK_{\bZ}$ over $q\in U_\ep$ are
$H_2((\bC^*)^2,C_q;\bZ)$ and $K_1(C_q;\bZ)$, respectively. Tensoring with $\bC$, we obtain the following  short exact sequence of
flat complex vector bundles over $U_\ep$:
$$
0\to \underline{\bC} \to \bH \to \bK \to 0
$$
where $\underline{\bC}$ is the trivial complex line bundle, and the fibers of $\bH$ and $\bK$ over $q\in U_\ep$ are
$H_2((\bC^*)^2,C_q,\bC)$ and $K_1(C_q;\bC)$, respectively.
Similarly, let $\tbK$ and $\tbH$ (resp. $\tbK_\bZ$ and $\tbH_\bZ$) be the flat bundles (resp. local systems of lattices) over $U_\ep$ whose fibers over $q$ are $H_1(\tC_q;\bC)$ and $H_2(\bC^2,\tC_q;\bC)$ (resp. $H_1(\tC_q;\bZ)$ and $H_2(\bC^2,\tC_q;\bZ)$) respectively. The ranks of $\bH$ and $\bK$ are $2\fg+\fn-1$ and $2\fg+\fn-2$ respectively, while the ranks of $\tbH$ and $\tbK$ are infinite. We have the following commutative diagrams of homology groups and bundle maps
\begin{gather*}
\begin{CD}
  0 @>>> 0=H_2(\bC^2;\bC) @>>>  H_2(\bC^2,\tC_q;\bC) @>{\tilde \partial}>{\cong}> H_1(\tC_q;\bC) @>>> 0\\
      &    &      @VVV           @VVV               @VV{p_*}V\\
  0@>>> H_2((\bC^*)^2;\bC)@>>> H_2((\bC^*)^2,C_q;\bC) @>{\partial} >> K_1(C_q;\bC) @>>> 0,
\end{CD}\\
\begin{CD}
  0 @>>> 0 @>>>  \tbH @>{\tilde \partial}>{\cong}> \tbK @>>> 0\\
      &    &      @VVV           @VVV               @VV{p_*}V\\
  0@>>> \underline \bC@>>> \bH @>{\partial} >> \bK @>>> 0
\end{CD}
\end{gather*}
where $p_*$ is surjective.

Let $\tp: \tU_\ep  \to U_\ep$ be the universal cover of $U_\ep$. Then the coordinates on $\tU_\ep$ are $(\log q_1,\ldots, \log q_{\fp'}, q_{\fp'+1},\ldots, q_\fp)$, and
$\tp^*\bH_{\bZ}$, $\tp^*\bK_{\bZ}$, $\tp^*\tbH_\bZ$ and $\tp^* \tbK_\bZ$ are trivial local systems of lattices over $\tU_\ep$. We say a section of these flat bundles is constant if it is flat w.r.t. the Gauss-Manin connection.

Let $x = -\log X$ and $y=-\log Y$. Then
$$
\omega:=dx \wedge dy =\frac{dX}{X}\wedge \frac{dY}{Y}
$$
is the standard holomorphic symplectic form on $(\bC^*)^2$. Note that $\omega|_{C_q}=0$, so
$\omega$ represents a class in $H^2((\bC^*)^2,C_q;\bC)$. The inclusion $T^2= (S^1)^2 \subset (\bC^*)^2$ is a homotopy equivalence, so
$H_2((\bC^*)^2;\bZ)\cong H_2(T^2;\bZ)=\bZ[T^2]$. We have
$$
\int_{[T^2]}\omega = (2\pi\sqrt{-1})^2.
$$

For $q\in U_\ep$, define $\mu: \check{\cX}_q\to \bR$ by $(u,v,X,Y)\mapsto |u|^2-|v|^2$
and $\pi: \check \cX_q\to (\bC^*)^2$ by $(u,v,X,Y)\mapsto (X,Y)$. Then
$\pi:\mu^{-1}(0)\to (\bC^*)^2$ is an circle fibration which degenerate along $C_q\subset (\bC^*)^2$.
Given a relative 2-cycle  $\Lambda \in Z_2( (\bC^*)^2,C_q)$, $\pi^{-1}(\Lambda)\cap \mu^{-1}(0)$ 
is a 3-cycle in $\check{\cX}_q$. The map $\Lambda\mapsto \pi^{-1}(\Lambda)\cap \mu^{-1}(0)$ induces
a group homomorphism $M: H_2((\bC^*)^2,C_q;\bZ) \to H_3(\check{\cX}_q;\bZ)$. Let  $\bH'$ (resp. $\bH'_\bZ$) be the flat complex vector bundle
(resp. local system of lattices) over $U_\ep$ whose fiber over $q\in U_\ep$ is  $H_3(\check \cX_q;\bC)$ (resp. $H_3(\check \cX_q;\bZ)$).

\begin{lemma}[{\cite[Section 5.1]{DK11} \cite[Section 4.2]{CLT13}}]
The map $M: H_2((\bC^*)^2,C_q;\bZ)\to H_3(\check{\cX}_q;\bZ)$ is an isomorphism,
$$
\int_{M(\Lambda)}\Omega_q  = \int_{M(\Lambda)} \frac{du}{u}\wedge
\frac{dX}{X}\wedge\frac{dY}{Y}
= 2\pi\sqrt{-1}\int_\Lambda \omega
$$
for any $\Lambda\in H_2((\bC^*)^2;C_q;\bQ)$.
In particular,
$$
\int_{M([T^2])} \Omega_q = (2\pi\sqrt{-1})^3.
$$
\end{lemma}
In particular, this lemma gives an isomorphism between $\bH$ and $\bH'$. Given a flat section $\bold \Gamma$ of $\tp^*\tbK$ (resp. $\tp^*\bH$,  $\tp^*\bH'$), let
\begin {equation}\label{eqn:period}
\int_{\bold\Gamma} ydx, \quad  \big(\textup{resp.}
\int_{\bold\Gamma}\omega, \ \int_{\bold\Gamma} \Omega_q \big)
\end{equation}
denote the paring of $\bold\Gamma$ and
$[ydx]\in H^1(\tC_q,\bC)$, (resp. $[\omega]\in H^2((\bC^*)^2,C_q;\bC)$, $[\Omega_q]\in H^3(\check{\cX}_q;\bC)$).
The integrals in  \eqref{eqn:period} are holomorphic functions on $\tU_\ep$, so $[ydx]$ (resp. $[\omega]$, $[\Omega_q]$)
can be viewed as a holomorphic (but non-flat) section of the dual vector bundle $\tp^*\tbK^\vee$  (resp. $\tp^*\bH^\vee$, $\tp^*(\bH')^\vee$) over
$\tU_\ep$.


Let $\tA'_0:= M([T^2])\in H_3(\check \cX_q;\bZ)$, so that
$$
\frac{1}{(2\pi\sqrt{-1})^3} \int_{\tA'_0}\Omega_q = 1.
$$
We regard $\tA_0'$ as a flat section of $\bH'$ (and also of $\tp^* \bH'$), since $[T^2]$ is a flat section of $\bH$.
\begin{proposition}
There exist $\tA'_1,\ldots, \tA'_{\fp}, \tB'_1,\ldots, \tB'_{\fg}$ as flat sections
of $\tp^*\bH'$ such that
$$
\frac{1}{(2\pi\sqrt{-1})^2}\int_{\tA'_a} \Omega_q = T^a(q),\quad a=1,\ldots, \fp,
$$
$$
\frac{1}{2\pi\sqrt{-1}} \int_{\tB'_i}\Omega_q = W_i(q),\quad i=1,\ldots,\fg,
$$
on $\tU_{\epsilon}$ for sufficiently small $\epsilon>0$.
\end{proposition}
\begin{proof} This follows from Proposition \ref{prop:PF}, and the fact that
solutions to the non-equivariant Picard-Fuchs system \eqref{eqn:PF}
are period integrals of $\Omega_q$.
\end{proof}

\begin{remark}
The existence of $\tA'_1,\ldots,\tA'_{\fp}$ follows from \cite[Theorem 1.6]{CCLT}, which covers semi-projective toric Calabi-Yau orbifolds of any dimension.
\end{remark}


The image of $[ydx]\in H^1(\tC_q;\bC)$ under the isomorphism $H^1(\tC_q;\bC)\stackrel{\cong}{\to} H^2(\bC^2,\tC_q;\bC)$
is the relative class $[dy\wedge dx] = -[dx\wedge dy] \in H^2(\bC^2,\tC_q;\bC)$. So for any flat section $\tD$ of $\tp^*\bH$ we have
$$
\int_{\tD} dx\wedge dy = - \int_{\tilde \partial \tD} ydx.
$$
For any flat section $\gamma$ of $\tp^* \bK$, let $\tilde \gamma_1,\tilde \gamma_2$ be two flat sections of $\tp^* \tbK$ with $p_* \tilde \gamma_i=\gamma$.
For $i\in \{1,2\}$, there exists a unique flat section $\tD_i$ of $\tp^* \tbH$ such that $\tilde \partial \tD_i =\tilde \gamma_i$. Then
\[
-\int_{\tilde\gamma_1 -\tilde\gamma_2} ydx = \int_{\tD_1-\tD_2} dx\wedge dy = \int_{D_1-D_2} \frac{dX}{X}\wedge \frac{dY}{Y}
\]
where flat sections $D_1$ and $D_2$ are the images of $\tD_1$ and $\tD_2$ under $\tp^* \tbH\to \tp^* \bH$, respectively.  Since $\partial(D_1-D_2)=0$ we must have
$D_1 -D_2 = c[T^2]$ for some $c\in \bC$ -- the flatness of $D_i$ and $[T^2]$ implies $c$ is a constant.

Based on the above discussion, we have:
\begin{lemma} \label{lemm:integral-K}
Given $\gamma$ as  a flat section of $\tp^* \bK$, and let flat section $\Gamma$ (resp. $\tilde\gamma$) of $\tp^* \bH$ (resp. $\tp^* \tbK$) be in the preimage $\gamma$ under the surjective  map bundle map $\tp^* \bH \to \tp^* \bK$ (resp. $\tp^* \tbK \to \tp^* \bK$). Then
$$
\frac{1}{(2\pi\sqrt{-1})^2 } \int_{\tilde{\gamma}} ydx = \frac{-1}{(2\pi\sqrt{-1})^2} \int_{\Gamma}\omega +c
$$
for some constant $c\in \bC$. Moreover, if $\gamma$ is a section of $\tp^* \bK_\bZ$ and we choose $\Gamma$, $\tilde \gamma$ to be sections of $\tp^* \bH_\bZ$ and $\tp^* \tbK_\bZ$, then
the constant $c$ is an integer.
\end{lemma}

\begin{remark}  \label{rem:integral-K}
Lemma \ref{lemm:integral-K} says the integration of $ydx$ over the flat cycles (sections) of $\tp^* \bK$ is well-defined up to a constant -- the constant ambiguity comes from the choice of lift in $\tp^*\bH$ (or $\tp^* \tbK$). For a given lift one may always choose another such that the difference of the integrations could be any constant $c\in\bC$. Similarly, the integration of $ydx$ over flat cycles of $\bK$ is also well-defined up to a constant -- it is defined as the integration of $ydx$ over its lift in $\tbK$, or equivalently as the integration of $\omega$ over its lift in $\bH$. Given such a lift one may find another such that the difference of their integrals is an arbitrary constant.
\end{remark}

\subsection{The equivariant small quantum cohomology}
\label{sec:qcoh}

Let $H(X,Y,q)$ be defined as in Section \ref{sec:H}. The
$\bT$-equivariant Landau-Ginzburg mirror of $\cX$ is
$( (\bC^*)^3, W_q^{\bT})$, where
$$
W^{\bT}_q(X,Y,Z)=H(X,Y,q)Z -\su_1\log X-\su_2\log Y-
\su_3\log Z
$$
Consider the universal superpotential $W^{\bT}(X,Y,Z,q)=W^{\bT}_q(X,Y,Z)$ defined
on $(\bC^*)^3\times \bC^{\fp}$. Then
$$
\mathrm{Jac}(W^{\bT}):= \frac{\bST[q_a, q_a^{-1},X,X^{-1}, Y, Y^{-1}, Z, Z^{-1}]}
{\langle \frac{\partial W^{\bT}}{\partial X}, \frac{\partial W^{\bT}}{\partial Y},
\frac{\partial W^{\bT}}{\partial Z}\rangle }
$$
is an algebra over $\bST[q_a, q_a^{-1}]$. For each fixed $q = (q_1,\ldots,q_p) \in \bC^{\fp}$, we
obtain
$$
\mathrm{Jac}(W_q^{\bT}):= \frac{\bST[X,X^{-1}, Y, Y^{-1}, Z, Z^{-1}]}
{\langle \frac{\partial W_q^{\bT}}{\partial X}, \frac{\partial W_q^{\bT}}{\partial Y},
\frac{\partial W_q^{\bT}}{\partial Z}\rangle }
$$
which is an algebra over $\bST$. By the argument in \cite{TsWa}, Theorem \ref{thm:IJ} ($\bT$-equivariant
mirror theorem) implies the following isomorphism of $\bST$-algebras for $q\in U_\ep$ with small enough $\epsilon$:
\begin{align}
\label{eqn:QH=Jac}
QH_{\CR,\bT}^*(\cX)\big|_{\btau=\btau({q}), Q=1} &\stackrel{\cong}{\longrightarrow}
\mathrm{Jac}(W_q^{\bT}):= \frac{\bST[X,X^{-1}, Y, Y^{-1}, Z, Z^{-1}]}
{\langle \frac{\partial W_q^{\bT}}{\partial X}, \frac{\partial W_q^{\bT}}{\partial Y},
\frac{\partial W^{\bT}_q}{\partial Z}\rangle }\\
\nonumber
H_a &\mapsto  [\frac{\partial W^{\bT}_q}{\partial \tau_a}(X,Y,Z)]
\end{align}
where $QH_{\CR,\bT}^*(\cX)$ is the small $\bT$-equivariant quantum cohomology of $\cX$,
and $\btau(q)$ is the closed mirror map.

Under this isomorphism,  the $\bT$-equivariant Poincar\'{e} pairing
on $QH^*_{\CR,\bT}(\cX)|_{\btau=\btau(q),Q=1}$ corresponds to the residue pairing
on $\Jac(W^{\bT}_q)$. More precisely, for $q\in U_\ep$ with generic $\su_1,\su_2,\su_3$,
$W_q^{\bT}$ is (locally) a holomorphic Morse function, i.e.,
the holomorphic 1-form $dW_q^{\bT}: (\bC^*)^3 \to T^*(\bC^*)^3$
intersects the zero section  of the cotangent bundle $T^*(\bC^*)^3$ transversally. The canonical basis of $\Jac(W_q^\bT)$ is represented by functions taking value $1$ at one critical point while being zero at other critical points, so the set of the zeros of $dW_q^{\bT}$ is identified with the set $I_\Si=\{(\si,\gamma): \si\in \Si(3),\gamma\in G_\si^*\}$, the labels of the canonical basis of the quantum cohomology. Let
$\bp_{\bsi}\in (\bC^*)^3$ be the zero of $dW_q^{\bT}$
associated to $\bsi \in I_\Si$.
Then
$$
(f,g) := \frac{1}{(2\pi\sqrt{-1})^3}
\int_{|dW_q^{\bT}|=\epsilon} \frac{fg dx\wedge dy \wedge dz}
{\frac{\partial W_q^{\bT}}{\partial x}\frac{\partial W_q^{\bT}}{\partial y}
\frac{\partial W_q^{\bT}}{\partial z}}
=\sum_{\bsi \in I_\Si}
\frac{f(\bp_\bsi)g(\bp_\bsi)}{
\det(\Hess(W_q^{\bT}))(\bp_\bsi)},
$$
where $x=-\log X$, $y=-\log Y$, $z=-\log Z$, and
$$
\Hess(W^{\bT}_q) =\left(\begin{array}{ccc}
(W^{\bT}_q)_{xx} & (W^{\bT}_q)_{xy} & (W^{\bT}_q)_{xz}\\
(W^{\bT}_q)_{yx} & (W^{\bT}_q)_{yy} & (W^{\bT}_q)_{yz}\\
(W^{\bT}_q)_{zx} & (W^{\bT}_q)_{zy} & (W^{\bT}_q)_{zz}
\end{array}\right).
$$

To summarize, there is an isomorphisms of Frobenius algebras
$$
QH^*_{\CR,\bT}(\cX)\big|_{\btau=\btau(q), Q=1}\cong \Jac(W^{\bT}_q)
$$
with
\begin{equation}\label{eqn:Delta-Hess}
\Delta^{\bsi}(\btau)\big|_{\btau=\btau(q),Q=1}= \det\left(\Hess(W^{\bT}_q)(\bp_{\bsi})\right)
\end{equation}
under the closed mirror map
\eqref{eqn:closed-mirror-map}. Setting $\su_3=0$, we have
\begin{equation}\label{eqn:qcoh-Tprime}
QH^*_{\CR,\bT'}(\cX)\big|_{\btau=\btau(q),Q=1}\cong \Jac(W^{\bT'}_q).
\end{equation}

\subsection{Dimensional reduction of the equivariant Landau-Ginzburg model}
\label{sec:LG}
In this subsection, we will see that the 3d Landau-Ginzburg B-model $((\bC^*)^3, W_q^{\bT'})$ is equivalent a 1d Landau-Ginzburg B-model
$(C_q,\hat{x})$ in two ways: (1) isomorphism of Frobenius algebras, and (2) identification of oscillatory integrals.
\medskip

\paragraph{\bf (1) Isomorphism of Frobenius algebras}
The $\bT'$-equivariant mirror of $\cX$ is a Landau-Ginzburg model
$( (\bC^*)^3, W_q^{\bT'})$, where
$W_q^{\bT'}: (\bC^*)^3 \to \bC$
is the $\bT'$-equivariant superpotential
\begin{equation}\label{eqn:WTprime}
W_q^{\bT'} =H(X,Y,q)Z - \su_1\log X- \su_2 \log Y
\end{equation}
which is multi-valued. Here we view $\su_1$ and $\su_2$ as complex parameters.
The differential
$$
dW^{\bT'}_q = \frac{\partial W_q^{\bT'}}{\partial X} dX
+\frac{\partial W_q^{\bT'}}{\partial Y} dY
+ \frac{\partial W_q^{\bT'}}{\partial Z} dZ
= Z dH + H dZ - \su_1\frac{dX}{X} -\su_2\frac{dX}{X}
$$
is a well-defined holomorphic 1-form on $(\bC^*)^3$.

We have
\begin{eqnarray*}
\frac{\partial W_q^{\bT'}}{\partial X}(X,Y,Z) &=& Z \frac{\partial H}{\partial X}(X,Y,q)
-\frac{\su_1}{X}\\
\frac{\partial W_q^{\bT'}}{\partial Y}(X,Y,Z) &=& Z\frac{\partial H}{\partial Y}(X,Y,q)
-\frac{\su_2}{Y}\\
\frac{\partial W_q^{\bT'}}{\partial Z}(X,Y,Z) &=& H(X,Y,q)
\end{eqnarray*}
Therefore,
$$
\frac{\partial W_q^{\bT'}}{\partial X}=0,\quad
\frac{\partial W_q^{\bT'}}{\partial Y}=0,\quad
\frac{\partial W_q^{\bT'}}{\partial Z}=0
$$
are equivalent to
$$
H(X,Y,q)=0,\quad \frac{\partial H}{\partial X}(X,Y,q) =
-\frac{1}{Z} \frac{\partial \hat{x}}{\partial X},
\quad \frac{\partial H}{\partial Y}(X,Y,q)= -\frac{1}{Z}\frac{\partial \hat{x}}{\partial Y}.
$$
where $\hat{x} = \su_1 x + \su_2 y$.
Therefore, the critical points of $W_q^{\bT'}(X,Y,Z)$, which are zeros
of the holomorphic differential $dW_q^{\bT'}$ on $(\bC^*)^3$,
can be identified with the critical points of $\hat{x}$, which
are zeros of the holomorphic differential
$$
d\hat{x} = -\su_1\frac{dX}{X} - \su_2 \frac{dY}{Y} = -\su_1(\frac{dX}{X}+\frac{\su_2}{\su_1}\frac{dY}{Y}).
$$
on the mirror curve
$$
C_q= \{ (X,Y)\in (\bC^*)^2: H(X,Y,q)=0\}.
$$
For $q\in U_\ep$, $C_q$ is a smooth Riemann surface
of genus $\fg$ with $\fn$ punctures. For fixed $q$, the zeros of $d\hat{x}$ depend only
on $f= \su_2/\su_1$.  For a fixed generic $f\in \bC$, there exists $\ep(f) \in (0,\ep)$ such that for
all $q\in U_{\ep(f)}$,
the section $d\hat{x}: C_q\to T^*C_q$ intersects
the zero section transversally at $2\fg-2+\fn$ points, and $W_q^{\bT'}$ is holomorphic Morse with $2\fg-2+\fn$ critical points. In the remainder of this section, we assume $f$ is generic and $q\in U_{\ep(f)}$.

We have the following isomorphism of $\bar \bS_{\bT'}$-algebras:

\begin{lemma}\label{lm:product}
$$
\Jac(W^{\bT'}_q)\cong H_B,
$$where
\begin{eqnarray*}
\Jac(W^{\bT'}_q)&:=& \frac{\bar{S}_{\bT'}[X, X^{-1}, Y, Y^{-1}, Z, Z^{-1}]}
{ \langle \frac{\partial W_q^{\bT'}}{\partial X},
\frac{\partial W_q^{\bT'}}{\partial Y},
\frac{\partial W_q^{\bT'}}{\partial X} \rangle},\\
H_B &:=&
\frac{\bar{S}_{\bT'}[X, X^{-1}, Y, Y^{-1}]}{\langle H(X,Y,q),
\su_2 X\frac{\partial H}{\partial X}(X,Y,q)
-\su_1 Y\frac{\partial H}{\partial Y}(X,Y,q)\rangle }.
\end{eqnarray*}
\end{lemma}

It is straightforward to check that, $(X_0, Y_0)$ is a solution to
$$
H(X,Y,q)= \su_2 X\frac{\partial H}{\partial X}(X,Y,q)
-\su_1 Y\frac{\partial H}{\partial Y}(X,Y,q) =0
$$
if and only if $\displaystyle{ (X_0, Y_0, \frac{\su_1}{X_0\frac{\partial H}{\partial X}(X_0,Y_0,q)})}$
is a solution to
$$
\frac{\partial W^{\bT'}_q}{\partial X} = \frac{\partial W_q^{\bT'}}{\partial Y}
=\frac{\partial W^{\bT'}_q}{\partial Z} =0.
$$
 Moreover, $\hat{y}:= y/\su_1$ is a local holomorphic coordinate
near $(X_0,Y_0)$, and
\begin{lemma}\label{lm:pairing}
$$
\det\Hess(W^{\bT'}_q)\left(X_0, Y_0,\frac{-\su_1}{\frac{\partial H}{\partial x}(X_0,Y_0,q)}\right)
= -\frac{d^2 \hat{x}}{d\hat{y}^2}.
$$
\end{lemma}

\begin{proof}
Since
$$
W^{\bT'}_q(X,Y,Z)= H(X,Y,q)Z-\su_1 \log X -\su_2 \log Y,
$$
we have
$$
\Hess(W^{\bT'}_q):= \left(
\begin{array}{ccc}
(W^{\bT'}_q)_{xx} & (W^{\bT'}_q)_{xy} & (W^{\bT'}_q)_{xz} \\
(W^{\bT'}_q)_{yx} & (W^{\bT'}_q)_{yy} & (W^{\bT'}_q)_{yz} \\
(W^{\bT'}_q)_{zx} & (W^{\bT'}_q)_{zy} & (W^{\bT'}_q)_{zz}
\end{array} \right)
=\left(
\begin{array}{ccc}
H_{xx} Z & H_{xy}Z & -H_x Z \\
H_{yx} Z & H_{yy}Z & -H_y Z \\
-H_x Z & -H_y Z & H Z
\end{array} \right),
$$
$$
\det(\Hess(W_q^{\bT'})) = Z^3 \det \left(
\begin{array}{ccc}
H_{xx} & H_{xy} & -H_x  \\
H_{yx}  & H_{yy} & -H_y  \\
-H_x  & -H_y  & H
\end{array} \right).
$$
Taking differential on both sides of
$$
H(X,Y,q)=0,
$$
$$
H_x \frac{dx}{dy} + H_y =0 \implies \frac{dx}{dy}= -\frac{H_y}{H_x}.
$$
Taking differential once again
$$
H_{xx} (\frac{dx}{dy})^2 + H_{xy} \frac{dx}{dy} + H_x \frac{d^2 x}{dy^2}
+ H_{yx}\frac{dx}{dy}+ H_{yy} =0
\implies H_{xx} \frac{H_y^2}{H_x^2} - 2H_{xy} \frac{H_y}{H_x} + H_x \frac{d^2 x}{dy^2}
+ H_{yy} =0.
$$
We conclude that
\begin{align*}
&\frac{d^2 x}{dy^2} = \frac{2H_{xy} H_x H_y - H_{xx}H_y^2 -H_{yy} H_x^2}{H_x^3}
= \frac{1}{H_x^3}\det \left(
\begin{array}{ccc}
H_{xx} & H_{xy} & -H_x  \\
H_{yx}  & H_{yy} & -H_y  \\
-H_x  & -H_y  & 0
\end{array} \right),\\
&\frac{d^2\hx}{d\hy^2} = (\frac{\su_1}{H_x})^3 \det \left(
\begin{array}{ccc}
H_{xx} & H_{xy} & -H_x  \\
H_{yx}  & H_{yy} & -H_y  \\
-H_x  & -H_y  & 0
\end{array} \right).
\end{align*}
Recall that $(X_0, Y_0)= (e^{-x_0}, e^{-y_0})$ satisfies
$$
H(X_0,Y_0,q)=0, \quad -\su_2 H_x(X_0,Y_0,q) +\su_1 H_y(X_0,Y_0,q)=0.
$$
So
$$
 \det(\Hess(W^{\bT'}_q))(X_0, Y_0, \frac{-\su_1}{H_x(X_0,Y_0,q)})\\
= (\frac{-\su_1}{H_x(X_0,Y_0,q)})^3 \det \left(
\begin{array}{ccc}
H_{xx} & H_{xy} & -H_x  \\
H_{yx}  & H_{yy} & -H_y  \\
-H_x  & -H_y  & H
\end{array} \right)(X_0, Y_0,q)\\
=\left. -\frac{d^2\hx}{d\hy^2}\right|_{\hy = \frac{y_0}{\su_1}}.
$$
\end{proof}

Combining Lemma \ref{lm:product} and Lemma \ref{lm:pairing},  we have an isomorphism $\Jac(W^{\bT'}_q)\cong H_B$ of Frobenius algebras.

\medskip

\paragraph{\bf (2) Identification of oscillatory integrals}

We first introduce some notation.
\begin{notation}
We use the notation in Section \ref{sec:H} and Section \ref{sec:mirror-curve}. For each flag $(\tau,\si)$, we define the following objects.
\begin{eqnarray}
&& x_{(\tau,\si)} := -\log (X_{(\tau,\si)}) = a(\tau,\si)x + b(\tau,\si)y + \delta_{(\tau,\si)}(q), \label{eqn:x-flag} \\
&& y_{(\tau,\si)} := -\log (Y_{(\tau,\si)}) = c(\tau,\si)x + d(\tau,\si)y + \ep_{(\tau,\si)}(q), \label{eqn:y-flag}
\end{eqnarray}
where $\delta_{(\tau,\si)}(q)$ and $\ep_{(\tau,\si)}(q)$ are linear functions in $\log q_i$ with rational coefficients.
\begin{equation}\label{eqn:su-flag}
\su_1(\tau,\si) = d(\tau,\si) \su_1 - c(\tau,\si) \su_2,\quad \su_2(\tau,\si) = -b(\tau,\si)\su_1 +a(\tau,\si) \su_2.
\end{equation}
Define
\begin{equation}
\hx_{(\tau,\si)}:= \su_1(\tau,\si) x_{(\tau,\si)} + \su_2(\tau,\si) y_{(\tau,\si)}.
\end{equation}
It follows from \eqref{eqn:x-flag}, \eqref{eqn:y-flag} and \eqref{eqn:su-flag} that
\begin{equation}\label{eqn:hx-flag}
\hx_{(\tau,\si)} =  \hx + \tc(\tau,\si)
\end{equation}
where $\hx = \su_1 x + \su_2 y$ and $\tc(\tau,\si)= \su_1(\tau,\si)\delta_{(\tau,\si)}(q) + \su_2(\tau,\si) \ep_{(\tau,\si)}(q)$.  Finally, we define
\begin{equation}\label{eqn:sw-flag}
\sw_1'(\tau,\si)= \frac{\su_1(\tau,\si)}{\fr(\tau,\si)}, \quad
\sw_2'(\tau,\si)= \frac{\fs(\tau,\si)}{\fr(\tau,\si)\fm_\tau}\su_1(\tau,\si) +\frac{\su_2(\tau,\si)}{\fm_\tau},
\quad\quad \sw_3'(\tau,\si)=-\sw_1'(\tau,\si)-\sw'_2(\tau,\si).
\end{equation}
\end{notation}

Denote $\hat X=e^{-\hat x}$ and $\hat Y=e^{-\hat y}$. Recall that there is a bijection between the zeros of $dW^{\bT'}_q$ (and also
$d \hat x$) and the set
$I_\Sigma$, and we denote the corresponding critical point by $\bp_\bsi(q)=(X_\bsi(q), Y_\bsi(q), Z_\bsi(q))$ and $p_\bsi(q)=(\check{u}^{\bsi}(q), \check{v}^\bsi(q))$.
Around  $p_{\bsi}$ we have
\begin{eqnarray*}
\hx &=& \check{u}^{\bsi} +\zeta_\bsi^2 \\
\hy &=& \check{v}^{\bsi} +\sum_{d=1}^\infty h^{\bsi}_d \zeta_\bsi^d
\end{eqnarray*}
where
\begin{equation}
  \label{eqn:h-Hess}
h_1^{\bsi}=\sqrt{ \frac{2}{\frac{d^2\hat{x}}{d\hat{y}^2}(p_{\bsi}) } }=\sqrt{\frac{2}{-\det\Hess(W_q^{\bT'})(\bp_\bsi)}}
\end{equation}
Let $\Gamma_{\bsi}$ be the Lefschetz thimble of the superpotential $\hat{x}:C_q\to \bC$
such that $\hat{x}(\Gamma_{\bsi})= \check{u}^{\bsi} +\bR_{\geq 0}$.
Then $\{\Gamma_{\bsi}:\bsi \in I_\Si\}$ is a basis of the relative
homology group $H_1(C_q, \{ \hat{x}\gg 0\})$.

\begin{lemma}\label{lm:oscillatory}
Suppose that $\su_1,\su_2$ are real numbers such that such that $\sw_i'(\tau,\si)$ is a nonzero real number for any flag $(\tau,\si)$ and for any $i\in \{1,2,3\}$, so
that $f=\su_2/\su_1$ is generic and $W_q^{\bT'}$ is holomorphic Morse with $2\fg-2+\fn$ critical points for $q\in U_{\ep(f)}$. There exists $\delta\in (0,\ep(f)]$ such that
if $q\in U_\delta$ then for each  $\bsi=(\si,\gamma)\in I_\Si$, there exists $\tGa_{\bsi}
\in H_3((\bC^*)^3,\{\Re(\frac{W_q^{\bT'}}{z} \gg 0\};\bZ)$ such that
$$
I_{\bsi}:=
\int_{\tGa_{\bsi}} e^{-\frac{W_q^{\bT'}}{z} } \Omega
=2\pi\sqrt{-1} \int_{\Gamma_{\bsi}} e^{-\hat{x}/z}\Phi
$$
where $\Phi= \hat{y}d\hat{x}$ and $\Omega=\frac{dX}{X}\wedge \frac{dY}{Y}\wedge \frac{dZ}{Z}$.
\end{lemma}

\begin{proof}
We have
$$
W^{\bT'}_q = H(X,Y,q)Z + \hx
$$
where $\hx = \su_1 x +\su_2 y$. By  \eqref{eqn:HH} and \eqref{eqn:hx-flag}, for any flag $(\tau,\si)$, we may write
$$
W^{\bT'}_q = H_{(\tau,\si)}(X_{(\tau,\si)}, Y_{(\tau,\si)},q) Z_{(\tau,\si)} + \hx_{(\tau,\si)} - \tc(\tau,\si),
$$
where $Z_{(\tau,\si)}=a_{i_3}(q) X^{m_{i_3}} Y^{n_{i_3}}$.

In the remainder of this proof, we fix a flag $(\tau,\si)$, and use the following notation:
$$
\tX = X_{(\tau,\si)},\quad \tY = Y_{(\tau,\si)},\quad \tZ= Z_{(\tau,\si)},\quad \tx = x_{(\tau,\si)},\quad \widetilde{\hx} = \hx_{(\tau,\si)},\quad \tc = \tc_{(\tau,\si)},
$$
$$
\tfr=\fr(\tau,\si),\quad  \tfs=\fs(\tau,\si), \quad \tsu_i =\su_i(\tau,\si),\quad  \tsw_j=\sw'_j(\tau,\si),
$$
where $i\in \{1,2\}$ and $j\in \{1,2,3\}$.
We further assume $\tsw_1 > \tsw_2 > 0$, so that $\tsw_3=-\tsw_1-\tsw_2 <0$; the other cases are similar.

The relative cycle $\Gamma_\bsi\subset C_q$ is characterized by $\tx(\Gamma_\bsi)=[\tx_{\si,\gamma},+\infty)$, where  $\tx_{\si,\gamma}= -\log \tX_{\si,\gamma}$. Thus $\Gamma_\bsi$ are actually defined for all $|q|<\ep$, even when $C_q$ is not smooth.
At $q=0$, the mirror curve equation $H_{(\tau,\si)}(\tX,\tY,q)=0$ becomes
$$
\tX^{\tfr} \tY^{-\tfs}+\tY^{\fm_\tau}+1=0,
$$
which is the equation of the mirror curve $C_{\si}$ of the affine toric Calabi-Yau 3-orbifold $\cX_\si$ defined by
the 3-cone $\si$. There are $\tilde{\fr}\fm_\tau$ critical points of the function
$\widetilde{\su}_1\tx + \widetilde{\su}_2 \ty =\hat{\tx}- \tc$ on
$C_q$ which can be holomorphically extended to
$C_{\si}$ when $q=0$ -- they are all critical points on
$C_{\si}$ (see Section \ref{sec:toric-degeneration}).  By direct computation (see e.g. \cite[Section 6.5]{FLZ}), for $\gamma\in G_\si^*$,
each critical point $(\tX_{\si,\gamma}(q),\tY_{\si,\gamma}(q))$ at $q=0$ satisfies
\begin{align*}
&-1<\tX_{\si,\gamma}(0)^{\tfr}\tY_{\si,\gamma}(0)^{-\tfs}=\frac{-\tsw_1}{\tsw_1+\tsw_2}<0,\quad
-1<\tY_{\si,\gamma}(0)^{\fm_\tau}=\frac{-\tsw_2}{\tsw_1+\tsw_2}<0
\end{align*}
For each $\bsi=(\si,\gamma)$, we define a relative cycle
$$
\tGamma^{\mathrm{red}}_{\bsi,q}=\{\tX\in \bR^+ \tX_\bsi(q), \tY_\si \in \bR^+ \tY_\bsi(q)\}.
$$
When $q=0$, we have a disjoint union of connected components
$$
\{ \tX^{\tfr}\tY^{-\tfs}\in \bR^-,\ \tY^{\fm_\tau} \in \bR^- \} =\bigsqcup_{\gamma\in G_\sigma^*} \tGamma^{\mathrm{red}}_{(\si,\gamma),0}
$$
where $\tGamma^{\mathrm{red}}_{(\si,\gamma),0}$ is the connected component
passing through $(\tX_{\si,\gamma}(0),\tY_{\si,\gamma}(0))$.
We define $\tGamma_\bsi=\tGamma_{\bsi,q}=\tGamma^{\mathrm{red}}_{\bsi,q} \times C$, where
$C=\{\tZ\in -1+\sqrt{-1}\bR\}$.
The cone $\si$ in the flag $(\tau,\si)$ is a $3$-cone. Then for any $\vec h= (r_i)_{i\in I_\si}\in \bZ^\fp$, define $c^\si_j(\vec h)\in \bQ$ for $j=1,2,3$ by
\[
\sum_{i\in I_\si}r_i b_i=\sum_{j=1}^3 c^\si_{j}(\vec h) b_{i_j}.
\]
We further define
\[
\chi_\alpha(\vec h)=\chi_\alpha(\sum_{j=1}^3 \{c^\si_{j}(\vec h)\} b_{i_j}),
\]
where $\alpha \in G_\si^*$. We consider $\sum_{j=1}^3 \{c^\si_{j}(\vec h)\} b_{i_j}$ as a box element in $\mathrm{Box}(\si)\cong G_\si$ (see Section \ref{sec:CR}).

As in \cite[Section 7.2]{FLZ}, we can compute
\begin{align*}
e^{-\tc/z}I_{\bsi}
&=\int_{\tGamma_{\bsi,q}} e^{-\tc/z} \cdot e^{-\frac{W_q^{\bT'}}{z}} \Omega\\
&=\frac{2\pi\sqrt{-1}}{|G_\si|}e^{\sqrt{-1} \mathrm{Im}\left(\widetilde{\hx}\right)/z}\sum_{\vec{h}=(r_i)_{i\in I_\si}, r_i\in \bZ_{\geq 0}}e^{\sqrt{-1}c_3(\vec h)}\chi_\gamma(\vec h)\prod_{i\in I_\si} \frac{(-a_i^\si(q))^{r_i}}{r_i!}
\cdot \frac{\Gamma(\frac{\tsw_1}{z}+c^{\si}_1(\vec h))\Gamma(\frac{\tsw_2}{z}+c^{\si}_2(\vec h))}{\Gamma(-\frac{\tsw_3}{z}-c^{\si}_3(\vec h)+1)}.
\end{align*}
We define
$$
L_{\bsi}:= \{q \in \bC^\fp: a_i^\si(q)\tX_{\bsi}(0)^{ m_i^{(\tau,\sigma)}   }\tY_{\bsi}(0)^{ n_i^{(\tau,\sigma)} }\in \bR \textup{ for }i\in I_\si \}
$$
which is a union of totally real linear subspace of $\bC^\fp$.
When $|q|$ is small (then $|a_i^\si(q)|$ are small), $\tGamma_{\bsi,q}=\tGamma_{\bsi,0}$ for $q$ in a sufficiently small neighborhood of $0\in  L_{\bsi}$. Set
\begin{align*}
&v^+,v^-\in \bC,\quad \overline \Gamma_{\bsi,q}=\tGamma_{\bsi,q}\times \{v^+=\overline{v^-}\}, \quad
  \overline \Gamma^{\mathrm{red}}_{\bsi,q}=\tGamma^{\mathrm{red}}_{\bsi,q}\times
  \{v^+=\overline{v^-}\},\\
&\Omega'=\frac{d\hat X}{\hat X}\wedge \frac{d\hat Y}{\hat Y} \wedge \frac{dv^-}{v^-}=\frac{d\tX}{\tX }\wedge \frac{d\tY}{\tY}\wedge \frac{dv^-}{v^-}.
\end{align*}
Let $q\in L_{\bsi}$. The dimensional reduction is \cite[Section 7.2]{FLZ}
$$
e^{-\tc/z}I_{\bsi}=-\int_{\overline \Gamma_{\bsi,q}^{\mathrm{red}}}e^{-\widetilde{\hx}/z}
  \Omega' =2\pi\sqrt{-1} \int_{\Gamma_\bsi=\overline \Gamma_{\bsi,q}^{\mathrm{red}}\cap \{\tH=v^-=0\}}
  e^{-\widetilde{\hx}/z} \hat yd\hat x.
$$
So when $q\in L_{\bsi}$,
$$
I_\bsi=2\pi\sqrt{-1}\int_{\Gamma_\bsi} e^{-\hx/z} \hy d\hx.
$$
Since both sides of the above equation are analytic in $q$, it holds for all $q$ in a small neighborhood of $0\in \bC^\fp$. We set $\delta\in (0,\ep(f)]$ small enough such that $U_\delta$ is contained in this neighborhood.
\end{proof}

\subsection{Action by the stacky Picard group}\label{sec:stacky-Pic}
In \cite{CCIT-hodge}, Coates-Corti-Iritani-Tseng introduce
the {\em stacky Picard group} $\Picst(\cX):= \Pic(\cX)/\Pic(X)$, where $\cX$ is a semi-projective smooth toric DM stack
and $X$ is its coarse moduli space. Coates-Corti-Iritani-Tseng define a $\Picst(\cX)$-action on the Landau-Ginzburg mirror of $\cX$.  Note that
$\Picst(\cX)$ is a finite abelian group.

In this subsection, we describe a $\Picst(\cX)$-action on the total spaces of the Hori-Vafa mirror family and  the family of mirror curves, based on the Galois action on the Landau-Ginzburg model described in \cite[Section 4.3]{CCIT-hodge}.

Given a line bundle $\cL$ on $\cX$ and an object $(x,k)$ in the inertia stack $\cI\cX$, where $x$ is an object in the groupoid $\cX$ and $k\in \Aut(x)$, $k$ acts on the fiber
$\cL_x$ with eigenvalue $\exp(2\pi\sqrt{-1} \ep(x,k))$ for a unique $\ep(x,k)\in [0,1) \cap \bQ$. The map
$(x,k)\mapsto \ep(x,k)$ defines a map $I\cX \to [0,1)\cap \bQ$ which is constant on each connected component $\cX_v$ of $\cI\cX$. We define the age of
$\cL$ along $\cX_v$ to be $\age_v(\cL)= \ep(x,k)$ for any $(x,k)$ in $\cX_v$. If $\cL= p^*L$, where $p:\cX\to X$ is the projection to the coarse moduli space and
$L\in \Pic(X)$ is a line bundle on $X$, then $\age_v(\cL)=0$ for all $v\in \BoxS$. So $\age_v(\cL)$ depends only on the class $[\cL]$ in the quotient group
$\Picst(\cX)=\Pic(\cX)/\Pic(X)$.

Let  $\Picst(\cX)$ act trivially on the variables $u,v$ in the Hori-Vafa mirror, and act on the variables $X$, $Y$, $q=(q_1,\ldots, q_{\fp})$ as follows:
\begin{equation}
[\cL]\cdot X = \exp\left(-2\pi\sqrt{-1} \age_{(1,0,1)}(\cL)\right) X, \quad [\cL]\cdot Y = \exp\left(-2\pi\sqrt{-1}\age_{(0,1,1)}(\cL)\right) Y
\end{equation}
\begin{equation}\label{eqn:L-q}
[\cL]\cdot q_a=
\begin{cases}
q_a , & 1\leq a\leq \fp';\\
\exp\left(2\pi\sqrt{-1} \age_{b_{i+3}}(\cL)\right) q_a, & \fp'+1\leq a\leq  \fp.
\end{cases}
\end{equation}
It is straightforward to check that $\Picst(\cX)$ acts trivially on $H(X,Y,q)$, so it acts on the total spaces of the Hori-Vafa mirror and the family of mirror curves.

When $\cX=[\bC^3/G]$ is affine, where $G$ is a finite subgroup of the maximal torus of $SL(3,\bC)$, the action of $\Picst(\cX)$ on the total space of the family of mirror curves
specializes to the action defined by Equation (24) of \cite{FLZ}; note that in this case $\Picst(\cX)=\Pic(\cX)= G^*$.

By Proposition \ref{prop:PF},  the solution space of the non-equivariant Picard-Fuchs system is  $(\fp+\fg+1)$-dimensional, spanned
the coefficients $1,\tau_1(q),\ldots \tau_{\fp}(q), W_1(q),\ldots, W_{\fg}(q)$ of the non-equivariant small $I$-function.
We have:
\begin{lemma} \label{lm:monodromy}
\begin{enumerate}
\item For $1\leq a\leq \fp'$, $\tau_a(q)$ is the unique solution to the non-equivariant Picard-Fuchs system such that
$$
[\cL]\cdot \tau_a(q)=\tau_a(q)
$$
for all $[\cL]\in \Picst(\cX)$, and
$$
\tau_a(q)= \log (q_a) + O(|q|).
$$
\item For $\fp'+1\leq a\leq \fp$, $\tau_a$ is the unique solution to the non-equivariant Picard-Fuchs system such that
$$
[\cL]\cdot \tau_a(q)= \exp\left(2\pi\sqrt{-1} \age_{b_{a+3}}(\cL)\right)\tau_a(q)
$$
for all $[\cL]\in \Picst(\cX)$, and
$$
\tau_a(q)= q_a + O(|q_{\mathrm{orb}}|^2) + O(|q_K|).
$$
\end{enumerate}
\end{lemma}
\begin{proof} Equation \eqref{eqn:L-q} implies
$$
[\cL]\cdot q^\beta= \exp\left(2\pi\sqrt{-1}\age_{v(\beta)}(\cL)\right) q^\beta
$$
for any $\beta\in \bK_{\eff}$. The lemma follows from the above equation and the explicit expression of $\tau_a(q)$ and $W_i(q)$.
\end{proof}

\section{Geometry of the Mirror Curve}
\label{sec:mirror-curve-geometry}

\subsection{Riemann surfaces}

In this subsection, we recall some classical results on Riemann surfaces.
The main reference of this subsection is \cite{Fay}.

Let $C$ be a non-singular complex projective curve, which can also be viewed
as a compact Riemann surface. Let $\fg\in \bZ_{\geq 0}$ be the genus of $C$.
Let $\cap$ denote the intersection pairing $H_{1}(C;\bZ)\times H_{1}(C;\bZ)\to \bZ$. We choose
a symplectic basis $\{ A_i, B_i: i=1,\ldots, \fg\}$ of $(H_{1}(C;\bZ), \cap)$:
$$
A_i \cap A_j = B_i\cap B_j =0,\quad A_i\cap B_j = -B_j\cap A_i =\delta_{ij},\quad i,j\in \{1,\ldots,\fg\}.
$$

Recall that on a Riemann surface $C$,
a differential of the first kind on $C$ is a holomorphic 1-form; a differential of the second kind on $C$ is a meromorphic
1-form whose residue at any of its pole is zero; a differential of the third kind on $C$ is a meromorphic 1-form
with only simple poles. If $\omega$ is a differential of the first or second kind then
$\int_A \omega$ is well-defined for $A\in H_1(C;\bZ)$.

The fundamental differential of the second kind on $C$ normalized by $A_1,\ldots, A_{\fg}$ is a bilinear symmetric meromorphic
differential $B(p_1,p_2)$ characterized by
\begin{itemize}
\item $B(p_1,p_2)$ is holomorphic everywhere except for a double pole along the diagonal $p_1=p_2$, where,
if $z_1$, $z_2$ are local coordinates on $C\times C$ near $(p,p)$ then
$$
B(z_1,z_2)= (\frac{1}{(z_1-z_2)^2} + f(z_1,z_2) )dz_1 dz_2.
$$
where $f(z_1, z_2)$ is holomorphic and $f(z_1,z_2)=f(z_2,z_1)$.
\item $\displaystyle{\int_{p_1\in A_i} B(p_1,p_2) =0}$, $i=1,\ldots, \fg$.
\end{itemize}
It is also called the Bergman kernel in \cite{EO07,EO15}.

Let $\omega_i\in H^0(C,\omega_C)$ be the unique holomorphic 1-form on $C$ such that
$$
\frac{1}{2\pi\sqrt{-1}}\int_{A_j}\omega_i =\delta_{ij}.
$$
Then $\{\omega_1,\ldots, \omega_{\fg}\}$ is a basis of $H^0(C,\omega_C)\cong \bC^{\fg}$, the space
of holomorphic 1-form on $C$, and
$$
\int_{p'\in B_i} B(p,p') =\omega_i(p).
$$
More generally, for any $\gamma\in H_1(C;\bZ)$,
$$
\omega_\gamma(p):= \int_{p'\in \gamma} B(p,p')
$$
is a holomorphic 1-form on $C$ and
$$
\frac{1}{2\pi\sqrt{-1}}\int_{A_j} \omega_\gamma = A_j\cap \gamma.
$$

We may extend the intersection pairing to a skew-symmetry $\bC$-bilinear map
$\cap: H_1(C;\bC)\times H_1(C;\bC)\to \bC$.
The above discussion remains valid if we choose a symplectic basis
$\{A_i, B_i:i=1,\ldots,\fg\}$ of $(H_1(C;\bC),\cap)$ instead of $(H_1(C;\bZ),\cap)$.

Let $\gamma$ be a path connecting two distinct points $p_1, p_2\in C$, oriented such that
$\partial \gamma = p_1 -p_2$. Then
$$
\omega_\gamma(p):= \int_{p'\in \gamma} B(p,p')
$$
is a meromorphic 1-form on $C$ which is holomorphic on $C\setminus\{p_1,p_2\}$ and has
simple poles at $p_1,p_2$. The residues of $\omega_\gamma$ at $p_1,p_2$ are
$$
\Res_{p\to p_1} \omega_\gamma(p) = 1,\quad \Res_{p\to p_2}\omega_\gamma(p)=-1.
$$

\subsection{The Liouville form}
Let
$$
\hx= \su_1 x + \su_2 y,\quad \hy=\frac{y}{\su_1}, \quad f=\frac{\su_2}{\su_1}
$$
as before. Define
$$
\lambda := \hy d\hx = y d(x + f  y),\quad \Phi:= \lambda|_{C_q}.
$$
Then $\Phi$ is a multi-valued holomorphic 1-form on $C_q$. Recall that there is
a regular covering map  $p: \tC_q\to C_q$ with fiber $\bZ^2$ which is the restriction of
$\bC^2\to (\bC^*)^2$ given by $(x,y)\mapsto (e^{-x}, e^{-y})$. Then $p^*\Phi$ is a holomorphic 1-form
on $\tC_q$.

\subsection{Differentials of the first kind and the third kind} \label{sec:first-third}
For any integers $m,n$,
$$
\varpi_{m,n}:= \Res_{H(X,Y,q)=0}\Bigl( \frac{X^m Y^n}{H(X,Y,q)}\cdot \frac{dX}{X}\wedge\frac{dY}{Y} \Bigr) =
\frac{-X^m Y^n}{\frac{\partial}{\partial \hat{y}} H(X,Y,q)}d\hat{x}
$$
is a holomorphic 1-form on the mirror curve $C_q$ and a meromorphic 1-form on the
compactified mirror curve $\Cbar_q$.

By results in \cite{BC94}, $\varpi_{m,n}$ is holomorphic on $\Cbar_q$ iff $(m,n)\in \Int(P)\cap N'$.
Recall that
$$
\fp=|P\cap N'|-3,\quad \fg=|\Int(P)\cap N'|,\quad \fn = |\partial P\cap N'|.
$$
For generic $q$, $\Cbar_q$ is a compact Riemann surface of genus $\fg$, intersecting
the anti-canonical divisor $\partial \bS_P:=\bS_P\setminus (\bC^*)^2$ transversally at $\fn$ points
$\bar{p}_1,\ldots, \bar{p}_{\fn}$, so $C_q$
a Riemann surface of genus $\fg$ with $\fn$ punctures.
The space of holomorphic 1-forms on $\Cbar_q$, $H^0(\Cbar_q,\omega_{\Cbar_q})$,
is $\fg$-dimensional, where $\fg$ can be zero. A basis of $H^0(\Cbar_q,\omega_{\Cbar_q})$ is given by
$$
\{ \varpi_{m,n} : (m,n)\in \Int(P)\cap N'\}.
$$
Let $D^\infty_q:= \Cbar_q\cap \partial \bS_P = \bar{p}_1 +\cdots+ \bar{p}_{\fn}$.
The space of meromorphic 1-forms on $\Cbar_q$ with at most simple poles at $\bar{p}_1,\ldots,\bar{p}_{\fn}$,  $H^0(\Cbar_q,\omega_{\Cbar_q}(D^\infty_q))$, is
$(\fg+\fn-1)$-dimensional.
It is spanned by the  $(\fg+\fn)$ 1-forms
$$
\{ \varpi_{m,n}: (m,n)\in P\cap N'\}= \{ \varpi_{m_i,n_i}: i=1,\ldots,\fp+3\}.
$$
with a single relation
$$
\sum_{i=1}^{\fp+3} a_i(q) \varpi_{m_i,n_i} =0.
$$

Let $U_\ep \subset (\bC^*)^{\fp'}\times \bC^{\fp-\fp'}$ be defined as in \eqref{eqn:U-ep}.
Choose $\ep>0$ small enough such that if $q\in U_\ep$ then $\Cbar_q$ is smooth and intersects $\partial\bS_P$ transversally
at $\fn$ points.
There is a holomorphic vector bundle $\bE$ of rank $\fg+\fn-1$ over $U_\ep$ whose fiber over
$q\in U_\ep$ is $H^0(\Cbar_q,\omega_{\Cbar_q}(D_q^\infty))$. The vector bundle $\bE$ has a natural $\Picst(\cX)$-equivariant structure, so $\Picst(\cX)$ acts linearly on the space of sections of $\bE$. For $i=1,\ldots, \fp+3$,
$\varpi_{m_i,n_i}$ defines a holomorphic section of $\bE$, and
$$
[\cL]\cdot \varpi_{m_i,n_i} = \exp\left(-2\pi\sqrt{-1} \age_{b_i}(\cL) \right) \varpi_{m_i,n_i}.
$$

For $a=1,\ldots, \fp'$, we have
\begin{eqnarray*}
q_a\frac{\partial \Phi}{\partial q_a} &= &
\Res_{H(X,Y,q)=0}\Bigl( \frac{q_a\frac{\partial H}{\partial q_a}(X,Y,q)}{H(X,Y,q)}\cdot \frac{dX}{X}\wedge\frac{dY}{Y}\Bigr),\\
{ [\cL]} \cdot q_a\frac{\partial \Phi}{\partial q_a}  &=& q_a\frac{\partial \Phi}{\partial q_a}  \quad \textup{ for all  }[\cL]\in \Picst(\cX).
\end{eqnarray*}

For $a=\fp'+1,\ldots, \fp$, we have
\begin{eqnarray*}
\frac{\partial \Phi}{\partial q_a} &=&
\Res_{H(X,Y,q)=0}\Bigl( \frac{\frac{\partial H}{\partial q_a}(X,Y,q)}{H(X,Y,q)}\cdot \frac{dX}{X}\wedge\frac{dY}{Y} \Bigr)
= \frac{a_{a+3}(q)}{q_a} \cdot\varpi_{m_{3+a},n_{3+a}}\\
{ [\cL]} \cdot \frac{\partial \Phi}{\partial q_a}  & = & \exp\left(-2\pi\sqrt{-1}\age_{b_{3+a}}(\cL)\right)\frac{\partial \Phi}{\partial q_a}
\quad  \textup{ for all }[\cL]\in \Picst(\cX).
\end{eqnarray*}

\subsection{Toric degeneration}
\label{sec:toric-degeneration}
The main reference of this subsection is \cite[Section 3]{NS06}.

For a generic choice of
$\eta \in \bL^\vee_\bQ$,
\[
\Theta_\eta=\bigcap_{I\in \cA_\eta} \sum_{i\in I} \bQ_{\geq 0} D_i\subset \bL^\vee_\bQ
\]
is a top dimensional convex cone in $\bL^\vee_\bQ \cong \bQ^{\fp}$. Given a semi-projective toric Calabi-Yau
$3$-orbifold $\cX$, there is always a choice of such $\eta$ such that
$\Theta_\eta$ is the extended Nef cone $\Nef(\bSi^\ext)$, and the
construction in Section \ref{sec:anticones} also defines $\cX$. Any $\eta$ in the extended K\"ahler cone $C(\bSi^\ext)$ gives rise to the same $\Theta_\eta$. The cone
$\Theta_\eta$ together with its faces is a fan in $\bL^\vee_\bR$
(still denoted by $\Theta_\eta$), and
determines a $\fp$-dimensional affine toric variety $X_{\Theta_\eta}$.

Consider  the exact sequence
\[
0\longrightarrow M'  \stackrel{\phi^\vee}{\longrightarrow} \tM' \stackrel{\psi^\vee}{\longrightarrow} \bL^\vee \longrightarrow 0
\]
where $M'=M/\langle e_3 \rangle$ and $\tM'=\tM/ \langle \phi^\vee(e_3)
\rangle.$ Let $D^{\bT'}_i$ be the image of $D^\bT_i$ when passing to $\tM'$.

For any proper subset $I\subset \{1,\dots,r\}$, define
\[
\tTheta_I=\sum_{i\in I} \bQ_{\geq 0} D^{\bT'}_i, \quad \tTheta_{I,\eta}=(\psi^\vee)^{-1} (\Theta_\eta)\cap \tTheta_I.
\]
Define a fan
\[
\tTheta_\eta=\{ \tTheta_{I,\eta}|I \subsetneq\{1,\dots,3+\fp\}\}.
\]
This fan determines a toric variety $X_{\tTheta_\eta}$. There is a fan
morphism $\rho': \widetilde \Theta_\eta\to \Theta_\eta$, which induces a flat
family of toric surfaces $\rho: X_{\tTheta_\eta}\to X_{\Theta_\eta}$.

Let $\Theta_H\subset \bL^\vee_{\bQ}$ be the cone spanned by $H_1,\ldots, H_{\fp}$. Let $\bL^\vee_H := \bigoplus_{a=1}^{\fp} \bZ H_a$ and
let $\bL_H$ be the dual lattice. Then $\bL^\vee_H$ is a sublattice of
$\bL^\vee$ of finite index, and $\bL$ is a sublattice of $\bL_H$ of finite index.
Let $\Theta_H\subset \bL^\vee_\bQ$ be the top dimensional cone spanned by the vectors $H_1,\ldots, H_{\fp}$ chosen in
Section \ref{sec:I-J}. Let $\Theta_\eta^\vee$ and $\Theta_H^\vee$ be the dual cones of $\Theta_\eta$ and
$\Theta_H$, respectively.  We have inclusions
$$
\Theta_H\subset \Theta_\eta \subset \bL^\vee_\bQ,\quad
\Theta_\eta^\vee \subset \Theta_H^\vee \subset \bL_\bQ.
$$
Note that $\Theta_\eta^\vee\cap \bL$ is a subset of
$\Theta^\vee_H \cap \bL_H$, so we have an injective ring homomorphism
$$
\bC[\Theta_\eta^\vee\cap \bL]\to \bC[\Theta_H^\vee \cap \bL_H]=\bC[q_1,\ldots, q_{\fp}]
$$
where $q_1,\ldots, q_{\fp}$ are the variables in Section \ref{sec:I-J}.  Taking the spectrum, we obtain
a morphism
$$
\bA^{\fp}= \mathrm{Spec}\left(\bC[q_1,\ldots,q_{\fp}]\right) \longrightarrow
X_{\Theta_\eta}= \mathrm{Spec}\left(\bC[\Theta_\eta^\vee\cap \bL]\right).
$$
and a cartesian diagram
\begin{equation} \label{eqn:pullback}
\begin{CD}
\fX @>{\tnu}>> X_{\tTheta_\eta} \\
@V{\trho}VV  @V{\rho}VV\\
\bA^{\fp} @>{\nu}>> X_{\Theta_\eta}
\end{CD}
\end{equation}
where $\trho:\fX\to \bA^{\fp}$ is a flat family of toric surfaces.

Given $\si\in \Si(1)\cup \Si(2)\cup \Si(3)$, let
$P_\si$ be the convex hull of $\{ (m_i,n_i): b_i\in \si\}$.
This gives a triangulation $T$ of $P$ with
vertices $\{P_\si: \si\in \Si(1)\}$, edges
$\{P_\si:\si\in \Si(2)\}$, and faces $\{ P_\si:\si\in \Si(3)\}$.

We choose a K\"{a}hler class $[\omega(\eta)]\in H^2(X_\Si;\bZ)$ associated to a lattice point
$\eta\in \bL^\vee$; $[\omega(\eta)]$  is the first Chern class of some
ample line bundle over $X_\Si$.  Then it determines a toric graph
$\Gamma\subset  M'_\bR\cong \bR^2$ up to translation by an element in $M'\cong\bZ^2$ (see Section \ref{sec:toric-graph}). The toric graph gives
a {\em polyhedral decomposition} of $M_\bQ$ in the sense of \cite[Section 3]{NS06}.
It is a covering $\cP$ of $M_\bQ$ by strongly convex lattice polyhedra. The asymptotic fan of $\cP$
is defined to be
$$
\Si_{\cP}:=\{ \lim_{a\to 0}a\Xi  \subset M'_\bQ:  \Xi\in \cP\}.
$$
The fan $\Si_{\cP}=\widetilde \Theta_\eta\cap \rho'^{-1}(0)$ defines the toric surface $\bS_P$, i.e.,
$X_{\Si_{\cP}}= \bS_P$.
For each $\Xi\in \cP$, let $C(\Xi)\subset M'_\bQ\times\bQ_{\geq 0}$ be the closure of the
cone over $\Xi\times \{1\}$ in $M'_\bQ\times \bQ$. Then
$$
\tSi_\cP:=\{ \sigma\text{ is a face of $C(\Xi)$}: \Xi\in
\cP\}=\widetilde \Theta_\eta \cap \rho'^{-1} (\bQ_{\geq0} \eta)
$$
is a fan in $M'_\bQ\times \bQ$ with support $|\tSi_\cP|= M_\bQ'\times
\bQ_{\geq 0}$. The projection $\pi':M'_\bQ\times \bQ\to \bQ$ to the second factor defines a map
from the fan $\tSi_\cP$ to the fan $\{0, \bQ_{\geq 0} \}$. This map of fans determines
a flat toric morphism $\pi: X_{\tSi_{\cP}}\to \bA^1$, where
$X_{\tSi_{\cP}}$ is the toric 3-fold defined by the fan $\tSi_\cP$, as shown in the following commutative diagram.
\begin{equation*}
\begin{CD}
X_{\tSi_{\cP}} @>{\tnu}>> X_{\tTheta_\eta} \\
@V{\pi}VV  @V{\rho}VV\\
\bA^{1}  @>>> X_{\Theta_\eta}.
\end{CD}
\end{equation*}

Let $t$ be a closed point in $\bA^1$, and let $X_t$ denote the fiber of $\pi$ over $t$. Then $X_t\cong \bS_P$ for
$t\neq 0$. As shown in \cite{NS06},  when $t=0$, we have a union of irreducible components
$$
X_0 =\bigcup_{\si\in \Si(3)}\bS_{P_\si}.
$$
If $(\si,\tau)$ is a flag, then $P_\tau$ is one of the three edges of the triangle $P_\si$ and  corresponds
to a torus invariant divisor $\bD_\tau$ in the toric surface $\bS_{P_\si}$.

The polytope $\mathrm{Hull}(\tb_1,\dots,\tb_{\fp+3})\subset \tN$ lies
on the hyperplane $\langle \phi^\vee(e^*_3),\bullet \rangle =1$. It
determines a polytope on $\tN'=\{\langle \phi^\vee(e_3^*),\bullet \rangle
=0\}$ up to a translation. The associated line bundle $\mathfrak L$ on
$X_{\widetilde \Theta_\eta}$ has sections $s_i,\ i=1,\dots, \fp+3$ associated to each
integer point in this polytope. Define
\[
s=\sum_{i=1}^{\fp+3} s_i,\quad \widetilde \fC=s^{-1}(0).
\]

The  divisor $\widetilde \fC\subset X_{\tTheta_\eta}$ forms a flat family of curves of
arithmetic genus $\fg$ over $X_{\Theta_\eta}$. Let $\fC:=\tnu^{-1}(\widetilde\fC)\subset \fX$
be the pullback divisor under the morphism $\tnu:\fX\to X_{\tTheta_\eta}$.  Then $\fC\to \bA^{\fp}$ is a flat family of curves of arithmetic genus $\fg$
over $\bA^{\fp}$.

For $q \neq 0$, $\fC_q \subset X_q=\rho^{-1}(q)$ can
be identified with the zero locus of
$$
H(X,Y,q)=\sum_{i=1}^{\fp+3} a_i(q) X^{m_i} Y^{n_i}.
$$
When $\fC_q$ is smooth, it is isomorphic to the compactified mirror curve $\Cbar_q$ defined in Section \ref{sec:mirror-curve}. When $q=0$, we have a union of irreducible components
$$
\fC_0 =\bigcup_{\si\in \Si(3)}\Cbar_\si,
$$
where $\Cbar_\si\subset \bS_{P_\si}$ is the zero locus of $H_{(\tau,\si)}(X_{(\tau,\si)},Y_{(\tau,\si)},0)$ in  \eqref{eqn:untwisted-affine}, viewed as a section
of the line bundle on $\bS_{P_\si}$ associated to the polytope $P_\si$;
here $\tau$ is a 2-cone contained in $\si$, so that  $X_{(\tau,\si)}$, $Y_{(\tau,\si)}$ are coordinates on an affine chart in $\bS_{P_\si}$  (see Section \ref{sec:H}) which extend to
rational functions on $\bS_{P_\si}$.

Let $C_\si = \Cbar_\si \cap (\bS_{P_\si}\setminus \partial \bS_{P_\si})$,
where $\partial \bS_{P_\si}=\bigcup_{\tau\subset \si,\tau\in \Si(2)} \bD_\tau$. Given $\tau\in \Si(2)$, let
$\fm_\tau= |G_\tau|= |P_\tau\cap N'|-1$ as before. Then
$\fC_0\cap \bD_\tau$ consists of $\fm_\tau$ points. When $q=0$,
the group of $\fm_\tau$-th roots of unity,
$\bmu_{\fm_\tau}\cong \bZ/\fm_\tau\bZ$,  acts freely and transitively on
the set $\fC_0\cap \bD_\tau$.

We have
$$
\Cbar_\si = C_\si \cup \left (\bigcup_{\tau\in \Si(2), \tau\subset \si} (\fC_0\cap \bD_\tau)\right).
$$
Let
$\si_1,\si_2 \in \Si(3)$ be two distinct 3-cones
in $\Si$. The intersection of $\Cbar_{\si_1}$ and $\Cbar_{\si_2}$ is non-empty if and only if
$\si_1\cap\si_2=\tau$ for some 2-dimensional cone $\tau\in \Si(2)$. In this case,
$\Cbar_{\si_1}$ and $\Cbar_{\si_2}$ intersect at $\fm_\tau$ nodes.

The genus of $\Cbar_\si$ is
$$
\fg_\si=|\mathrm{Int}(P_\si)\cap N'|.
$$
Let
$$
\fn_\sigma = |\partial P_\si \cap N'|=\sum_{\tau\in \Si(2),\tau\subset \si} \fm_\tau.
$$
Then $C_\si$ is a genus $\fg_\si$  Riemann surface with $\fn_\si$ punctures.
Let $\Gamma_{\fC_0}$ be the dual graph of the nodal curve $\fC_0$. Then
$$
\fg =\sum_{\si\in \Si(3)} \fg_\si + b_1(\Gamma_{\fC_0}).
$$

\begin{example}[Example \ref{exp:graph}, continued]
  \label{exp:b-model}

  We choose $H_1=D_4$ and $H_2=D_5$. Then $q_K=q_1$ and $q_\mathrm{orb}=q_2$. The mirror curve equation $H(X,Y,q)=0$ for Example \ref{exp:polytope} in ($\fr=1,\fs=0,\fm=2$) is given by
  \[
  H(X,Y,q)=X+Y^2+1+q_1 X^2 Y^{-1}+q_2 Y,
  \]
  while the $\mathbb T'$-equivariant Landau-Ginzburg  is given by
  \[
  W^{\bT'}=H(X,Y,q)-\su_1 \log X - \su_2 log Y.
  \]
  The mirror curve $C_q=\{(X,Y)\in \mathbb C^*|H(X,Y,q)=0\}$ is illustrated in Figure \ref{fig:mirror-curve}. It is a fattened tube around the toric graph Figure \ref{fig:graph}. Notice that the gerby leg contributes to two punctures $\{\bar p_0,\bar p_1\}\subset \overline C_q\setminus C_q$, on which $X=0$ and $Y^2=-1+O(q)$. The degenerated mirror curve $\mathfrak C_0$ is illustrated in Figure \ref{fig:degen-mirror-curve}. It is a nodal curve with four components.

\end{example}

\begin{figure}[h]
\begin{center}
\psfrag{p1}{\small $\bar p_0$}
\psfrag{p2}{\small $\bar p_1$}
\includegraphics[scale=0.3]{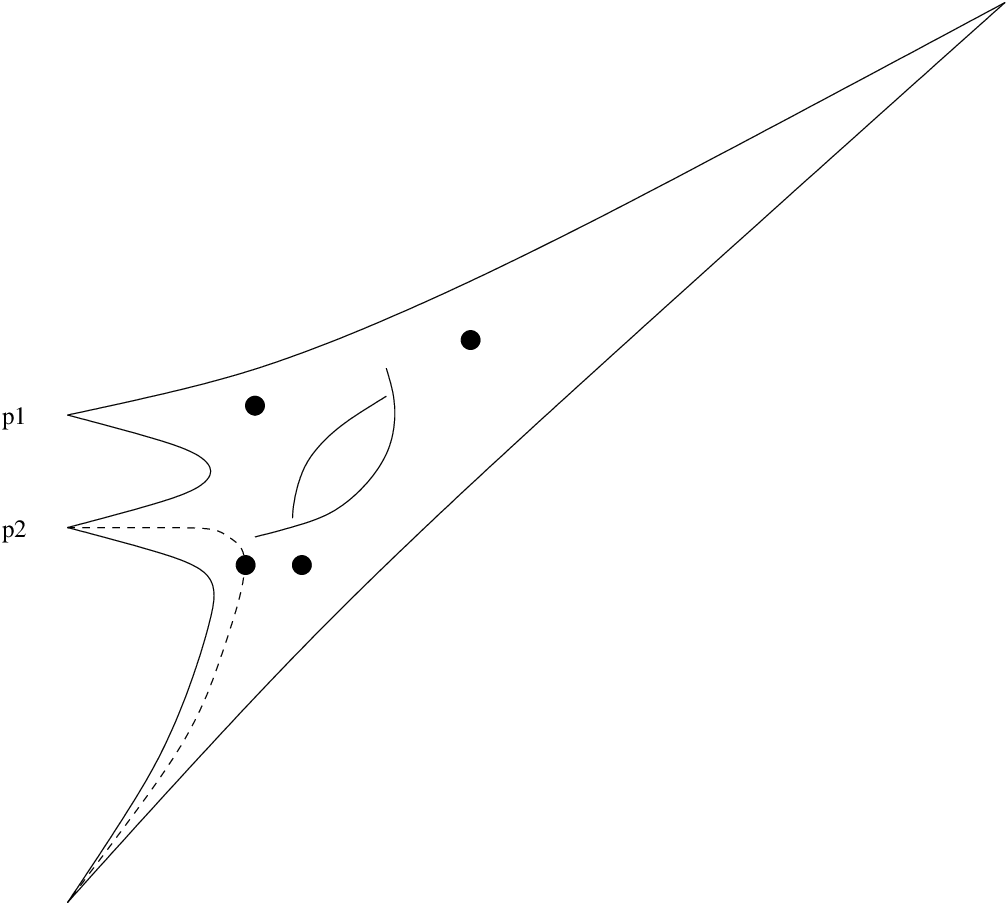}
\caption{The mirror curve of $\cX$ in Example \ref{exp:polytope}. It can be regarded as a curve in the toric surface $\bS_P$. Black dots are the critical points $\{d\hat x=0\}$, while the dashed curve is the Lefschetz thimble passing through one critical point, and at its ends $\hat x \to \infty$.}
\label{fig:mirror-curve}
\end{center}
\end{figure}

\begin{figure}[h]
\begin{center}
\psfrag{p1}{\small $\bar p_0$}
\psfrag{p2}{\small $\bar p_1$}
\includegraphics[scale=0.6]{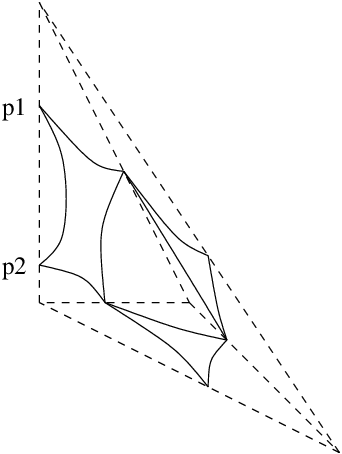}
\caption{The degenerated mirror curve $\fC_0$ of $\cX$ in Example \ref{exp:polytope} at $q=0$. It can be regarded as a curve in the degenerated toric surface - a normal crossing of two $\bP^2$ and a $\bP(1,1,2)$, whose moment polytopes are triangles with dashed lines.}
\label{fig:degen-mirror-curve}
\end{center}
\end{figure}

\subsection{Degeneration of 1-forms}

We first introduce some notation. Let $s_{a i}^\si \in \bZ_{\geq 0}$ be defined as in Section \ref{sec:choice-H}. For $i=1,\ldots,\fp$, let
$$
s_{ai} = \begin{cases}
0, & 1\leq i\leq 3,\\
s_{ai}^{\si_0}, & i\in I_{\si_0}=\{ 4,\ldots, \fp+3\}.
\end{cases}
$$

\begin{proposition}\label{prop:untwisted-limit}
If $1\leq a\leq \fp'$,  the 1-form $q_a\frac{\partial \Phi}{\partial q_a}  \in H^0(\fC_q,\omega_{\fC_q}(D^\infty_q))$ degenerates to
$$
\lim_{q\to 0} q_a\frac{\partial \Phi}{\partial q_a}  \in H^0(\fC_0,\omega_{\fC_0}(D_0^\infty)),
$$
where
$$
\lim_{q\to 0} \left. q_a\frac{\partial \Phi}{\partial q_a}\right|_{C_\si}
=  \frac{ -\sum_{j=1}^3 s_{a i_j} (X_{(\tau,\si)})^{m^{(\tau,\si)}_{i_j}} (Y_{(\tau,\si)} )^{n^{(\tau,\si)}_{i_j}}   }{
Y_{(\tau,\si)}\frac{\partial  H_{(\tau,\si)}}{\partial Y_{(\tau,\si)}}(X_{(\tau,\si)}, Y_{(\tau,\si)},0)} \cdot  \frac{dX_{(\tau,\si)}}{X_{(\tau,\si)}} .
$$
Here $\tau$ is a 2-cone contained in the 3-cone $\si$.
\end{proposition}
\begin{proof} In this proof, $1\leq a\leq \fp'$, and we fix a flag $(\tau,\si)$.
Recall that (see Section \ref{sec:first-third})
$$
q_a \frac{\partial \Phi}{\partial q_a} = \Res_{H(X,Y,q)=0}\left(\frac{q_a\frac{\partial H}{\partial q_a}(X,Y,q)}{H(X,Y,q)}\cdot \frac{dX}{X}\wedge \frac{dY}{Y}\right)
$$
To take the desired limit on $C_\si$, we  rewrite the above expression in coordinates $\chX:=X_{(\tau,\si)}$ and $\chY:=Y_{(\tau,\si)}$.
If $1\leq i\leq \fp'+3$ then  $a_i(q)$ and $a_i^\si(q)$ do not depend on $q_\mathrm{orb}$,  so we may write
$a_i(q)=a_i(q_K)$ and $a_i^\si(q)=a_i^\si(q_K)$. We have
\begin{eqnarray*}
\left.H(X,Y,q)\right|_{q_\mathrm{orb} =0} &=& \sum_{i=1}^{\fp'+3} a_i(q_K) X^{m_i} Y^{n_i} =  a_{i_3}(q_K) X^{m_{i_3}} Y^{m_{i_3}}
\left( \sum_{i=1}^{\fp'+3} a_i^\si(q_K) \chX^{m_i^{(\tau,\si)} }\chY^{m_i^{(\tau,\si)} }\right), \\
\left.q_a\frac{\partial H}{\partial q_a} (X,Y,q)\right|_{q_\mathrm{orb} =0} &=& \sum_{i=1}^{\fp'+3} s_{ai} a_i(q_K) X^{m_i} Y^{n_i} =  a_{i_3}(q_K) X^{m_{i_3}} Y^{m_{i_3}}
\left( \sum_{i=1}^{\fp'+3} s_{ai} a_i^\si(q_K) \chX^{m_i^{(\tau,\si)}} \chY^{m_i^{(\tau,\si)}} \right).
\end{eqnarray*}
So
\begin{eqnarray*}
\left. \frac{ q_a\frac{\partial H}{\partial q_a}(X,Y,q)}{H(X,Y,q)}\right|_{q_\mathrm{orb}=0} \cdot \frac{dX}{X}\wedge \frac{dY}{Y} &=&
\frac{\sum_{i=1}^{\fp'+3} s_{ai} a_i^\si(q_K) \chX^{m_i^{(\tau,\si)}} \chY^{m_i^{(\tau,\si)}}  }{ \sum_{i=1}^{\fp'+3} a_i^\si(q_K) \chX^{m_i^{(\tau,\si)} }\chY^{m_i^{(\tau,\si)} }
}\cdot  \frac{d\chX}{\chX} \wedge \frac{d\chY}{\chY}\\
&\stackrel{q_K\to 0}{\longrightarrow}&  \frac{\sum_{j=1}^3 s_{a i_j} \chX^{m_{i_j}^{(\tau,\si)}} \chY^{n_{i_j}^{(\tau,\si)} } }{ H_{(\tau,\si)}(\chX,\chY, 0)}
\cdot  \frac{d\chX}{\chX}\wedge \frac{d\chY}{\chY}.
\end{eqnarray*}
$$
\lim_{q\to 0} \left. q_a \frac{\partial \Phi}{\partial q_a}\right|_{C_\si}
= \Res_{H_{(\tau,\si)}(\chX,\chY,0)} \left(  \frac{\sum_{j=1}^3 s_{a i_j} \chX^{m_{i_j}^{(\tau,\si)} } \chY^{n_{i_j}^{(\tau,\si)} } }{ H_{(\tau,\si)}(\chX,\chY, 0)}
\cdot  \frac{d\chX}{\chX}\wedge \frac{d\chY}{\chY}.\right)
= \frac{- \sum_{j=1}^3 s_{a i_j} \chX^{m_{i_j}^{(\tau,\si)} } \chY^{n_{i_j}^{(\tau,\si)} } }{ \chY\frac{\partial H_{(\tau,\si)} }{\partial \chY} (\chX,\chY, 0)}
\cdot  \frac{d\chX}{\chX}
$$
\end{proof}

\begin{proposition} \label{prop:twisted-limit}
If  $\fp'+1 \leq a \leq \fp$, the 1-form $\frac{\partial \Phi}{\partial q_a}  \in H^0(\fC_q,\omega_{\fC_q}(D^\infty_q))$ degenerates to
$$
\lim_{q\to 0} \frac{\partial \Phi}{\partial q_a}  \in H^0(\fC_0,\omega_{\fC_0}(D_0^\infty))
$$
such that
$$
\lim_{q\to 0} \left.\frac{\partial \Phi}{\partial q_a}\right|_{C_\si}
=\begin{cases}
 \displaystyle{
\frac{-(X_{(\tau,\si)})^{m^{(\tau,\si)}_{3+a}} (Y_{(\tau,\si)} )^{n^{(\tau,\si)}_{3+a}}}{
Y_{(\tau,\si)}\frac{\partial  H_{(\tau,\si)}}{\partial Y_{(\tau,\si)}}(X_{(\tau,\si)}, Y_{(\tau,\si)},0)} \cdot  \frac{dX_{(\tau,\si)}}{X_{(\tau,\si)}}  },
& b_{3+a}\in \mathrm{Box}(\si),\\
\quad 0, & b_{3+a} \notin \mathrm{Box}(\si),
\end{cases}
$$
where $\tau$ is a 2-cone contained in $\si$.
\end{proposition}
\begin{proof}
In this proof, $\fp'+1\leq a\leq \fp$, and we fix a flag $(\tau,\si)$.
Recall that (see Section \ref{sec:first-third})
$$
\frac{\partial \Phi}{\partial q_a} = \Res_{H(X,Y,q)=0}\left(\frac{a_{a+3}(q)}{q_a} \cdot \frac{X^{m_{a+3}} Y^{m_{a+3}}}{H(X,Y,q)}\cdot \frac{dX}{X}\wedge \frac{dY}{Y}\right)
$$
Again, in order to take the desired limit on $C_\si$, we rewrite the above expression in coordinates $\chX:=X_{(\tau,\si)}$ and $\chY:=Y_{(\tau,\si)}$:
$$
\frac{a_{a+3}(q)}{q_a} \cdot \frac{X^{m_{a+3}} Y^{n_{a+3}}}{H(X,Y,q)}\cdot \frac{dX}{X}\wedge \frac{dY}{Y}
=\frac{a_{a+3}^\tau(q)}{q_a} \cdot \frac{\chX^{m^{(\tau,\si)}_{a+3}} \chY^{n^{(\tau,\si)}_{a+3}}}{H_{(\tau,\si)}(\chX,\chY,q)}\cdot \frac{d\chX}{\chX}\wedge \frac{d\chY}{\chY}.
$$
So
\begin{eqnarray*}
\lim_{q\to 0} \left.\frac{\partial \Phi}{\partial q_a}\right|_{C_\si} &=& \lim_{q\to 0} \frac{a_{a+3}^\si(q)}{q_a}\cdot
\Res_{H_{(\tau,\si)}(\chX,\chY,0)=0} \left( \frac{\chX^{m^{(\tau,\si)}_{a+3}} \chY^{n^{(\tau,\si)}_{a+3}}}{H_{(\tau,\si)}(\chX,\chY,0)}\cdot \frac{d\chX}{\chX}\wedge \frac{d\chY}{\chY} \right)\\
&=& \Bigl( \lim_{q\to 0} \frac{a_{a+3}^\si(q)}{q_a} \Bigr)\cdot \frac{-\chX^{m^{(\tau,\si)}_{a+3}} \chY^{n^{(\tau,\si)}_{a+3}}}{\chY \frac{\partial H_{(\tau,\si)}}{\partial \chY}(\chX,\chY,0)} \cdot \frac{d\chX}{\chX},
\end{eqnarray*}
where
$$
\lim_{q\to 0} \frac{a_{a+3}^\si(q)}{q_a} =\begin{cases}
1, & b_{3+a}\in \mathrm{Box}(\si)\\
0, & b_{3+a}\notin \mathrm{Box}(\si).
\end{cases}
$$
\end{proof}

\subsection{The action of the stacky Picard group on on the central fiber}
The action of the stacky Picard group $\Picst(\cX)$ described in Section \ref{sec:stacky-Pic} extends to an action on $\fC$ which preserves the
central fiber $\fC_0$, so we have a group homomorphism
$\Picst(\cX) \to \Aut'(\fC_0)$,
where $\Aut'(\fC_0)$ is the subgroup of $\Aut(\fC_0)$ given by
$$
\Aut'(\fC_0) =\{ \phi\in \Aut(\fC_0): \phi(C_\si)=C_\si \textup{ for all } \si\in \Si(3)\}.
$$
Each element of $\Aut'(\fC_0)$ restricts to an automorphism of $C_\si$, which gives rise to a group homomorphism
$j_\si^*: \Aut'(\fC_0)\to \Aut(C_\si)$ for each $\si\in \Si(3)$.

For each $\si\in \Si(3)$, the inclusion $\cX_\si\hookrightarrow \cX$ induces a surjective group homomorphism
$\Pic(\cX) \to \Pic(\cX_\si)$  given by $\cL\mapsto \cL|_{\cX_\si}$, which descends to a surjective group homomorphism
$$
i_\si^* :\Picst(\cX)=\Pic(\cX)/\Pic(X)\longrightarrow \Pic(\cX_\si) =\Picst(\cX_\si).
$$
We have a commutative diagram
$$
\begin{CD}
\Picst(\cX) @>>> \Aut'(\fC_0)\\
@VV{i_\si^*}V  @VV{j_\si^*}V\\
\Picst(\cX_\si) @>>> \Aut(C_\si)
\end{CD}
$$

\subsection{The Gauss-Manin connection and flat sections}  \label{sec:GM}

Let $U_\ep \subset (\bC^*)^{\fp'}\times \bC^{\fp-\fp'}$ be defined as in \eqref{eqn:U-ep} as before. We assume that $\ep>0$ small enough such that
$\fC_q \cong \Cbar_q$ is smooth and intersects  $\partial \bS_P$ transversally at $\fn$ points.  We introduce some notation.
\begin{itemize}
\item Let $\bU$ (resp. $\bU_\bZ$) be the flat complex vector bundle (resp. local system of lattices) over $U_\ep$ whose fiber
over $q\in U_\ep$ is $H_1(C_q;\bC)$ (resp. $H_1(C_q;\bZ)$).
\item Let $\overline{\bU}$ (resp. $\overline{\bU}_\bZ$) be the flat complex vector bundle (resp. local system of lattices) over $U_\ep$ whose fiber
over $q\in U_\ep$ is $H_1(\Cbar_q;\bC)$ (resp. $H_1(\Cbar_q;\bZ)$).
\end{itemize}
The flat vector bundle $\bK$ defined in Section \ref{sec:HV}, where each fiber is $K_1(C_q;\bC)$, is a flat subbundle of $\bU$. Continuous sections of $\bU_\bZ$ (resp. $\overline{\bU}_\bZ$) are flat sections of $\bU$ (resp. $\overline{\bU}$) w.r.t. the Gauss-Manin connection.

The $\Picst(\cX)$-action on the total space of the family of mirror curves over $U_\ep$ induces an action on $\bU_\bZ$ and $\bU$. In particular,
$\bU$ is a $\Picst(\cX)$-equivariant flat complex vector bundle over $U_\ep$, so $\Picst(\cX)$ acts on the space of sections
of $\bU$.

\begin{lemma} \label{lm:flat}
For each $\si\in \Si(3)$ there exists $\ep(\si) \in (0,\ep]$ and a $\Picst(\cX)$-equivariant linear map
$H_1(C_\si;\bC) \to \Gamma(U_{\ep(\si)}, \bU)$, $\gamma\mapsto \gamma(q)$, such that
if $\gamma\in H_1(C_\si;\bZ)$ then $\gamma(q)\in H_1(C_q;\bZ)$ for every $q\in U_{\ep(\si)}$.
In particular, $\gamma(q)$ is flat w.r.t. the Gauss-Manin connection for all $\gamma\in H_1(C_\si;\bC)$. Moreover, this linear map restricts to $K_1(C_\si;\bC)\to \Gamma(U_{\epsilon(\si)},\bK)$.
\end{lemma}
\begin{proof}  Choose a 2-cone $\tau\subset \si$, such that $(\tau,\si)$ is a flag. From Section \ref{sec:H}, the equation of $C_q$ can be written
as
$$
0 = H_{(\tau,\si)} (X_{(\tau,\si)}, Y_{(\tau,\si)}, q)
=\sum_{i=1}^{\fp+3} a_i^\si(q) (X_{(\tau,\si)})^{ \fm_i^{(\tau,\si)}  } (Y_{(\tau,\si)})^{\fn_i^{(\tau,\si)} }.
$$
and the equation of $C_\si$ is given by
$$
0 = H_{(\tau,\si)}(X_{(\tau,\si)}, Y_{(\tau,\si)}, 0) = (X_{(\tau,\si)})^{\fr(\tau,\si)} (Y_{(\tau,\si)})^{-\fs(\tau,\si)} + (Y_{(\tau,\si)})^{\fm(\tau,\si)} +1.
$$

We have $H_1(C_\si;\bC)\cong \bZ^{2\fg_\si+\fn_\si-1}$. We choose smooth loops $\tgamma_i: S^1 \to C_\si$,
where $1\leq i\leq 2\fg_\si+\fn_\si -1$, such that the image of $\gamma_i$ does not meet the set or critical points
of $X_{(\tau,\si)}:C_\si \to \bC^*$, and $\{ \gamma=[\tgamma_i]\in H_1(C_q;\bZ): 1\leq i\leq 2\fg_\si+\fn_\si-1 \}$ is a $\bZ$-basis
of $H_1(C_\si;\bZ)$. Then $\tgamma_i(t) = (a_i(t), b_i(t))$, where $a_i, b_i: S^1 \to \bC^*$ are smooth maps such that
$$
H_{(\tau,\si)}(a_i(t), b_i(t), 0) =0,\quad \frac{\partial H_{(\tau,\si)} }{\partial Y_{(\tau,\si)} }(a_i(t), b_i(t)) \neq 0
$$
for all $t\in S^1$. By the implicit function theorem and compactness of $S^1$,  there exists
$\ep(\si)\in (0,\ep)$ and $b_i(t, q)$ which is a smooth function in $t, a_i^\si(q)$ such that
$b_i(t,0)= b_i(t)$ and
$$
H_{(\tau,\si)}(a_i(t), b_i(t,q),q)=0,\quad 1\leq i\leq 2\fg_\si+\fn_\si-1
$$
for small enough $q$. Then $\tgamma_i(t,q)= (a_i(t), b_i(t,q))$, where $t\in S^1$, is a loop in $C_q$ and defines $\gamma_i(q) \in H_1(C_q;\bZ)$. We may view
$\gamma_i(q)$ as a section in  $\Gamma(U_{\ep(\si)},\bU_{\bZ})$, so it is a flat section of $\bU$ on $U_{\ep(\si)}$ w.r.t. the Gauss-Manin connection.
Note that $\gamma_i(q) \in H_1(C_q;\bZ)$ depends only on the class $\gamma_i \in H_1(C_\si;\bZ)$, not the choice of loop $\tgamma_i$ such that
$[\tgamma_i]=\gamma_i$.

For any $\gamma\in H_1(C_\si;\bC)$ there exist unique constants $\{ c_i\in \bC: 1\leq i\leq 2\fg_\si+\fn_\si-1\}$ such that
$$
\gamma = \sum_{i=1}^{2\fg_\si+n_\si -1} c_i[\gamma_i].
$$
Define
$$
\gamma(q) = \sum_{i=1}^{2\fg_\si+ n_\si -1} c_i \gamma_i(q) \in H^1(C_q;\bC).
$$
It follows from the construction that  the map $H_1(C_q;\bC) \to \Gamma(\bU_{\ep(\si)}, \bU)$ is $\bC$-linear, and that $\gamma(q) \in K_1(C_q;\bC)$ if
$\gamma\in K_1(C_\si;\bC)$; $\gamma(q)$ is flat since each $\gamma_i(q)$ is flat.

An element $[\cL]\in \Picst(\cX)$ defines a diffeomorphism $\phi_{[\cL]}: C_q \longrightarrow C_{[\cL]\cdot q}$ which induces an isomorphism
$\phi_{[\cL]*}: H^1(C_q;\bC)\to H^1(C_{[\cL]\cdot q};\bC)$. The section $\gamma(q)$ is $\Picst(\cX)$-equivariant in the following sense:
$$
\gamma([\cL]\cdot q) = \phi_{[\cL]*} (\gamma(q)) \quad \textup{for all } [\cL]\in \Picst(\cX).
$$
\end{proof}

\subsection{Vanishing cycles and loops around punctures}  \label{sec:vanishing}
Let $(\tau,\si)$ be a flag.
For $0\leq \ell \leq \fm_\tau-1$, let $\bar{p}^{(\tau,\si)}_\ell  \in \Cbar_\si\cap \bD_\tau$ be given by
\begin{equation}\label{eqn:pjXY}
X_{(\tau,\si)}=0,  \quad Y_{(\tau,\si)} = \exp(\frac{\pi\sqrt{-1}}{\fm_\tau}(1+2\ell))
\end{equation}
And let $\delta^{(\tau,\si)}_\ell$ be a small loop in $C_\si$ around the puncture $\bar{p}^{(\tau,\si)}_\ell$ such that
$\delta^{(\tau,\si)}_\ell$ is contractible in $\Cbar_\si$.  Then we may construct
$\delta^{(\tau,\si)}_{\ell}(q) \in H_1(C_q;\bZ)$ as in the proof of Lemma \ref{lm:flat}.  We will often
view $\delta^{(\tau,\si)}_\ell$ as a section of $\bU_\bZ$ and write $\delta^{(\tau,\si)}_\ell$ instead of $\delta_\ell^{(\tau,\si)}(q)$.
\begin{itemize}
\item  If $\tau\in \Si(2)_c := \{ \tau\in \Si(2): \ell_\tau=\bP^1\}$, then $\tau$ is the intersection of two 3-cones $\si, \si'$. From the construction,  it is straightforward
to check that
$$
\bar{p}_\ell^{(\tau,\si')} = \bar{p}_{\fm_\tau-1-\ell}^{(\tau,\si)},\quad
\delta_\ell^{(\tau,\si')} = - \delta^{(\tau,\si)}_{\fm_\tau-1-\ell}  \in H_1(C_q;\bZ)
$$
for $0 \leq \ell \leq \fm_\tau-1$. The class $\delta_\ell^{(\tau,\si')}(q) \in H_1(C_q;\bZ)$ is the vanishing cycle associated to
the node $\bar{p}_\ell^{(\tau,\si)} \in \Cbar_\si\cap \Cbar_{\si'}$.
\item If $\tau \in \Si(2)\setminus \Si(2)_c$ then $\sigma$ is the unique 3-cone containing $\tau$, and $\delta_\ell^{(\tau,\si)}$ is the class of a small loop around a puncture
$\bar{p}^{(\tau,\si)}_\ell(q)$ in $\Cbar_q \setminus C_q$.
\end{itemize}

\begin{lemma} \label{lm:delta}
\begin{enumerate}
\item[(a)] If $1\leq  a \leq \fp'$ then
$$
\lim_{q\to 0}\frac{1}{2\pi\sqrt{-1}} \int_{\delta^{(\tau,\si)}_\ell}  q_a \frac{\partial \Phi}{\partial q_a} =
\frac{s_{a i_3} - s_{a i_2} }{\fm_\tau}.
$$
\item[(b)] If $\fp'+1\leq a \leq \fp$ then
$$
\lim_{q\to 0}\frac{1}{2\pi\sqrt{-1}} \int_{\delta^{(\tau,\si)}_\ell} \frac{\partial \Phi}{\partial q_a} =
\begin{cases}
\displaystyle{ \frac{ e^{\pi\sqrt{-1}(2\ell+1)j/\fm_\tau} }{\fm_\tau} }, & \textup{if }b_{3+a} = v_j^{(\tau,\si)} \textup{ for some } j\in \{1,\ldots,\fm_\tau-1\},\\
0, & \textup{otherwise}.
\end{cases}
$$
\end{enumerate}
\end{lemma}
\begin{proof}
(a) follows from Proposition \ref{prop:untwisted-limit} and \eqref{eqn:pjXY}. (b) follows from Proposition  \ref{prop:twisted-limit} and \eqref{eqn:pjXY}.
\end{proof}

We now describe the action of the stacky Picard group $\Picst(\cX)$ more explicitly. We fix a flag $(\tau,\si)$ and define $i_1, i_2, i_3$ as in Section \ref{sec:H}. Then
$$
\mathrm{Box}(\tau)\setminus \{0\} = \{ v^{(\tau,\si)}_j:= (1-\frac{j}{\fm_\tau}) b_{i_2} + \frac{j}{\fm_\tau} b_{i_3} :  j=1,\ldots, \fm_\tau-1\}.
$$
For any $[\cL]\in \Picst(\cX)$, we have
$$
[\cL]\cdot Y_{(\tau,\si)} =\exp\left(-2\pi\sqrt{-1} \age_{v_1^{(\tau,\si)} }(\cL)\right) Y_{(\tau,\si)},
$$
where $\age_{v_1^{(\tau,\si)} }(\cL) \in  \frac{1}{\fm_\tau} \bZ$. If $\phi\in \Aut'(\fC_0)$ is the image of $[\cL] \in \Picst(\cX)$ under the group homomorphism
$\Picst(\cX) \to \Aut'(\fC_0)$, and $\age_{v_1^{(\tau,\si)}} (\cL)= \frac{j}{\fm_\tau}$, where $j\in \{0,1,\ldots, \fm_\tau-1\}$, then
$$
\phi({ \bar p}_\ell^{(\tau,\si)}) =
\begin{cases}
\bar{p}_{\ell- j}^{(\tau,\si)}, &\ell\geq j,\\
 \bar{p}_{\fm_\tau+ \ell- j }^{(\tau,\si)}, &  \ell<j.
\end{cases}
$$
So
\begin{equation}\label{eqn:L-delta}
[\cL]\cdot \delta_\ell^{(\tau,\si)}=
\begin{cases}
 \delta_{\ell-j}^{(\tau,\si)}, & \ell\geq j,\\
 \delta_{\fm_\tau+ \ell -j}^{(\tau,\si)}, &  \ell<j.
\end{cases}
\end{equation}

Let  $\fM_{X_{(\tau,\si)}}, \fM_{Y_{(\tau,\si)}}, \fM_X, \fM_Y: H_1(C_q;\bZ) \to \bZ$ be defined as in Section \ref{sec:HV}. Then
\begin{equation}\label{eqn:fM}
\fM_{X_{(\tau,\si)}} = a(\tau,\si)\fM_X + b(\tau,\si)\fM_Y,\quad  \fM_{Y_{(\tau,\si)}} = c(\tau,\si)\fM_X + d(\tau,\si)\fM_Y
\end{equation}
where $a(\tau,\si), b(\tau,\si), c(\tau,\si), d(\tau,\si)\in \bZ$ are defined as in Section \ref{sec:H}. Recall that
$a(\tau,\si)d(\tau,\si)-b(\tau,\si)c(\tau,\si)=1$.

We have
\begin{equation}\label{eqn:fM-vanishing}
\fM_{X_{(\tau,\si)}}(\delta_\ell^{(\tau,\si)})=1,\quad \fM_{Y_{(\tau,\si)}} (\delta^{(\tau,\si)}_\ell)=0.
\end{equation}
Equation \eqref{eqn:fM} and Equation \eqref{eqn:fM-vanishing} imply
\begin{equation}\label{eqn:fM-vanishing-cycle}
\fM_X(\delta_\ell^{(\tau,\si)})=d(\tau,\si),\quad \fM_Y (\delta^{(\tau,\si)}_\ell)= -c(\tau,\si).
\end{equation}

For $k=0,1,\ldots, \fm_\tau-1$, we define
\begin{equation} \label{eqn:A-delta}
A_k^{(\tau,\si)} = \sum_{\ell=0}^{\fm_\tau-1} e^{-\pi\sqrt{-1}(2\ell+1)k} \delta^{(\tau,\si)}_\ell.
\end{equation}
Then
\begin{equation}\label{eqn:Ak-monodromy}
\fM_X(A_k^{(\tau,\si)}) = \delta_{0,k} \fm_\tau d(\tau,\si),\quad
\fM_Y(A_k^{(\tau,\si)}) = - \delta_{0,k} \fm_\tau c(\tau,\si).
\end{equation}
In particular, $A_k^{(\tau,\si)} \in K_1(C_q,\bC)$ for $k\neq 0$.

\begin{lemma}
\begin{enumerate}
\item[(a)] For any $[\cL]\in \Picst(\cX)$,
$$
[\cL]\cdot A_k^{(\tau,\si)} = \exp\left(2\pi\sqrt{-1} k \cdot \age_{v_1^{(\tau,\si)}} (\cL)\right) A_k^{(\tau,\si)}.
$$
\item[(b)] If $1\leq a\leq \fp'$ then
$$
\lim_{q\to 0} \frac{1}{2\pi\sqrt{-1}} \int_{A_k^{(\tau,\si)}} q_a \frac{\partial \Phi}{\partial q_a} = \delta_{0,k}(s_{a i_3}- s_{a i_2}).
$$
In particular, when $\si=\si_0$ we have
$\; \displaystyle{
\lim_{q\to 0} \frac{1}{2\pi\sqrt{-1}} \int_{A_k^{(\tau,\si_0)}} q_a \frac{\partial \Phi}{\partial q_a} = 0.
}$
\item[(c)] If $\fp'+1\leq a\leq \fp$ then
$$
\lim_{q\to 0} \frac{1}{2\pi\sqrt{-1}} \int_{A_k^{(\tau,\si)}} \frac{\partial \Phi}{\partial q_a}
=\begin{cases}
1, & \textup{ if $k\in \{1,\ldots, \fm_\tau-1\}$ and $b_{a+3}= v_k^{(\tau,\si)}$},\\
0, & \textup{otherwise}.
\end{cases}
$$
\end{enumerate}
\end{lemma}
\begin{proof} (a) follows from  \eqref{eqn:L-delta}. (b) and (c) follow from Lemma \ref{lm:delta}.
\end{proof}

\subsection{B-model flat coordinates} \label{sec:B-flat}

We first introduce some notation.
\begin{itemize}
\item For any $\si\in \Si(3)$, we define
\begin{equation}\label{eqn:Up-si}
\Up_\si:= \{ a\in \{ \fp'+1,\ldots, \fp\}: (m_{3+a}, n_{3+a})\in \Int(P_\si)\}.
\end{equation}
Then $|\Up_\si|=\fg_\si$.
Define
$$
\Up_3:= \bigcup_{\si\in \Si(3)} \Up_\si.
$$

\item For any $\tau\in \Si(2)$, we define
$$
\Up_\tau=\{ a\in \{\fp'+1,\ldots, \fp\}: b_{3+a}\in \mathrm{Box}(\tau)\}.
$$
Then $|\Up_\tau| =\fm_\tau-1$. Define
$$
\Up_2:= \bigcup_{\tau\in \Si(2)} \Up_\tau.
$$
\end{itemize}
Note that $\Up_2\cap \Up_3=\emptyset$ and $\Up_2\cup \Up_3 = \{\fp'+1,\ldots, \fp\}$.

\subsubsection{Interior of a $3$-cone}
In this subsection, we define B-model flat coordinates $\check{\tau}_a$ for $a\in \Up_3$.

If $\si\in \Si(3)$ then
$$
\{ \varpi^\si_a:= \left(\lim_{q\to 0}\frac{\partial \Phi}{\partial q_a}\right)\Bigr|_{C_\si}:  a\in \Up_\si \}
$$
is a basis of $H^0(\Cbar_\si, \omega_{\Cbar_\si})$, and
$$
[\cL] \cdot \varpi^\si_a =  \exp\left(-2\pi\sqrt{-1} \age_{b_{3+a}}(\cL)\right)\varpi^\si_a
$$
for all $a\in \Up_\si$ and all $[\cL]\in \Picst(\cX)$. For each $a\in \Up_\si$, there is a unique $A_a^\si \in H_1(\Cbar_\si;\bC)$ such that
$$
[\cL]\cdot A^\si_a = \exp\left(2\pi\sqrt{-1} \age_{b_{3+a}}(\cL)\right)A^\si_a.
$$
and
$$
\frac{1}{2\pi\sqrt{-1}}\int_{A_a^\si} \varpi^\si_{a'} = \delta_{aa'}\quad \textup{ for all }a'\in \Up_\si.
$$
There exists a unique $A_a\in K_1(C_\si;\bC)\subset H_1(C_\si; \bC)$ such that $J_*(A_a)=A^\si_a$, and
$$
[\cL]\cdot A_a = \exp\left(2\pi\sqrt{-1}\age_{b_{3+a}}(\cL)\right) A_a
$$
for all $[\cL]\in \Picst(\cX)$. By Lemma \ref{lm:flat} we extend $A_a$ to a flat section of $\bK$ over $U_{\ep(\si)}$. Then
$$
[\cL] \cdot A_a = \exp\left(2\pi\sqrt{-1}\age_{b_{3+a}}(\cL)\right)A_a
$$
for all $a\in \Up_3$ and $[\cL]\in \Picst(\cX)$.

For $a\in \Up_3$, we define B-model flat coordinates
\begin{equation}\label{eqn:flat-three}
\check{\tau}_a(q):= \frac{1}{2\pi\sqrt{-1}} \int_{\tA_a} \Phi.
\end{equation}
where $\tA_a$ is a flat section of $\widetilde{\bK}$ such that $p_*\tA_a =A_a$.
The right hand side of \eqref{eqn:flat-three} is a holomorphic function in $q$, which we fix at $\check{\tau}(0)=0$ by choosing the lift $\tA_a$ (see Lemma \ref{lemm:integral-K} and Remark \ref{rem:integral-K}).
Then $\check{\tau}_a(q)$ is a solution of the non-equivariant Picard-Fuch system, and
\begin{itemize}
\item $\displaystyle{
[\cL]\cdot \check{\tau}_a(q) = \exp\left(2\pi\sqrt{-1} \age_{b_{3+a}}(\cL)\right) \check{\tau}_a(q)
}$ for all $[\cL]\in \Picst(\cX)$,
\item $\check{\tau}_a(q) = q_a + O(|q_\mathrm{orb}|^2) + O(|q_K|)$.
\end{itemize}
By Lemma \ref{lm:monodromy},
$$
\check{\tau}_a(q)= \tau_a(q).
$$

\subsubsection{Interior of a 2-cone}
In this subsection, we define B-model flat coordinates $\check{\tau}_a$ for $a\in \Up_2$.

Given any $a\in \Up_2$, we have
$$
b_{a+3}= v^{(\tau,\si)}_k
$$
for some flag $(\tau,\si)$ and some $k\in \{1,\ldots,\fm_\tau-1\}$.
Define  $A_a= A_k^{(\tau,\si)}$, where $A_k^{(\tau,\si)}$ is defined in Section \ref{sec:vanishing}. Then
$A_a$ is a flat section of $\bK$ over $U_{\ep(\si)}$.
We define
\begin{equation}\label{eqn:flat-two}
\check{\tau}_a(q):= \frac{1}{2\pi\sqrt{-1}}\int_{\tA_a}\Phi
\end{equation}
where $\tA_a$ is a flat section of $\widetilde{\bK}$ such that $p_*\tA_a =A_a$.
The right hand side of \eqref{eqn:flat-two} is a holomorphic function in $q$
defined up to addition of a constant, which we fix by requiring $\check{\tau}(0)=0$.
Then $\check{\tau}_a(q)$ is a solution of the non-equivariant Picard-Fuch system, and
\begin{itemize}
\item $\displaystyle{
[\cL]\cdot \check{\tau}_a(q) = \exp\left(2\pi\sqrt{-1} \age_{b_{3+a}}(\cL)\right) \check{\tau}_a(q)
}$ for all $[\cL]\in \Picst(\cX)$,
\item $\check{\tau}_a(q) = q_a + O(|q_\mathrm{orb}|^2) + O(|q_K|)$. 
\end{itemize}
By Lemma \ref{lm:monodromy},
$$
\check{\tau}_a(q)= \tau_a(q).
$$

\subsubsection{Coordinates mirror to K\"{a}hler parameters}
In this subsection we define B-model flat coordinates $\check{\tau}_a$ for $a\in \{1,\ldots,\fp'\}$.
Recall that $I'_{\si_0} =\{1,2,3\}$ and $I'_{\tau_0}=\{2,3\}$. Let $\tau_1$ be the 2-cone spanned by $b_3$ and $b_1$, so that
$\tau_1\subset \sigma_0$ and $I'_{\tau_1} =\{ 1, 3\}$.
For every flag $(\tau,\si)$, there exist unique  $e_0(\tau,\si),e_1(\tau,\si)\in \bQ$ such that
$$
A_0^{(\tau,\si)} - e_0(\tau,\si) A_0^{(\tau_0,\si_0)} - e_1(\tau,\si) A_0^{(\tau_1,\si_0)} \in K_1(C_q,\bQ).
$$
Denote the above cycle by $\hat{A}^{(\tau,\si)}$. Then  $\hat{A}^{(\tau,\si)}$ is a flat section of $\bK$ over $U_{\ep(\si)}$. Note that  $\hat{A}^{(\tau,\si)}$ is a section of
$\bK_{\bZ}$ when $\cX$ is smooth. We define
\begin{equation}\label{eqn:flat-one}
\check{\tau}^{(\tau,\si)}(q):= \frac{1}{2\pi\sqrt{-1}}\int_{\tA^{(\tau,\si)}} \Phi.
\end{equation}
where $\tA^{(\tau,\si)}$ is a flat section of $\widetilde{\bK}$ such that $p_*\tA^{(\tau,\si)} =\hat{A}^{(\tau,\si)}$.
Then
$$
\check{\tau}^{(\tau,\si)}(q) =  c + \sum_{a=1}^{\fp'} c^{(\tau,\si)}_a \log q_a + O(|q_\mathrm{orb}|^2) + O(|q_K|)
$$
where $c$ is a constant and $c^{(\tau,\si)}_a = s_{a i_3} -s_{a i_2} \in \bZ$. (Recall that the right hand side of \eqref{eqn:flat-one} is defined up to addition of a constant depending on the choice of $\tA^{(\tau,\si)}$, and $i_2, i_3$ are determined by
the flag $(\tau,\si)$.) There exist (non-unique) rational numbers $\hat{c}^{(\tau,\si)}_b \in \bQ$ such that
$$
\sum_{(\tau,\si)} \hat{c}^{(\tau,\si)}_a c^{(\tau,\si)}_b = \delta_{a,b}
$$
where the sum is over all flags $(\tau,\si)$. For $a\in \{1,\ldots \fp'\}$, let
$$
\tA_a = \sum_{(\tau,\si)} \hat{c}^{(\tau,\si)}_a \tA^{(\tau,\si)},\quad  A_a=p_* \tA_a,\quad \check{\tau}_a := \frac{1}{2\pi\sqrt{-1}} \int_{\tA_a} \Phi.
$$
Then  $\check{\tau}_a(q)$ is a solution of the non-equivariant Picard-Fuch system, and
\begin{itemize}
\item $[\cL]\cdot \check{\tau}_a(q) =  \check{\tau}_a(q)$ for all $[\cL]\in \Picst(\cX)$,
\item $\check{\tau}_a(q) = c_a + \log q_a + O(|q_\mathrm{orb}|^2) + O(|q_K|)$, where $c_a$ is a constant which we may choose to be zero.
\end{itemize}
By Lemma \ref{lm:monodromy},
$$
\check{\tau}_a(q)= \tau_a(q).
$$

\subsubsection{$B$-cycles}
By the Lefschetz duality, there is a perfect pairing
$$
\cap: H_1(C_q;\bC)\times H_1(\Cbar_q,D_q^\infty;\bC)\to \bC,
$$
where $\dim_\bC H_1(C_q;\bC)=\dim_\bC H_1(\Cbar_q,D_q^\infty;\bC) = 2\fg+\fn-1$.
This gives an isomorphism $H_1(\Cbar_q,D_q^\infty;\bC)\cong H_1(C_q;\bC)^\vee$.
We also have an intersection pairing
$$
\cap: H_1(\Cbar_q;\bC)\times H_1(\Cbar_q;\bC)\to \bC.
$$
This gives an isomorphism $H_1(\Cbar_q;\bC)\to H_1(\Cbar_q;\bC)^\vee$.
Under the above isomorphisms, the injective map
$H_1(\Cbar_q;\bC)\to H_1(\Cbar_q,D_q^\infty)$ can be identified with the dual
of the surjective linear map $J_*: H_1(C_q;\bC)\to H_1(\Cbar_q;\bC)$.
Let $K_1(C_q;\bC)^\perp\subset H_1(\Cbar_q,D_q^\infty;\bC)$ be the subspace of
annihilators of $K_1(C_q;\bC)$. Then $K_1(C_q;\bC)^\perp$ is 2-dimensional, and
there is a perfect pairing
$$
\cap: K_1(C_q;\bC)\times H_1(\Cbar_q,D_q^\infty;\bC)/K_1(C_q;\bC)^\perp  \longrightarrow \bC.
$$
By permuting $A_1,\ldots, A_{\fp}$ if necessary, we may choose $B_1,\ldots, B_{\fp}\in H_1(\Cbar_q,D_q;\bC)/K_1(C_q;\bC)^\perp$
such that
\begin{enumerate}
\item $\{ A_1,\ldots, A_{\fg}, B_1,\ldots, B_{\fg}\}$ is a symplectic basis of
$H_1(\Cbar_q;\bC)$, and $\{B_1,\ldots, B_{\fg}, -A_1,\ldots, -A_{\fg}\}$ is the dual basis of $H_1(\Cbar_q;\bC)$,
\item $\{ A_1,\ldots, A_{\fp}, B_1,\ldots, B_{\fg}\}$ is a basis of $K_1(C_q;\bC)$, and
$\{B_1,\ldots, B_{\fp}, -A_1,\ldots, -A_{\fg}\}$ is the dual basis of $H_1(\Cbar_q, D_q^\infty;\bC)$.
\end{enumerate}

\subsection{Differentials of the second kind}\label{sec:second-kind}
Let $B(p_1,p_2)$ be the fundamental differential of the second kind normalized by $A_1,\ldots, A_{\fg}$.
Then
$$
\frac{\partial \Phi}{\partial \tau_a}(p)= \int_{p'\in B_a} B(p,p'),\quad a=1,\ldots, \fp.
$$

Following \cite{E11, EO15}, given any $\bsi \in I_\Si$, let
$$
\zeta_{\bsi} =\sqrt{\hx-\check{u}^{\bsi}}
$$
be local holomorphic coordinate near the critical point $p_{\bsi}$. For any non-negative integers $d$,  define
$$
\theta_{\bsi}^d(p):= -(2d-1)!! 2^{-d}\Res_{p'\to p_{\bsi}}
B(p,p')\zeta_\bsi^{-2d-1}.
$$
Then $\theta_{\bsi}^d$ satisfies the following properties.
\begin{enumerate}
\item $\theta_{\bsi}^d$ is a meromorphic 1-form on $\Cbar_q$ with
a single pole of order $2d+2$ at $p_{\bsi}$.

\item In local coordinate $\zeta_{\bsi} =\sqrt{ \hat{x}-\check{u}^{\bsi}}$ near $p_{\bsi}$,
$$
\theta_{\bsi}^d = \Big( \frac{-(2d+1)!!}{2^d \zeta_\bsi^{2d+2}}
+ f(\zeta_\bsi)\Big) d\zeta_\bsi,
$$
where $f(\zeta_\bsi)$ is analytic around $p_{\bsi}$.
The residue of $\theta_{\bsi}$ at $p_{\bsi}$ is zero,
so $\theta_\bsi$ is a differential of the second kind.
\item
\[
\int_{A_i} \theta_\bsi^d=0, \ i=1,\ldots,\fg.
\]
\end{enumerate}
The meromorphic 1-form $\theta_{\bsi}^d$ is uniquely characterized by the above
properties; $\theta_{\bsi}^d$ can be viewed as a section in
$H^0(\Cbar_q, \omega_{\Cbar_q}((2d+2) p_{\bsi}))$.
At $q=0$, if $\bsi=(\si,\alpha)$ then  $\theta^d_{\bsi}(0)|_{\Cbar_{\si'}}=0$ for $\si'\neq \si$, and
\begin{enumerate}
\item $\theta_\bsi^d(0)|_{\Cbar_\si}$ is a meromorphic 1-form on $\Cbar_\si$ with
a single pole of order $2d+2$ at $p_{\si,\alpha}(0)$.

\item In local coordinate $\zeta_{\bsi} =\sqrt{ \hat{x}-\check{u}^{\bsi}}$ near $p_{\bsi}(0)\in \Cbar_\si$,
$$
\theta_{\bsi}^d(0)= \Big( \frac{-(2d+1)!!}{2^d \zeta_\bsi^{2d+2}}
+ f(\zeta_\bsi)\Big) d\zeta_\bsi,
$$
where $f(\zeta_\bsi)$ is analytic around $p_{\bsi}(0)$.
The residue of $\theta_{\bsi}$ at $p_{\bsi}$ is zero,
so $\theta_\bsi$ is a differential of the second kind on $\Cbar_\si$.
\item
$\displaystyle{\int_{A_i} \theta_\bsi^d(0)|_{\Cbar_\si}=0}$ for $i\in \Up_\si$,
where $\Up_\si$ is defined by Equation \eqref{eqn:Up-si}.
\end{enumerate}
Therefore, $\theta_{\bsi=(\sigma,\alpha)}^d|_{\Cbar_\si}$ coincides with the differential of the second kind
$\theta^\alpha_d$ on $\Cbar_\si$ (the mirror curve for the affine toric Calabi-Yau 3-orbifold $\cX_\si =[\bC^3/G_\si]$) in \cite[Section 6.6]{FLZ}.
\

\section{B-model Topological Strings}
\label{sec:b-string}

\subsection{Canonical basis in the B-model: $\theta_{\bsi}^0$ and $[V_{\bsi}(\btau)]$.}

For any Laurent polynomial
$$f\in
\bar{S}_{\bT'}[X,X^{-1},Y,Y^{-1},Z,Z^{-1}],
$$
let
$[f]\in\Jac(W^{\bT'}_q)$ be the image of $f$ in the Jacobian ring.
For $\bsi\in I_\Si$, let $V_{\bsi}(\btau)$ be a
Laurent polynomial in $X,Y,Z$ such that $V_{\bsi}(\btau)(\bp_{\bsi'})=\delta_{\bsi,\bsi'}$. The collection $\{ [V_{\bsi}(\btau)] \}_{\bsi\in I_\Sigma}$ is a canonical basis
of the semisimple Frobenius algebra $\Jac(W^{\bT'}_q)$.

Let $\check H_a=[\frac{\partial W^{\bT'}}{\partial \tau_a}]$, which correspond to $H_a$ under the isomorphism in Equation \eqref{eqn:QH=Jac}. Since $\check H_1,\dots,\check H_{\fp}$ multiplicatively generate $\Jac(W^{\bT'}_q)$,
we may choose a basis of $\Jac(W^{\bT'}_q)$ of the form
$$
1, \check H_1,\ldots, \check H_{\fp}, E_1,\ldots, E_{\fg}
$$
where $E_i =\check{H}_{a_i}\check{H}_{b_i}$ for some $a_i,b_i \in \{1,\ldots,\fp\}$.
We can write $[f]$ in the following decomposition
$$
[f]=\sum_{i=1}^{\fg} A_i(q) \check H_{a_i}\cdot \check H_{b_i}+ \sum_{a=1}^{\fp}  B_a(q) \check H_a + C(q) \one.
$$
Let $D_a f=\frac{\partial W^{\bT'}_q}{\partial \tau_a} f - z \frac{\partial f}{\partial \tau_a}$. Define the \emph{standard form} of $[f]$ to be
$$
\bar f=\sum_{i=1}^{\fg} A_i(q) D_{a_i}D_{b_i}1+\sum_{a=1}^{\fp}  B_a(q) D_a1 +C(q)=\sum_{d=0}^2 z^d f_d,$$
and define the oscillating integral of $[f]$ to be
$$
\int_{\tGa} e^{-\frac{W^{\bT'}_q}{z}}\bar f\Omega.
$$
We see that $[\bar{f}] =[f]$. Direct calculation shows the following.

\begin{lemma}\label{lm:H}
We have the following identities in the Jacobian ring $\Jac(W^\bT_q)$.
\[
\one=\sum_{\bsi\in I_\Sigma} [V_{\bsi}(\btau)],\quad\
\check H_a =-\sum_{\bsi\in I_\Sigma}B'^{\bsi}_a(q) [V_{\bsi}(\btau)],\quad\
\check H_a \cdot \check H_b =\sum_{\bsi\in I_\Sigma} C'^{\bsi}_{ab}(q) [V_{\bsi}(\btau)],
\]
in which the coefficients are
\begin{align*}
B'^{\bsi}_a(q) =&\left.\frac{\frac{\partial H}{\partial \tau_a}}{\frac{\partial H}{\partial \hat x}}\right|_{(X,Y)=(X_{\bsi}(q),Y_{\bsi}(q))},\\
C'^{\bsi}_{ab}(q) =&\left.\frac{\frac{\partial H}{\partial \tau_a}\frac{\partial H}{\partial \tau_b}}{(\frac{\partial H}{\partial \hat x})^2}\right|_{(X,Y)=(X_{\bsi}(q),Y_{\bsi}(q))}.
\end{align*}
\end{lemma}

\begin{lemma}\label{lm:d}
\[
d(\frac{d\hat y}{d\hat x})=\sum_{\bsi \in I_\Sigma}\frac{h^{\bsi}_1 \theta^0_{\bsi}}{2},\quad \
d\left(\frac{\frac{\partial\Phi}{\partial \tau_a}}{d\hat x}\right)=\sum_{\bsi\in I_\Sigma}B'^{\bsi}_a(q) \frac{h^{\bsi}_1\theta_{\bsi}^0}{2},\quad \
\frac{\partial^2\Phi}{\partial \tau_a\partial \tau_b}=\sum_{\bsi \in I_\Sigma} C'^{\bsi}_{ab}(q)\cdot \frac{h^{\bsi}_1\theta^0_{\bsi}}{2}.
\]
\end{lemma}

\begin{proof}
See the proofs of Proposition 6.4 and Proposition 6.5 of \cite{FLZ}.
\end{proof}

By Lemma \ref{lm:H} and Lemma \ref{lm:d}, we conclude this subsection with the following proposition.
\begin{proposition}
\label{prop:b-canonical}
If
$$
[V_{\bsi}(\btau)]=\sum_{i=1}^{\fg} A^i_{\bsi}(q) \check H_{a_i}\cdot \check H_{b_i} - \sum_{a=1}^{\fp} B^a_{\bsi}(q) \check H_a + C_\bsi(q) \one,
$$
then
$$
\frac{h^{\bsi}_1}{2} \theta_{\bsi}^0=\sum_{i=1}^{\fg} A^i_{\bsi}(q) \frac{\partial^2 \Phi}{\partial \tau_{a_i}\partial \tau_{b_i}}+
\sum_{a=1}^{\fp} B^a_{\bsi}(q)d(\frac{\frac{\partial \Phi}{\partial \tau_a}}{d\hat x}) + C_\bsi(q) d(\frac{d\hat y}{d\hat x}).
$$
\end{proposition}
If the genus $\fg$ of the compactified mirror curve $\overline{C}_q$ is zero then $\theta_{\bsi}^0$, or more generally any
differential of the second kind on $\overline{C}_q$, is exact.

\subsection{Oscillating integrals and the B-model $R$-matrix}
Let $V_{\bsi}=V_{\bsi}(0)$ be the flat basis in $\Jac(W^{\bT'}_q)$ such that when $q=0$, $V_{\bsi}(\bp_{\bsi'})=\delta_{\bsi, \bsi'}$. Then define
\begin{align}
\label{eqn:B-S}
\check{S}_{\bsi'}^{\spa {\tbsi}}(z) = (2\pi z)^{-\frac{3}{2}}\int_{\tGa_{\bsi}} e^{-\frac{W_q^{\bT'}}{z}} \overline V_{\bsi'}\Omega,\qquad
\check{S}_{\hat \bsi'}^{\spa {\tbsi}}(z) = (2\pi z)^{-\frac{3}{2}}\sqrt{\Delta^{\bsi'}}\int_{\tGa_{\bsi}} e^{-\frac{W^{\bT'}_q}{z}} \overline V_{\bsi'}\Omega.
\end{align}
\begin{remark}
We use the index $\tbsi$ with tilde to indicate that the integral $\check S_{\bsi'}^{\spa \tbsi}$ is not equal to $S_{\bsi'}^{\spa \bsi}$ in which one inserts the classical canonical basis $\phi_\bsi$ (Equation \eqref{eqn:S}) -- the insertion needs modification by a characteristic class involving the Gamma function in \cite{Iri} to make it equal to $\check S_{\bsi'}^{\spa \tbsi}$. We omit the proof of this fact since it is not directly related to the proof of the Remodeling Conjecture.
\end{remark}
The matrix $\check{S}$ plays the role of the fundamental solution of the B-model
quantum differential equation. The matrix $\check S_{\hat\bsi'}^{\spa
  \tbsi}(z)$ has the following asymptotic expansion.
\begin{proposition}
  \label{prop:check-S}
$$
\check S_{\hat \bsi'}^{\spa \tbsi}(z)\sim\sum_{\bsi''\in I_\Si} \Psi_{\bsi'}^{\spa \bsi''}\check R_{\bsi''}^{\spa \bsi}(z) e^{-\frac{\check u^\bsi}{z}},
$$
where $\Psi_{\bsi'}^{\ \bsi''}$ is the matrix such that $\hat{\phi}_{\bsi'}=\sum_{\bsi''}\Psi_{\bsi'}^{\ \bsi''} \hat{\phi}_{\bsi''}(\btau)$ as in
Equation \eqref{eqn:Psi-matrix}, and $\check u^\bsi$ is the critical
value of $W^{\bT'}_q$ at $\bp_\bsi$. Furthermore, the matrix $\check R_{\bsi''}^{\spa \bsi}=\delta_{\bsi''}^{\spa \bsi}+O(z)$.
\end{proposition}

\begin{proof}
Under the isomorphism \eqref{eqn:QH=Jac},
$\sqrt{\Delta^{\bsi'}}\cdot \overline{V_{\bsi'}}=\sum_{\bsi''\in I_\Si}\Psi_{\bsi'}^{\spa \bsi''} \sqrt{\Delta^{\bsi''}(\btau)}\cdot \overline{V_{\bsi''}(\btau)}$. By stationary phase expansion,

\begin{align}
  \label{eqn:stationary}
\int_{\tGa_\bsi} e^{-\frac{W^{\bT'}_q}{z}}\overline{V_{\bsi''}(\btau)}\sim\frac{(2\pi z)^{ \frac{3}{2} } }{\sqrt{\det\Hess(W^{\bT'}_q)(\bp_{\bsi''})}}(\delta_{\bsi''}^{\spa \bsi}+O(z)).
\end{align}

Then
\begin{align}
\nonumber
\check S_{\hat \bsi'}^{\spa \tbsi}(z)&=(2\pi z)^{-\frac{3}{2}}\sum_{\bsi''\in I_\Si} \Psi_{\bsi'}^{\spa \bsi''}\sqrt{\Delta^{\bsi''}}\cdot \int_{\tGa_{\bsi}} e^{-\frac{W^{\bT'}_q}{z}} \overline {V_{\bsi''}(\btau)}\Omega\\
&\sim  \sum_{\bsi''\in I_\Si} \Psi_{\bsi'}^{\spa \bsi''}\cdot
e^{-\frac{{\check u}^\bsi}{z}} (\delta_{\bsi''}^{\spa \bsi}+O(z)). \label{eqn:R-expansion-2}
\end{align}
Here we use the fact that $\Delta^{\bsi''}(\btau)=\det\Hess(W^{\bT'}_q)(\bp_{\bsi''})$ (Equation \eqref{eqn:Delta-Hess}). So the matrix $\check S$ can be asymptotically expanded in the desired form, and the matrix $\check R$ is
$$
\check R_{\bsi'}^{\spa \bsi}(z)=\sum_{k=0}^\infty (\check R_k)_{\bsi'}^{\spa \bsi} z^k=\delta_{\bsi'}^{\spa \bsi} + O(z).
$$
\end{proof}

Given any $\bsi\in I_\Si$, define $\htheta^0_\bsi=\theta^0_\bsi$. For any positive integer $k$, define
$$
\hat \xi^k_{\bsi}: =(-1)^k(\frac{d}{d\hat x})^{k-1} \frac{\theta_{\bsi}^0}{d\hat x},\quad
\htheta^k_\bsi: =d\hxi^k_\bsi.
$$
Notice that $\hxi^k_\bsi$ is a meromorphic function on $\overline C_q$.   Let
\begin{equation}\label{eqn:theta-z}
\theta_\bsi(z):=\sum_{k=0}^\infty \theta_\bsi^k z^k,\quad \htheta_\bsi(z):=\sum_{k=0}^\infty \htheta^k_\bsi z^k.
\end{equation}
Following Eynard-Orantin \cite{EO15}, let
$$
f_{\bsi'}^{\spa \bsi}(u) =
\frac{e^{u\check{u}^{\bsi}}}{2\sqrt{\pi u}}
\int_{\Gamma_{\bsi}} e^{-u\hat{x}}\theta^0_{\bsi'}.
$$
Assume
$$
[V_{\bsi'}(\btau)]=\sum_{i=1}^{\fg} A^i_{\bsi'}(\btau)\check H_{a_i}\cdot \check H_{b_i}-
\sum_{a=1}^{\fp}  B^a_{\bsi'}(\btau) \check H_b+C_{\bsi'}(\btau) \one,
$$
then by Proposition \ref{prop:b-canonical}, we have
\begin{align*}
\int_{\tGa_\bsi} e^{-\frac{W^{\bT'}_q}{z}} \overline{V_{\bsi'}(\btau)}\Omega
=& (z^2  \sum_{i=1}^{\fg}A^i_{\bsi'} (\btau) \frac{\partial^2}{\partial \tau_{a_i}\partial \tau_{b_i}}
+ z\sum_{a=1}^{\fp}  B^a_{\bsi'}(\btau) \frac{\partial}{\partial \tau_a}+C_{\bsi'}(\btau))\int_{\tGa_\bsi} e^{-\frac{W^{\bT'}_q}{z}} \Omega \\
=& (z^2  \sum_{i=1}^{\fg} A^i_{\bsi'} (\btau) \frac{\partial^2}{\partial \tau_{a_i}\partial \tau_{b_i}}
+ z\sum_{a=1}^{\fp} B^a_{\bsi'}(\btau) \frac{\partial}{\partial \tau_a}+C_{\bsi'}(\btau)) 2\pi \sqrt{-1}\int_{\Gamma_\bsi}e^{-\frac{\hat x}{z}}\Phi\\
=& 2\pi \sqrt{-1}z^2\int_{\Gamma_\bsi}e^{-\frac{\hat x}{z}} (\sum_{i=1}^{\fg}A^i_{\bsi'} (\btau) \frac{\partial^2\Phi}{\partial \tau_{a_i}\partial \tau_{b_i}}
+ \sum_{a=1}^{\fp}B^a_{\bsi'}(\btau) d(\frac{\partial \hat y}{\partial \tau_a})+C_{\bsi'}d (\frac{d \hat y}{d\hat x})) \\
=& 2\pi\sqrt{-1} z^2 \int_{\Gamma_\bsi} e^{-\frac{\hat x}{z}}\frac{h_1^{\bsi'} \theta_{\bsi'}^0}{2}=2\sqrt{-1}(\pi z)^{\frac{3}{2}}e^{-\frac{\check{u}^\bsi}{z}} h_1^{\bsi'}f_{\bsi'}^{\spa \bsi}(\frac{1}{z}).
\end{align*}
From Equation \eqref{eqn:stationary}
$$
\int_{\tGa_\bsi} e^{-\frac{W^{\bT'}_q}{z}} \overline{V_{\bsi'}(\btau)}\Omega\sim \frac{(2\pi z)^\frac{3}{2}e^{-\frac{\check u^{\bsi}}{z}}}{\sqrt{\det\Hess({W^{\bT'}_q})(\bp_{\bsi'})}} \check R^{\spa \bsi}_{\bsi'}(z),
$$
and by
$h_1^{\bsi'}=\sqrt{\frac{2}{-\det\Hess(W^{\bT'}_q)(\bp_{\bsi'})}}$ (Equation \eqref{eqn:h-Hess}), it is
easy to see
$$
\check{R}_{\bsi'}^{\spa \bsi}(z) = f_{\bsi'}^{\spa \bsi}(\frac{1}{z}).
$$

Following Eynard \cite{E11}, define Laplace transform of the Bergman kernel
\begin{equation}
\check{B}^{\bsi,\bsi'}(u,v,q) :=\frac{uv}{u+v} \delta_{\bsi,\bsi'}
+ \frac{\sqrt{uv}}{2\pi} e^{u\check u^\bsi + v \check u^{\bsi'}}\int_{p_1\in \Gamma_\bsi}\int_{p_2\in \Gamma_{\bsi'}}
B(p_1,p_2) e^{-ux(p_1) -v x(p_2)},
\end{equation}
where $\bsi,\bsi'\in I_\Si$. By \cite[Equation (B.9)]{E11},
\begin{equation}\label{eqn:B-ff}
(u+v)\check{B}^{\bsi,\bsi'}(u,v,q) = uv(\delta_{\bsi,\bsi'}
-\sum_{\bsi''\in I_\Si} \check R^{\spa \bsi}_{\bsi''}(\frac{1}{u})\check R^{\spa \bsi'}_{\bsi''}(\frac{1}{v})).
\end{equation}
Setting $u=-v$, we conclude that $(\check R^*(\frac{1}{u}) \check R(-\frac{1}{u}))^{\bsi\bsi'}=\{\sum_{\bsi''\in I_\Si} \check R^{\spa \bsi}_{\bsi''}(\frac{1}{u})\check R^{\spa \bsi'}_{\bsi''}(-\frac{1}{u})\}=\delta^{\bsi\bsi'}$. This shows $\check R$ is unitary.

The following proposition is a consequence of Lemma 6.5 in \cite{FLZ} and Equation (\ref{eqn:B-ff}).
\begin{proposition}
\label{prop:b-open-leaf}
$$
\theta_\bsi(z)=\sum_{\bsi'\in I_\Si}\check R_{\bsi'}^{\spa \bsi} (z) \hat \theta_{\bsi'}(z).
$$
\end{proposition}

The following proposition is due to Dubrovin \cite{D93} and Givental
\cite{G97, G98, G01}, but we only consider the small phase space $H_{\CR,\bT'}^2(\cX;\bC)$.
\begin{proposition}
\label{prop:fundamental}
Let $\hat \phi^B_{\bsi}$ be the Jacobian ring element corresponding to $\hat\phi_{\bsi}$ under the isomorphism \eqref{eqn:qcoh-Tprime}.
Assume $\check \rA^\tbsi(z)(\btau)=\sum_{\bsi'\in I_\Si} \check \rA_{\hat \bsi'}^{\spa \tilde \bsi}(z)(\btau)\hat \phi_{\bsi'}^B$
has the following asymptotic expansions where $\btau\in H^2_{\CR,\bT'}(\cX,\bC)$
$$
\check \rA^{\spa \tbsi}_{\hat \bsi'}(z)\sim\sum_{\bsi''\in I_\Si} \Psi_{\bsi'}^{\spa  \bsi''} \check \rB_{\bsi''}^{\spa \bsi}(z) e^{-\frac{\check s^\bsi}{z}},
$$
such that
\begin{itemize}
\item $\Psi$ is the transition matrix defined in Equation \eqref{eqn:Psi-phi};
\item The matrix function $\check \rB_{\bsi''}^{\spa \bsi}(z)=\delta_{\bsi''}^{\spa \bsi} + O(z)$ is unitary
   $$\sum_{\bsi''\in I_\Si} \check \rB_{\bsi''}^{\spa \bsi}(z) \check \rB_{\bsi''}^{\spa \bsi'}(-z)
=\delta^{\bsi,\bsi'};$$
\item The functions $\check s^\bsi$ and the canonical coordinates
 $u^\bsi$ differ by constants, i.e. $\frac{\partial \check s^\bsi}{\partial \tau_i}=\frac{\partial u^\bsi}{\partial \tau_i}$ for $i=1,\dots,p$.
\end{itemize}
If each function $\rA^\tbsi$ satisfies the quantum differential equations for $1\leq i \leq p$
$$
-z \frac{\partial}{\partial \tau_i}\check  \rA^{\tbsi}=\check
H_i\cdot \check \rA^\tbsi,
$$
then $\check \rB$ is unique up to a right multiplication of $\exp(\sum_{i=1}^\infty a_{i} z^{2i-1})$, where $a_{i}$ is a constant diagonal matrix.
\end{proposition}
\begin{proof}
This proof is essentially the same as the Proposition in \cite[Section 1.3 (p1269)]{G01}. The only minor difference is that we only considers
the small phase space. Let $\check s$ be the diagonal matrix with the
diagonal elements $\{\check s^\bsi\}_{\bsi\in \bSi}$. Notice that substituting the series into the quantum differential equations gives
$$
(\frac{\partial}{\partial \tau_i}+\Psi\frac{ \partial\Psi}{\partial \tau_i})\check \rB_{k-1}=-[\frac{\partial \check s}{\partial \tau_i},\check \rB_k].
$$
This gives a recursion which determines $\check \rB$. The off-diagonal terms in $\check \rB_k$ are directly expressed in $\check \rB_{k-1}$ algebraically, and the diagonal terms could be solved by integration, noting that $[\frac{\partial \check s}{\partial \tau_i},\check \rB_k]$ has vanishing diagonal. 

Let $\rP^{\bsi,\bsi'}(z)=\sum_{k\geq 0} (\rP_k)^{\bsi,\bsi'} z^k
=\sum_{\bsi''\in I_\Si} \check \rB^{\spa \bsi}_{\bsi''}(z) \check \rB^{\spa \bsi'}_{\bsi''}(-z)$, then the quantum differential equations produce
$$
-[\frac{\partial \check s}{\partial \tau_i}, \rP_k]=d  \rP_{k-1} + [\Psi\frac{\partial \Psi}{\partial \tau_i}, \rP_{k-1}].
$$
Note that $\check \rB$ is unitary, i.e. $\rP_k=0$ for $k\geq 1$ and $\rP_0=I$. For $k$ odd, the equation above ensures that $\rP_k=0$ from $\rP_{k-1}=0$ (or $I$ when $k=1$) since $\rP_k$ is anti-symmetric. For even $k$, the ambiguity of the integrating constants in determining the diagonal terms of $\rB_k$ in the process above is fixed by
$$
0=(\rB_k)_{\bsi}^{\spa \bsi'}+(\rB_k)_{\bsi'}^{\spa \bsi}+ \text{terms involving $\rB_i$, $i=1,\dots, k-1$}.
$$
We see that this is equivalent to a right multiplication of $\exp(\sum_{i=1}^\infty a_iz^{2i-1})$.
\end{proof}

Since $\check u^\bsi$ is a critical value,
$$
\frac{\partial \check u^\bsi}{\partial \tau_i}=\frac{dW^{\bT'}_q(\bp_\bsi, \btau)}{d \tau_i}
=\frac{\partial W^{\bT'}_q}{\partial \tau_i}(\bp_\bsi).
$$
The Jacobian ring element $\check H_i=[\frac{\partial W^{\bT'}_q}{\partial \tau_i}]$ corresponds
to $H_i$ in the quantum cohomology. Then by the following identity
$$
[\frac{\partial W^{\bT'}_q}{\partial \tau_i}]=\sum_{\bsi\in I_\Si} \frac{\partial W^{\bT'}_q}{\partial \tau_i}(\bp_\bsi) [V_\bsi(\btau)],
$$
we have
$$
\frac{\partial u^\bsi}{\partial \tau_i}=\frac{\partial \check u^\bsi}{\partial \tau_i},
$$
which implies the critical values are canonical coordinates.
The function $\check S^{\tbsi} =\sum_{\bsi' \in I_\Si} \check S_{\hat \bsi'}^{\spa \tbsi} \hat \phi^{\bsi'}$ is a solution to the quantum differential equation
$$
-z \frac{\partial}{\partial \tau_i}\check S^{\tbsi}=(\frac{\partial W^{\bT'}_q}{\partial \tau_i }) \check S^\tbsi,
$$
For all $\bsi \in I_\Si$, $\check S^{\tbsi}$ satisfy the condition
of Proposition \ref{prop:fundamental}.

\subsection{The Eynard-Orantin topological recursion and the B-model graph sum}
\label{sec:eynard-orantin}

Let $\omega_{g,n}$ be defined recursively by the Eynard-Orantin topological recursion \cite{EO07}:
$$
\omega_{0,1}=0,\quad  \omega_{0,2}=B(p_1,p_2).
$$
When $2g-2+n>0$,
\begin{eqnarray*}
\omega_{g,n}(p_1,\ldots, p_n) &=& \sum_{\bsi\in I_\Si}\Res_{p \to p_\bsi}
\frac{\int_{\xi = p}^{\bar{p}} B(p_n,\xi)}{2(\Phi(p)-\Phi(\bar{p}))}
\Big( \omega_{g-1,n+1}(p,\bar{p},p_1,\ldots, p_{n-1}) \\
&&\quad\quad  + \sum_{g_1+g_2=g}
\sum_{ \substack{ I\cup J=\{1,..., n-1\} \\ I\cap J =\emptyset } } \omega_{g_1,|I|+1} (p,p_I)\omega_{g_2,|J|+1}(\bar{p},p_J)\Big). 
\end{eqnarray*}

Following \cite{DOSS}, the B-model invariants $\omega_{g,n}$ are expressed in terms of graph sums. We first introduce some notation.
\begin{itemize}
\item  For any $\bsi\in I_{\Sigma}$, we define
\begin{equation}
\check{h}^{\bsi}_{k} :=\frac{(2k-1)!!}{2^{k-1}}h^\bsi_{2k-1}.
\end{equation}
Then
$$
\check{h}^{\bsi}_k = [u^{1-k}]\frac{u^{3/2}}{\sqrt{\pi}} e^{u\check{u}^{\bsi}}
\int_{p\in \Gamma_{\bsi}}e^{-u \hat{x}(p)}\Phi(p).
$$

\item For any $\bsi,\bsi'\in I_\Sigma$, we expand
$$
B(p_1,p_2) =\Big( \frac{\delta_{\bsi,\bsi'}}{ (\zeta_\bsi-\zeta_{\bsi'})^2}
+ \sum_{k,l\in \bZ_{\geq 0}} B^{\bsi,\bsi'}_{k,l} \zeta_\bsi^k \zeta_{\bsi'}^l \Big) d\zeta_\bsi d\zeta_{\bsi'},
$$
near $p_1=p_{\bsi}$ and $p_2=p_{\bsi'}$, and define
\begin{equation}\label{eqn:BcheckB}
\check{B}^{\bsi,\bsi'}_{k,l} := \frac{(2k-1)!! (2l-1)!!}{2^{k+l+1}} B^{\bsi,\bsi'}_{2k,2l}.
\end{equation}
Then
\[
\check{B}^{i,j}_{k,l}=[u^{-k}v^{-l}]\left(\frac{uv}{u+v}(\delta_{\bsi,\bsi'}
-\sum_{\bgamma\in I_\Si} f^{\ \bsi}_{\bgamma}(u)f^{\ \bsi'}_{\bgamma}(v))\right)
=[z^{k}w^{l}]\left(\frac{1}{z+w}(\delta_{\bsi,\bsi'}
-\sum_{\bgamma\in I_\Si} f^{\ \bsi}_{\bgamma}(\frac{1}{z})f^{\ \bsi'}_{\bgamma}(\frac{1}{w}))\right).
\]
\end{itemize}

Given a labeled graph $\vGa \in \bGa_{g,n}(\cX)$ with
$L^o(\Ga)=\{l_1,\ldots,l_n\}$, and $\bullet=\bu$ or $O$, 
we define its weight to be
\begin{eqnarray*}
w_B^{\bullet}(\vGa) &=& (-1)^{g(\vGa)-1}\prod_{v\in V(\Gamma)} \Big(\frac{h^{\bsi(v)}_1}{\sqrt{-2}}\Big)^{2-2g-\val(v)} \langle \prod_{h\in H(v)} \tau_{k(h)}\rangle_{g(v)}
\prod_{e\in E(\Gamma)} \check{B}^{\bsi(v_1(e)),\bsi(v_2(e))}_{k(e),l(e)}  \\
&& \cdot \prod_{l\in \cL^1(\Gamma)}(\check{\cL}^1)^{\bsi(l)}_{k(l)}
\prod_{j=1}^n (\check{\cL}^\bullet)^{\bsi(l_j)}_{k(l_j)}(l_j)
\end{eqnarray*}
where
\begin{itemize}
\item (dilaton leaf)
$$
(\check{\cL}^1)^{\bsi}_k = \frac{-1}{\sqrt{-2}}\check{h}^{\bsi}_k.
$$
\item (descendant leaf)
$$
(\check{\cL}^\bu)^{\bsi}_k(l_j) =  \frac{1}{\sqrt{-2}} \theta_{\bsi}^k(p_j).
$$
\item (open leaf)
Let
$$
\psi_\ell:= \frac{1}{\fm} \sum_{k=0}^{\fm-1}\omega_\fm^{-k\ell} \one'_{\frac{k}{\fm}}\in H^*_\CR(\cB\bmu_\fm;\bC), \quad \ell=0,1, \ldots, \fm-1,
$$
where $\omega_\fm= e^{2\pi\sqrt{-1}/\fm}$. We may regard $\ell\in \bmu_\fm^*$ such that $\ell(\one'_{\frac{k}{\fm}})=\omega_\fm^{-k\ell}.$
$$
(\check{\cL}^O)^{\bsi}_k (l_j) = \frac{1}{\sqrt{-2}} \sum_{\ell\in \bmu_\fm^*}\int_0^{X'_j} \rho_\ell^*(\theta_{\bsi}^k)\psi_\ell.
$$

\end{itemize}

In our notation \cite[Theorem 3.7]{DOSS} is equivalent to:
\begin{theorem}[Dunin-Barkowski--Orantin--Shadrin--Spitz \cite{DOSS}] \label{thm:DOSS}
For $2g-2+n>0$,
$$
\omega_{g,n} = \sum_{\Gamma \in \bGa_{g,n}(\cX)}\frac{w_B^{\bu}(\vGa)}{|\Aut(\vGa)|}.
$$
\end{theorem}
We now consider the unstable case $(g,n)=(0,2)$.
Recall that $dx =-\frac{dX}{X}$ is a meromorphic 1-form on $\bar{C}_q$, and
$\frac{d}{dx}=-X\frac{d}{dX}$ is a meromorphic vector field on $\bar{C}_q$.
Define
\begin{equation}\label{eqn:C}
C(p_1,p_2):=
(-\frac{\partial}{\partial x(p_1)}-\frac{\partial}{\partial x(p_2)})\Big(\frac{\omega_{0,2}}{dx(p_1) dx(p_2)}\Big)(p_1,p_2)
d(x(p_1))(dx(p_2)).
\end{equation}
Then $C(p_1,p_2)$ is meromorphic on $(\bar{C}_q)^2$ and is holomorphic
on $(\overline{C}_q \setminus \{p_\bsi:\bsi \in I_\Si\})^2$.
\begin{lemma}\label{lemm:C}
$$
C(p_1,p_2)= \frac{1}{2} \sum_{\bsi \in I_\Si} \theta_\bsi^0(p_1) \theta_\bsi^0(p_2).
$$
\end{lemma}
\begin{proof}
For any $\bsi,\bsi'\in I_\Si $, we compute their Laplace transforms
\begin{align*}
&\int_{p_1\in \Gamma_\bsi}\int_{p_2\in \Gamma_{\bsi'}}
e^{-\frac{x(p_1)-\check u^\bsi}{z_1}-\frac{x(p_2)-\check u^{\bsi'}}{z_2}} C(p_1, p_2)\\
=& (-\frac{z_1+z_2}{z_1z_2})\int_{p_1\in \Gamma_\bsi}\int_{p_2\in \Gamma_{\bsi'}}
e^{-\frac{x(p_1)-\check u^\bsi}{z_1}-\frac{x(p_2)-\check u^{\bsi'}}{z_2}}\omega_{0,2}\\
=& \frac{2\pi}{\sqrt{z_1z_2}} \sum_{\bsi'' \in I_\Si}\check
   R^{\spa \bsi}_{\bsi''}(z_1)\check R^{\spa \bsi'}_{\bsi''}(z_2)\\
=&\frac{1}{2}\sum_{\bsi''\in I_\Si}\int_{p_1\in \Gamma_\bsi}\int_{p_2\in \Gamma_{\bsi'}}
e^{-\frac{x(p_1)-\check u^\bsi}{z_1}-\frac{x(p_2)-\check u^{\bsi'}}{z_2}} \theta_{\bsi''}^0(p_1) \theta_{\bsi''}^0(p_2).
\end{align*}
Define
\begin{align*}
\omega&=C(p_1,p_2)-\frac{1}{2} \sum_{\bsi \in I_\Si}
        \theta_\bsi^0(p_1) \theta_\bsi^0(p_2).
\end{align*}
Since for $i=1,\dots,\fg$, $\int_{p_2\in A_i} \omega_{0,2}(p_1,p_2)=0,
\ \int_{A_i}
\theta_\bsi^0=0$, we have $\int_{p_2\in A_i}\omega =0$, and the
following residue $1$-form has
$$
\int_{p_2\in A_i} \Res_{p_1\to p_\bsi} \zeta_{\bsi}(p_1) \omega(p_1,p_2)=0,
$$
for all $i=1,\dots, \fg$. Notice that the $1$-form $\Res_{p_1\to
  p_\bsi} \zeta_{\bsi}(p_1)\omega(p_1,p_2)$ has no poles, otherwise
a possible double pole at $p_{{\bsi'}}$ implies non-zero Laplace
transform of $\omega$ at $\Gamma_\bsi\times \Gamma_{\bsi'}$. It
follows from the vanishing A-cycles integrals that
$$
\Res_{p_1\to p_\bsi} \zeta_{\bsi}(p_1) \omega(p_1,p_2)=0,
$$ and then $\omega$ does
not have any poles. Therefore by the vanishing A-periods of $\omega$
we know $\omega=0$.
\end{proof}

\subsection{B-model open potentials}\label{sec:B-potential}
In this section, we fix $\su_1=1$ and $\su_2=f$. Choose $\delta>0$, $\epsilon>0$ sufficiently small, such that
for $|q|<\epsilon$, the meromorphic function
$\hX: \bar{C}_q \to \bC\cup \{\infty\}$
restricts to an isomorphism
$$
\hX_q^\ell: D_q^\ell\to D_\delta=\{ \hX\in \bC: |X|<\delta\},
$$
where $D_q^\ell$ is an open neighborhood of $\bar{p}_\ell:=\bar{p}_\ell^{(\tau_0,\si_0)}\in \hX^{-1}(0)$, $\ell=0,\ldots,m-1$.
Define
$$
\rho_q^{\ell_1,\ldots,\ell_n}:= (\hX_q^{\ell_1})^{-1}\times \cdots \times (\hX_q^{\ell_n})^{-1}:
(D_\delta)^n\to D_q^{\ell_1}\times \cdots \times D_q^{\ell_n} \subset (\bar{C}_q)^n.
$$

\begin{enumerate}
\item (disk invariants)
At $q=0$, $\hY(\bar{p}_\ell)^{\fm}=-1$ for $\ell=0,\ldots,\fm-1$. When
$\epsilon$ and $\delta$ are sufficiently small,
$\hY(\rho_q^\ell(\hX))\in \bC\setminus[0,\infty)$. Choose a branch of logarithm
$\log:\bC\setminus [0,\infty)\to (0,2\pi)$, and define
$$
\hy_q^\ell (\hX) = -\log \hY (\rho_q^\ell(\hX)).
$$
The function $\hy_q^\ell(X)$ depends on the choice of logarithm, but
$\hy_q^\ell(X)-\hy_q^\ell(0)$ does not. $d\hx = -d\hX/\hX$ is a meromorphic
1-form on $\bC$ with a simple pole at $\hX=0$, and
$$
(\hy_q^\ell(X)-\hy_q^\ell(0))d\hx
$$
is a holomorphic 1-form on $D_\delta$.

Define the {\em B-model disk potential} by
$$
\check{F}_{0,1}(q;\hX):= \sum_{\ell\in I_\Si}\int_0^{\hX} (\hy_q^\ell(X')-\hy_q^\ell(0))(-\frac{dX'}{X'}) \cdot \psi_\ell,
$$
which takes values in $H^*(\cB\bmu_\fm;\bC)$.

\item (annulus invariants)
$$
(\rho_q^{\ell_1,\ell_2})^*\omega_{0,2}-\frac{d\hX_1d\hX_2}{(\hX_1-\hX_2)^2}
$$
is holomorphic on $D_\delta\times D_\delta$.  Define the {\em B-model annulus potential} by
$$
\check{F}_{0,2}(q;\hX_1,\hX_2):=
\sum_{\ell_1,\ell_2\in \bmu_\fm^*}
\int_0^{\hX_1}\int_0^{\hX_2} \Big((\rho_q^{\ell_1,\ell_2})^*\omega_{0,2}-\frac{dX_1' dX_2'}{(X_1'-X_2')^2}\Big) \cdot
\psi_{\ell_1}\otimes\psi_{\ell_2},
$$
which takes values in $H^*(\cB\bmu_\fm;\bC)^{\otimes 2}$.
\item For $2g-2+n>0$, $(\rho_q^{\ell_1,\ldots,\ell_n})^*\omega_{g,n}$ is
holomorphic on $(D_\delta)^n$. Define
$$
\check{F}_{g,n}(q;\hX_1,\ldots,\hX_n):=
\sum_{\ell_1,\ldots,\ell_n\in \bmu_\fm^*}
\int_0^{\hX_1}\cdots\int_0^{\hX_n} (\rho_q^{\ell_1,\ldots,\ell_n})^*\omega_{g,n} \cdot
\psi_{\ell_1}\otimes\cdots \otimes\psi_{\ell_n},
$$
which takes values in $H^*(\cB\bmu_\fm;\bC)^{\otimes n}$.
\end{enumerate}
For $g\in\bZ_{\geq 0}$ and $n\in \bZ_{>0}$,
$\check{F}_{g,n}(q;\hX_1,\ldots,\hX_n)$ is holomorphic on $B_\epsilon \times (D_\delta)^n$ when
$\epsilon, \delta>0$ are sufficiently small. By construction,
the power series expansion of  $\check{F}_{g,n}(q;\hX_1,\ldots,\hX_n)$ only involves
positive powers of $\hX_i$.

For $k\in \bZ_{\geq 0}$, define
$$
\xi_\bsi^k(\hX):=\sum_{\ell \in \bmu_\fm^*} \int_0^{\hX}(\rho_q^\ell)^*\theta_\bsi^k\psi_\ell,\quad
\xi_\bsi(z, \hX):=\sum_{\ell \in\bmu_\fm^*}
\int_0^{\hX} (\rho_q^\ell)^*\hat \theta_\bsi(z)\psi_\ell,
$$
where $\hat{\theta}_\bsi(z)$ is defined as in Equation \eqref{eqn:theta-z}.

\subsection{B-model free energies}

In this section, $g>1$ is an integer.

\begin{definition}[{cf. \cite[Definition 4.3]{EO07}}] \label{df:stable-Fg-B}
The {\em $B$-model genus $g$ free energy} is  defined to be
$$
\check{F}_g:= \frac{1}{2-2g} \sum_{\bsi\in I_\Si}
\Res_{p\to p_{\bsi}} \omega_{g,1}(p)\tPhi_{\bsi}(p).
$$
where $\tPhi_{\bsi}$ is a function
defined on an open neighborhood of $p_{\bsi}$ in $\Si_q$ such that
$d\tPhi_{\bsi}=\Phi$.
\end{definition}
Notice that the definition does not depend on the choice of
$\tPhi_{\bsi}$.

\begin{proposition} \label{prop:stable-Fg-B}
$$
\check{F}_g = \frac{1}{2-2g}\sum_{\vGa\in \Gamma_{g,1}(\cX)}
\frac{w^\bu_B(\vGa)\Big|_{ (\check{\cL}^{\bu})^{\bsi}_k(l_1) = (\check{\cL}^1)^{\bsi}_k } }{|\Aut(\vGa)|}.
$$
\end{proposition}
\begin{proof}
Recall that
$$
(\check{\cL}^\bu)^{\bsi}_k(l_1)  =  \frac{1}{\sqrt{-2}}  \theta^k_{\bsi}(p_1),\quad
(\check{\cL}^1)^{\bsi}_k = \frac{-1}{\sqrt{-2}} \check{h}^{\bsi}_k.
$$
By the graph sum formula of $\omega_{g,1}$
(Theorem \ref{thm:DOSS}) and the definition of $\check{F}_g$ (Definition
\ref{df:stable-Fg-B}),  it suffices to show that
$$
\Res_{p\to p_{\bsi}} \theta^k_{\bsi}(p)\tPhi_{\bsi}(p)
= -\check{h}_k^{\bsi}.
$$
Near $p_{\bsi}$, we have
$$
\theta^k_{\bsi} = \Big(\frac{-(2k+1)!!}{2^k \zeta_{\bsi}^{2k+2} } + f(\zeta_{\bsi})\big) d\zeta_{\bsi}
$$
where $f(\zeta_{\bsi})$ is analytic around $p_{\bsi}$, and
$$
d\tPhi_{\bsi} = \hat{y}d\hat{x} = (\check{v}^{\bsi} +\sum_{d=1}^\infty h^{\bsi}_d \zeta_{\bsi}^d)(2\zeta_{\bsi} d\zeta_{\bsi}),
$$
so up to a constant,
$$
\tPhi_{\bsi} = \check{v}^{\bsi} +\sum_{d=1}^\infty \frac{2h^{\bsi}_d}{d+2}\zeta_{\bsi}^{d+2}.
$$
Therefore,
$$
\Res_{p\to p_{\bsi}} \theta^k_{\bsi}(p)\tPhi_{\bsi}(p)
= \frac{-(2k-1)!!}{2^{k-1}} h_{2k-1}^{\bsi}= -\check{h}_k^{\bsi}.
$$
\end{proof}

\section{All Genus Mirror Symmetry}
\label{sec:mirror-symmetry}

\subsection{Identification of A-model and B-model $R$-matrices}\label{Id fundamental}
Recall that there is an isomorphism of Frobenius algebras (cf. Equation  \eqref{eqn:qcoh-Tprime} in Section \ref{sec:qcoh}):
$$
QH^*_{CR,\bT'}(\cX)\Bigr|_{\btau=\btau(q), Q=1} \cong \Jac(W_q^{\bT'}).
$$
Equation \eqref{eqn:Delta-Hess} and Lemma \ref{lm:pairing} imply
$$
h_1^{\bsi}(q)=\sqrt{ \frac{2}{\frac{d^2\hx}{d\hy^2}(\check{v}^{\bsi}) } } = \left. \sqrt{\frac{-2}{\Delta^{\bsi}(\btau) }} \, \right|_{\btau=\btau(q), Q=1} .
$$
We are working with {\em non-conformal} Frobenius manifolds, and the solution of the quantum differential equation is not unique.
The ambiguity is fixed by the following theorem.
\begin{theorem}
For any $\bsi=(\si,\alpha)$ and $\bsi'=(\si',\alpha')$,
$$
R_{\bsi'}^{\spa \bsi}(z)\big|_{t=\btau,Q=1}=\check R_{\bsi'}^{\spa \bsi}(-z).
$$
\end{theorem}
\begin{proof}
We know that $S_{\hat \bsi'}^{\spa \bsi}(-z)$ and $\check S_{\hat \bsi'}^{\spa \tbsi}(z)$ satisfy the conditions of Proposition
\ref{prop:fundamental} by setting $\check A_{\hat \bsi'}^{\spa\tbsi}(z)=S_{\hat\bsi'}^{\spa \bsi}(-z)$ or $\check S_{\hat \bsi'}^{\ \tbsi}(z)$. So we only need to show $R$ and $\check R$ match when $q=0$.
Recall from  Section \ref{sec:toric-degeneration} that
when $q=0$, the compactified mirror curve $\Cbar_q$ degenerates into a nodal curve
$\fC_0= \bigcup_{\si\in \Si(3)} \Cbar_\si$, where the irreducible component $\Cbar_\sigma$ can be identified with the compactified mirror curve of
the affine toric Calabi-Yau 3-orbifold $\cX_\si$ defined by the 3-cone $\si$.  Recall from Section \ref{sec:second-kind} that the $1$-form $\theta^0_{\sigma,\alpha}(0)|_{\Cbar_{\sigma'}}$ vanishes when  $\sigma'\neq \sigma$, and $\theta^0_{\sigma,\alpha}(0)|_{\Cbar_\sigma}$  coincides with $\theta_0^\alpha$ in \cite[Section 6.6]{FLZ}.
As computed in \cite[Theorem 7.5]{FLZ}
\[
\check R_{\bsi'}^{\spa \bsi}(-z)|_{q=0}
=\delta_{\si,\si'}\sum_{h\in G_\sigma}\frac{\chi_\alpha(h) \chi_{\alpha'}(h^{-1})}{|G_\sigma|}\exp\left (\sum_{m\geq 1} \frac{(-1)^m}{m(m+1)}\sum_{i=1}^3 B_{m+1} (c_i(h))(\frac{z}{\w_i(\sigma)})^m\right)
\]
which is precisely $R_{\bsi'}^{\spa \bsi}(z)|_{q=0}$ given in Equation (\ref{eqn:R-at-zero}). Here $\bsi=(\sigma,\alpha),\bsi'=(\sigma',\alpha')$.
\end{proof}

\subsection{Identification of graph sums}
In this subsection, we identify the graph sums on A-model and B-model.

For $l=1,\cdots,n$ and $\bsi\in I_\Si$, let
$$
\tilde{\bu}_l^\bsi(z)=\sum_{a\geq 0}(\tilde{u}_l)^\bsi_az^a:=\sum_{\bsi'\in I_\Si}
\left(\frac{\bu_l^{\bsi'}(z)}{\sqrt{\Delta^{\bsi'}(\btau)} }
S^{\widehat{\underline{\bsi}} }_{\spa \widehat{\underline{\bsi'}}}(z)\right)_+
$$
The identification $R(z)|_{t=\btau,Q=1}=\check{R}(-z)$ implies the following theorem:

\begin{theorem}\label{thm:graph-match}
For any $\vGa\in \bGa_{g,n}(\cX)$,
$$
w^\bu_B(\vGa)|_{ \frac{1}{\sqrt{-2}}\htheta_{\bsi}^a(p_l)=-(\tilde{u}_l)^\bsi_a}
=(-1)^{g(\vGa)-1+n} w^\bu_A(\vGa)\big|_{t=\btau,Q=1},
$$
under the closed mirror map.
\end{theorem}
\begin{proof}
\begin{enumerate}
\item {\em Vertex}.
By the discussion in Section \ref{sec:qcoh} and Section \ref{sec:LG}, $h_1^\bsi=\sqrt{\frac{-2}{\Delta^{\bsi}(\btau)}}$ for any $\bsi\in I_\Si$. So in the B-model vertex term, $\frac{h^\bsi_1}{\sqrt{-2}}=\sqrt{\frac{1}{\Delta^\bsi(\btau)}}$. Therefore the B-model vertex matches the A-model vertex.

\item  {\em Edge}. By the property for $\check{B}^{\bsi,\brho}_{a,b}$,
$$
\check{B}^{\bsi,\brho}_{a,b}=[u^{-a}v^{-b}]\left(\frac{uv}{u+v}(\delta_{\bsi,\brho}
-\sum_{\bgamma\in I_\Si} f^{\spa \bsi}_\bgamma(u)f^{\spa \brho}_\bgamma(v))\right)\\
=[z^{a}w^{b}]\left(\frac{1}{z+w}(\delta_{\bsi,\brho}
-\sum_{\bgamma\in I_\Si} f^{\spa \bsi}_\bgamma(\frac{1}{z})f^{\spa \brho}_\bgamma(\frac{1}{w}))\right).
$$
Therefore, the identification $R(z)_{\brho}^{\spa\bsi}\big|_{t=\btau, Q=1}=\check{R}_\brho^{\spa\bsi}(-z)=f^{\spa \bsi}_{\brho}(-\frac{1}{z})$ gives us $$
\check{B}^{\bsi,\brho}_{a,b} = \cE^{\bsi,\brho}_{a,b}\big|_{t=\btau,Q=1}.
$$

\item {\em Ordinary leaf}. By Proposition \ref{prop:b-open-leaf}, we have the following expression for $\theta_\bsi^a $:
$$
\theta_\bsi^a = \sum_{c=0}^{a}\sum_{\brho\in I_\Si}([z^{a-c}](\check{R}_\brho^{\spa\bsi}(z))\htheta_{\brho}^c.
$$
Notice that $\check{R}(z)=R(-z)\big|_{t=\btau, Q=1}$. So
$$
(\check{\cL^\bu})^{\bsi}_k \Big|_{ \frac{1}{\sqrt{-2}}\hat{\theta}^a_\bsi(p_l)=- (\tu_l)^{\bsi}_a  } = -(\cL^\bu)^{\bsi}_k(l_j)\vert_{t=\btau, Q=1}.
$$

\item {\em Dilaton leaf}. We have the following relation between $\check{h}^\bsi_a$ and $f^{\spa \bsi}_{\brho}(u)$ (see \cite{FLZ})
$$
\check{h}^\bsi_a=[u^{1-a}]\sum_{\brho\in I_\Si}h_1^\brho f^{\spa \bsi}_\brho(u).
$$
By the relation
$$
R_\brho^{\,\ \bsi}(z)\big|_{t=\btau,Q=1} =  f^{\spa\bsi}_\brho(-\frac{1}{z})
$$
and the fact $h^\bsi_1=\sqrt{\frac{-2}{\Delta^\bsi(\btau)}}$, it is easy to see that
the B-model dilaton leaf matches the A-model dilaton leaf.
\end{enumerate}
\end{proof}

\subsection{BKMP Remodeling Conjecture: the open string sector}

In this subsection, we fix $\su_1=1$ and $\su_2=f$. We compare A and
B-model open leafs. The disk potential with respect to the
Aganagic-Vafa brane $\cL$ is given by localization, as in Proposition
\ref{prop:open-descendant} (computed in \cite{FLT}).
\begin{equation}
\label{eqn:disk-localization}
(\hat X\frac{d}{d\hat X})^2F_{0,1}^{\cX,(\cL,f)}(\btau,\tX)= [z^0] \sum_{\bsi'\in I_\Sigma} \txi^{\bsi'} (z,\tX)S(1,\phi_{\bsi'})\Big|_{t=\btau,Q=1}.
\end{equation}
The following theorem is proved by Tseng and the first two authors \cite{FLT}.
\begin{theorem}[Genus zero open-closed mirror symmetry]
\label{thm:AKV-disk}
Under the closed mirror map given by Equation \eqref{eqn:closed-mirror-map} and the open map given by
\begin{equation}
\label{eqn:open-mirror-map}
\log \tX= \log \hat X+\sum_{m=1}^{3} w_i A_i(q),
\end{equation}
we have
$$
F_{0,1}^{\cX,(\cL,f)}(\btau,\tX) =\check{F}_{0,1}(q;\hX)=\sum_{\ell\in \bmu_\fm^*}(
\int_{0}^{\hat X}    \rho_\ell^*( \hat{y}(X')-\hat{y}(0) )(-\frac{d  \hX'}{ \hX'})  \psi_\ell.
$$
\end{theorem}
This theorem, together with Equation \eqref{eqn:disk-localization},
implies that under the open-closed mirror map, as power series in
$\hat X$,
\begin{equation}
\label{eqn:z-disk}
U(z)(\btau,\tX):= \sum_{\bsi'\in I_\Sigma} \txi^{\bsi'} (z,\tX)S(1,\phi_{\bsi'})\big|_{t=\btau, Q=1}=-\sum_{n\geq 0} z^n (-\frac{d}{d\hat x})^n  \frac{d}{d\hat x}\sum_{\ell\in \bmu_\fm^*} \rho_\ell^*(\hat y) \psi_\ell.
\end{equation}
Notice that from Proposition \ref{prop:b-canonical}, if
$$
\hat \phi_{\bsi}(\btau(q))=\sum_{i=1}^{\fg} \hat A^i_{\bsi}(q) H_{a_i}\star_{\btau} H_{b_i}
+ \sum_{a=1}^{\fp} \hat B^a_{\bsi}(q)  H_a + \hat C_\bsi(q) \one,
$$
then
\begin{equation}
\label{eqn:b-canonical}
\frac{\theta_{\bsi}^0}{\sqrt{-2}}=\sum_{i=1}^{\fg} \hat A^i_{\bsi}(q)
\frac{\partial^2 \Phi}{\partial \tau_{a_i}\partial \tau_{b_i}}
+ \sum_{a=1}^{\fp} \hat B^a_{\bsi}(q)d(\frac{\frac{\partial \Phi}{\partial \tau_a}}{d\hat x})
+ \hat C_\bsi(q) d(\frac{d\hat y}{d\hat x}).
\end{equation}
Therefore
\[
\sum_{\bsi'\in I_\Sigma} \txi^{\bsi'} (z,\tX) S(\hat \phi_\bsi(\btau), \phi_{\bsi'})\big|_{t=\btau, Q=1}
=\sum_{i=1}^{\fg} z^2 \hat A^i_{\bsi}(q) \frac{\partial^2
   U}{\partial\tau_{a_i}\partial \tau_{b_i}}+ \sum_{a=1}^{\fp} z \hat B^a_{\bsi}(q)
   \frac{\partial U}{\partial \tau_a} + \hat C_\bsi(q) U.
\]
By Equation \eqref{eqn:z-disk} and \eqref{eqn:b-canonical}, under the
open-closed mirror map
\begin{equation}
\label{eqn:xi-in-disk}
z^2\sum_{\bsi'\in I_\Si} \txi^{\bsi'}(z,\tX)
S( \hat\phi_{\bsi}(\btau),\phi_{\bsi'})\big|_{t=\btau, Q=1} =-\sum_{\ell\in \bmu_\fm^*}\int_0^{\hX}
\frac{ \rho^*_\ell \hat \theta_{\bsi}(z)}{\sqrt{-2}} \psi_\ell=-\frac{\xi_\bsi(z, \hat{X})}{\sqrt{-2}}.
\end{equation}

\begin{proposition}[Annulus open-closed mirror symmetry]\label{prop:BKMP-annulus}
Under the open-closed mirror map,
$$
\check{F}_{0,2}(q;\hX_1,\hX_2) =-F_{0,2}^{\cX,(\cL,f)}(\btau;\hX_1,\hX_2).
$$
\end{proposition}
\begin{proof} The symmetric meromorphic $2$-form $C(p_1,p_2)$ is defined by \eqref{eqn:C}. Then
\begin{align*}
& (\hX_1\frac{\partial}{\partial \hX_1} + \hX_2\frac{\partial}{\partial \hX_2})\check{F}_{0,2}(q;\hX_1,\hX_2)\\
=& \sum_{\ell_1,\ell_2\in \bmu_\fm^*} \int_0^{\hX_1}\int_0^{\hX_2}(\rho_q^{\ell_1,\ell_2})^*C
\psi_{\ell_1}\otimes\psi_{\ell_2}\\
=&\frac{1}{2} \sum_{\ell_1,\ell_2\in \bmu_\fm^*} \sum_{\bsi \in I_\Si} \int_0^{\hX_1}(\rho_q^{\ell_1})^* \theta_\bsi^0
\int_0^{\hX_2}(\rho_q^{\ell_2})^* \theta_\bsi^0 \psi_{\ell_1}\otimes \psi_{\ell_2}\\
=&\frac{1}{2}\sum_{\bsi\in I_\Si} \xi_\bsi^0(\hX_1)\xi_\bsi^0(\hX_2)\\
=&-[z_1^{-2}z_2^{-2}] \sum_{\bsi,\bsi',\bsi''\in I_\Si} \txi^{\bsi''}(z_1,\tX_1)
   \txi^{\bsi'}(z_2,\tX_2) S(\hat\phi_{\bsi}(\btau),\phi_{\bsi'})\big|_{t=\btau, Q=1}  S(\hat\phi_{\bsi}(\btau),\phi_{\bsi''})\big|_{t=\btau, Q=1}\\
=&-[z_1^{-2}z_2^{-2}](z_1+z_2) \sum_{\bsi',\bsi''}V( \phi_{\bsi'},\phi_{\bsi''})\big|_{t=\btau, Q=1}\txi^{\bsi'}(z_1,\tX_1) \txi^{\bsi''}(z_2,\tX_2)\\
=&-[z_1^{-1}z_2^{-1}] (\hX_1\frac{\partial }{\partial
   \hX_1}+\hX_2\frac{\partial}{\partial \hX_2}) \sum_{\bsi',\bsi''} V( \phi_{\bsi'},\phi_{\bsi''})\big|_{t=\btau, Q=1}\txi^{\bsi'}(z_1,\tX_1) \txi^{\bsi''}(z_2,\tX_2)\\
=& -(\hX_1\frac{\partial}{\partial \hX_1} + \hX_2\frac{\partial}{\partial \hX_2})F_{0,2}^{\cX,(\cL,f)}(\btau;\tX_1,\tX_2).
\end{align*}
where the second equality follows from Lemma \ref{lemm:C}, the fourth
equality follows from Equation \eqref{eqn:xi-in-disk}, the fifth equality
is WDVV (Equation \eqref{eqn:two-in-one}), and the last equality follows from \eqref{eqn:annulus-zero}.
Both $\check{F}_{0,2}(q;\hX_1,\hX_2)$ and $F^{\cX,(\cL,f)}_{0,2}(\btau;\hX_1,\hX_2)$ are $H^*_{\CR}(\cB\bmu_\fm;\bC)^{\otimes 2}$-valued
power series in $\hX_1, \hX_2$ which vanish at $(\hX_1,\hX_2)=(0,0)$, so
$$
\check{F}_{0,2}(q;\hX_1,\hX_2)= - F_{0,2}^{\cX,(\cL,f)}(\btau;\hX_1,\hX_2).
$$
\end{proof}

\begin{theorem}[All genus open-closed mirror symmetry, a.k.a. BKMP Remodeling Conjecture]\label{main}
Under the open and closed mirror maps,
$$
\check{F}_{g,n}(q,\hX_1,\ldots, \hX_n) =(-1)^{g-1+n}  F_{g,n}^{X, (\cL,f)}(\btau;\tX_1,\ldots, \tX_n).
$$
\begin{proof}
For the unstable cases $(g,n)=(0,1)$ and $(0,2)$, this theorem is
Theorem \ref{thm:AKV-disk} and Proposition \ref{prop:BKMP-annulus} respectively.

For stable cases $2g-2+n>0$, the graph sums are matched in
Theorem \ref{thm:graph-match} except for open leafs. We match them here.

The A-model open leaf is
$$
(\cL^O)^\bsi_k =- [z^k]\frac{1}{\sqrt{-2}}\sum_{\ell \in \bmu_\fm^*} \int_0^{\hX}  \sum_{\bsi'\in I_\Sigma} \left(R_{\bsi'}^{\spa \bsi}(-z)\big|_{t=\btau,Q=1} \right)\rho_\ell^* \hat \theta_{\bsi'}(z) \psi_\ell.
$$
By $\theta_\bsi(z)=\sum_{\bsi'\in I_\Sigma} \left(R_{\bsi'}^{\spa \bsi}(-z)\big|_{t=\btau, Q=1}\right) \hat \theta_{\bsi'}(z)$ (Proposition \ref{prop:b-open-leaf}), the B-model open leaf is
$$
(\check{\cL}^O)^\bsi_k=[z^k]\frac{1}{\sqrt{-2}}\sum_{\ell \in \bmu_\fm^*} \int_0^{\hX}  \sum_{\bsi'\in I_\Sigma} \left(R_{\bsi'}^{\spa \bsi}(-z)\big|_{t=\btau, Q=1}\right) \rho_\ell^* \hat \theta_{\bsi'}(z) \psi_\ell.
$$
Then $(\cL^O)^\bsi_k=-(\check{\cL}^O)^\bsi_k$, and this proves the BKMP Remodeling Conjecture.

\end{proof}
\end{theorem}

\subsection{BKMP Remodeling Conjecture: the free energies}
Recall that in Definition \ref{def:Fgn},
$$
F^{\cX}_g(\btau):= \llangle \ \rrangle^{\cX}_{g,0}\big|_{t=\btau, Q=1}.
$$

\subsubsection{The case $g>1$}

\begin{theorem}\label{thm:stable-Fg}
When $g>1$, we have
$$
F^{\cX}_g(\btau)= (-1)^{g-1} \check{F}_g(q).
$$
\end{theorem}
\begin{proof}
From the proof of Theorem \ref{thm:graph-match},
\begin{equation}\label{eqn:stable-Fg-AB}
w^{\bu}_B(\vGa)\Big|_{(\check{\cL}^{\bu})^{\bsi}_k(l_1)=(\check{\cL}^1)^{\bsi}_k}
=  (-1)^{g-1} w_A^{\bu}(\vGa) \Big|_{(\cL^{\bu})^{\bsi}_k(l_1)=(\cL^1)^{\bsi}_k,t=\btau,Q=1}.
\end{equation}
for any labelled graph $\vGa\in \Gamma_{g,1}(\cX)$.
Theorem \ref{thm:stable-Fg} follows from Proposition
\ref{prop:stable-Fg-A}, Proposition \ref{prop:stable-Fg-B}, and Equation \eqref{eqn:stable-Fg-AB}.
\end{proof}

Theorem \ref{thm:stable-Fg} was proved in the special case $\cX=\bC^3$ in \cite{BCMS}.

\subsubsection{The case $g=1$}

The genus-one free energy has a different formula on both A-model and B-model. On A-model side, since the graph sum formula is for $2g-2+n>0$, we need to find a different formula for $F^{\cX}_1$. In \cite{Zo}, the third author proved a formula for the genus-one Gromov-Witten potential of any GKM orbifolds. It expresses $F^{\cX}_1$ in terms of the Frobenius structures. In our case, we have the following theorem:

\begin{theorem}[Givental {\cite{G98}}, Zong {\cite{Zo}}]
The following formula holds for the genus one Gromov-Witten potential
$F^{\cX}_1(\btau)$:
$$dF^{\cX}_1(\btau)=\sum_{\bsi\in I_\Si}\frac{1}{48}d\log\Delta^\bsi(\btau)+
\sum_{\bsi\in I_\Si}\frac{1}{2}(R_1)_\bsi^{\spa \bsi}du^\bsi.$$
\end{theorem}
\noindent
In \cite{G98}, Givental conjectured  that the above formula holds for genus-one GW potential of a compact symplectic manifold with generically
semisimple quantum cohomology and proved this formula for any GKM manifolds.

\medskip

On B-model side, the genus-one free energy is defined in the following way (see \cite{EO07}):

\begin{definition}[genus-one B-model free energy]
The genus-one B-model free energy $\check{F}_1$ is defined as
$$
\check{F}_1=-\frac{1}{2}\log\tau_B-\frac{1}{24}\sum_{\bsi\in I_\Si}\log h_1^\bsi
$$
where $\tau_B$ is the Bergmann $\tau$-function determined by
$$
d (\log\tau_B)=\sum_{\bsi} \Res_{p\to p_\bsi}\frac{B(p,\bar{p})}{d\hat{x}(p)} d\check{u}^{\bsi}.
$$
\end{definition}

The Bergmann $\tau-$function is defined up to a constant and so is $\check{F}_1$. The mirror symmetry for the genus-one free energy is the following theorem:

\begin{theorem}[mirror symmetry for genus-one free energy]\label{F1}
Under the closed mirror map,
$$
dF^{\cX}_1(\btau)=d\check{F}_1(q)
$$
\end{theorem}
\begin{proof}
First by the identification $h_1^\bsi=\sqrt{\frac{-2}{\Delta^{\bsi}(\btau)}}$, we have
$$
-\frac{1}{24}\sum_{\bsi\in I_\Si}d\log h_1^\bsi=-\frac{1}{24}\sum_{\bsi\in I_\Si}d\log\sqrt{\frac{-2}{\Delta^{\bsi}(\btau)}}
=\frac{1}{48}\sum_{\bsi\in I_\Si}d\log \Delta^\bsi(\btau).
$$
So in order to prove the theorem, we only need to show that
$$
-\frac{1}{2}d\log\tau_B=\sum_{\bsi\in I_\Si}\frac{1}{2}(R_1)_\bsi^{\spa \bsi}du^\bsi
\big|_{t=\btau, Q=1}.
$$
Note that since $\{\check{u}^\bsi\}_{\bsi\in I_\Si}$ is the set of B-model canonical coordinates and so $du^\bsi\big|_{t=\btau, Q=1}=d\check{u}^\bsi$ for any $\bsi\in I_\Si$. Therefore
$$
-\frac{1}{2}d\log\tau_B
=-\frac{1}{2}\sum_{\bsi\in I_\Si}\Res_{p\to p_\bsi}\frac{B(p,\bar{p})}{d\hat{x}(p)}
du^\bsi\big|_{t=\btau, Q=1}.
$$
By the local expansions of $\hat{x}$ and $B(p,\bar{p})$ near $p_\bsi$, we have
\begin{eqnarray*}
\hat{x}&=&\check{u}^\bsi+\zeta_\bsi^2\\
B(p,\bar{p})&=&\left(\frac{1}{(2\zeta_\bsi)^2}+\sum_{k,k'\geq 0}B^{\bsi,\bsi}_{k,k'}\zeta_\bsi^k(-\zeta_\bsi)^{k'}\right)d\zeta_\bsi d(-\zeta_\bsi).
\end{eqnarray*}
Recall that
$$
\check{B}^{\bsi,\bsi}_{k,l}=
\frac{(2k-1)!!(2l-1)!!}{2^{k+l+1}}B^{\bsi,\bsi}_{2k,2l}.
$$
Substituting the local expansions of $\hat{x}$ and $B(p,\bar{p})$ into $\Res_{p\to p_\bsi}\frac{B(p,\bar{p})}{d\hat{x}(p)}$, we have
\begin{eqnarray*}
\Res_{p\to p_\bsi}\frac{B(p,\bar{p})}{d\hat{x}(p)}&=&-\check{B}^{\bsi,\bsi}_{0,0}\\
&=& -[z^0 w^0]\left(\frac{1}{z+w}(\delta_{\bsi,\bsi}
-\sum_{\bgamma\in I_\Si} f^{\spa \bsi}_\bgamma(\frac{1}{z})f^{\spa \bsi}_\bgamma(\frac{1}{w}))\right)\\
&=& -[z^0 w^0]\left(\frac{1}{z+w}(\delta_{\bsi,\bsi}
-\sum_{\bgamma\in I_\Si}  R^{\spa \bsi}_\bgamma(-z)R^{\spa \bsi}_\bgamma(-w))\right)\\
&=& -(R_1)_\bsi^{\spa \bsi}\big|_{t=\btau, Q=1}.
\end{eqnarray*}
Therefore
$$
-\frac{1}{2}d\log\tau_B=-\frac{1}{2}\sum_{\bsi\in I_\Si}\Res_{p\to p_\bsi}\frac{B(p,\bar{p})}{d\hat{x}(p)}du^\bsi\big|_{t=\btau, Q=1}=\sum_{\bsi\in I_\Si}\frac{1}{2}(R_1)_\bsi^{\spa \bsi}du^\bsi\big|_{t=\btau, Q=1}
$$
which finishes the proof.
\end{proof}

\subsubsection{The case $g=0$}
Another special case is the genus-zero free energy. In this case, instead of giving the definition of $\check F_0$ directly, we will use the special geometry property to build the mirror symmetry. Recall that we have the following \textbf{special geometry} property (see \cite{EO07}): for $i=1,\ldots, \fp$,
\begin{equation}\label{sg1}
\frac{\partial \omega_{g,n}}{\partial \tau_i}(p_1,\ldots,p_n)=\int_{p_{n+1}\in B_i}\omega_{g,n+1}(p_1,\cdots,p_{n+1}),\quad  (g,n)\neq (0,0), (0,1).
\end{equation}
\begin{equation}\label{sg2}
\frac{\partial \Phi}{\partial \tau_i}(p_1)=\int_{p_{2}\in B_i}\omega_{0,2}(p_1,p_2),
\end{equation}
\begin{equation}\label{sg3}
\frac{\partial \check{F}_0}{\partial \tau_i}=\int_{p\in B_i}\Phi(p)
\end{equation}
Here when $n=0$, the invariant $\omega_{g,0}$ is just the free energy $\check{F}_g$. We will use the special geometry property to show the following theorem:
\begin{theorem}[mirror symmetry for genus-zero free energy]\label{F0}
For any $i,j,k\in\{1,\cdots,\fp\}$, we have
$$
\frac{\partial^3 F^{\cX}_0}{\partial\tau_i\partial\tau_j\partial\tau_k}(\btau)=-\frac{\partial^3 \check{F}_0}{\partial\tau_i\partial\tau_j\partial\tau_k}(q)
$$
under the closed mirror map.

\end{theorem}
\begin{proof}
Recall that $\{H_1,\cdots,H_{\fp}\}$ is the basis of $H^2_{\CR,\bT'}(\cX)$ corresponding to the coordinates $\{\tau_1,\cdots,\tau_{\fp}\}$. Then
\begin{eqnarray*}
\frac{\partial^3 F^{\cX}_0}{\partial\tau_i\partial\tau_j\partial\tau_k}=\llangle H_i,H_j,H_k\rrangle^{\cX,\bT'}_{0,3}\big|_{t=\btau,Q=1}.
\end{eqnarray*}
By the graph sum formula described in Section \ref{sec:Agraph}, we know that $\llangle H_i,H_j,H_k\rrangle^{\cX,\bT'}_{0,3}$ has the same graph sum formula with that of $\llangle \bu_1,\bu_2,\bu_3\rrangle^{\cX,\bT'}_{0,3}$ except that the ordinary leaves are replaced by
\begin{equation}\label{oA}
[z^0](\sum_{\brho\in I_\Si}\Psi_l^{\,\ \brho}R(-z)_{\brho}^{\,\ \bsi})
\end{equation}
with $l=i,j,k$ respectively. Here $\Psi_l^{\,\ \brho}$ is defined as $H_l=\sum_{\brho\in I_\Si}\Psi_l^{\,\ \brho}
\hat\phi_\brho(\btau)$.

Now let us consider $\frac{\partial^3 \check F_0}{\partial\tau_i\partial\tau_j\partial\tau_k}$. By the special geometry property
\eqref{sg1} \eqref{sg2} \eqref{sg3}, we have
\begin{eqnarray*}
 \frac{\partial^3 \check F_0}{\partial\tau_i\partial\tau_j\partial\tau_k}
&=&\frac{\partial^2}{\partial \tau_i \partial \tau_j} \int_{p_1\in B_k}\Phi(p_1)
=\frac{\partial}{\partial \tau_i} \int_{p_1\in B_k}\frac{\partial \Phi}{\partial \tau_j}(p_1)\\
&=&\frac{\partial}{\partial \tau_i} \int_{p_1\in B_k}\int_{p_2\in B_j}\omega_{0,2}(p_1,p_2)
=\int_{p_1\in B_k}\int_{p_2\in B_j} \frac{\partial \omega_{0,2}}{\partial \tau_i}(p_1,p_2)\\
&=&\int_{p_1\in B_k}\int_{p_2\in B_j}\int_{p_3\in B_i}\omega_{0,3}(p_1,p_2,p_3).
\end{eqnarray*}
By the graph sum formula for $\omega_{0,3}$, we know that $\int_{p_1\in B_k}\int_{p_2\in B_j}\int_{p_3\in B_i}\omega_{0,3}(p_1,p_2,p_3)$
has the same graph sum formula with that of $\omega_{0,3}$ except that the ordinary leaves are replaced by
\begin{equation}\label{oB}
\frac{1}{\sqrt{-2}}\int_{p\in B_l}\theta^0_\bsi(p)
\end{equation}
with $l=k,j,i$ respectively. It is easy to see that
$$
\theta^0_\bsi(p)=[z^0](\frac{-e^{\frac{\check{u}^\bsi}{z}}}{\sqrt{\pi z}}\int_{p'\in\Gamma_\bsi} B(p,p')e^{-\frac{\hat{x}(p')}{z}}).
$$
Define
\[
\check S_l^{\spa \tbsi}(z)=(2\pi z)^{-\frac{3}{2}}\int_{\tGamma_\bsi}e^{-\frac{W^{\bT'}_q}{z}}\frac{\partial W^{\bT'}_q}{\partial \tau_l}
\Omega=(2\pi)^{-\frac{1}{2}}z^{-\frac{3}{2}}\sqrt{-1}\int_{\Gamma_\bsi}e^{-\frac{\hx}{z}}\frac{\partial \Phi}{\partial \tau_l}.
\]
Noting that $[\frac{\partial W^{\bT'}_q}{\partial \tau_l}]\in \Jac(W_q^{\bT'})$ corresponds to $H_l$ under Equation \eqref{eqn:qcoh-Tprime}, by argument similar to Proposition \ref{prop:check-S}, we have
\[
\check S_l^{\spa \tbsi}(z)=\sum_{\brho\in I_\Si}\Psi_l^{\spa \brho} \check R(z)_\brho^{\spa \bsi}e^{-\frac{\check u^\bsi}{z}}.
\]
Therefore
\begin{eqnarray*}
\frac{1}{\sqrt{-2}}\int_{p\in B_l}\theta^0_\bsi(p)&=&\frac{1}{\sqrt{-2}}\int_{p\in B_l}[z^0](\frac{-e^{\frac{\check{u}^\bsi}{z}}}{\sqrt{\pi z}}\int_{p'\in\Gamma_\bsi}
B(p,p')e^{-\frac{\hat{x}(p')}{z}})\\
&=&[z^0](\frac{-1}{\sqrt{-2}}\frac{e^{\frac{\check{u}^\bsi}{z}}}{\sqrt{\pi z}}\int_{p'\in\Gamma_\bsi}\frac{\partial \Phi}{\partial \tau_l}e^{-\frac{\hat{x}(p')}{z}})\\
&=&[z^0](-e^{\frac{\check{u}^\bsi}{z}}\check{S}_l^{\,\ \tbsi}(z))\\
&=&[z^0](-\sum_{\brho\in I_\Si}\Psi_l^{\,\ \brho}\check{R}(z)_{\brho}^{\,\ \bsi})\\
&=&[z^0](-\sum_{\brho\in I_\Si}\Psi_l^{\,\ \brho}R(-z)_{\brho}^{\,\ \bsi})\big|_{t=\btau,Q=1}.
\end{eqnarray*}
Comparing with (\ref{oA}), we see that the three new ordinary leaves on A-model differ those on B-model by a minus sign. So by Theorem \ref{thm:graph-match} for $(g,n)=(0,3)$, we conclude that
$$
\frac{\partial^3 F_0^{\cX}}{\partial\tau_i\partial\tau_j\partial\tau_k}=-\frac{\partial^3 \check F_0}{\partial\tau_i\partial\tau_j\partial\tau_k}.
$$

\end{proof}

\begin{remark}
The proof of Theorem \ref{F0} can be directly generalized to show that the first derivatives of
$F_g^{\cX}$ match the first derivatives of $\check F_g$ for any $g\geq 1$
by replacing $[z^0]$ by $[z^k]$ for any $k\in \bZ_{\geq 0}$ in the computation of new ordinary leaves. In particular, this gives another proof of Theorem \ref{F1}.

\end{remark}



\end{document}